\documentclass[11pt,reqno]{amsart}
\usepackage[colorlinks=true,allcolors=blue,backref=page]{hyperref}
\usepackage{color}
\usepackage{amsmath, amssymb, amsthm}

\usepackage{mathrsfs}
\usepackage{mathtools}
\usepackage[noabbrev,capitalize,nameinlink]{cleveref}
\crefname{equation}{}{}
\usepackage{fullpage}
\usepackage[noadjust]{cite}
\usepackage{graphics}
\usepackage{pifont}
\usepackage{tikz}
\usepackage{bbm}
\usepackage[T1]{fontenc}

\usetikzlibrary{arrows.meta}

\usepackage{environ}
\usepackage{framed}
\usepackage{url}
\usepackage[linesnumbered,ruled,vlined]{algorithm2e}
\usepackage[noend]{algpseudocode}
\usepackage[labelfont=bf]{caption}
\usepackage{cite}
\usepackage{framed}
\usepackage[framemethod=tikz]{mdframed}
\usepackage{appendix}
\usepackage{graphicx}
\usepackage[textsize=tiny]{todonotes}
\usepackage{tcolorbox}
\usepackage{enumerate}
\usepackage[shortlabels]{enumitem}
\allowdisplaybreaks[1]
\apptocmd{\sloppy}{\hbadness 10000\relax}{}{} 

\crefname{algocf}{Algorithm}{Algorithms}

\crefname{equation}{}{} 
\crefname{conjecture}{Conjecture}{Conjectures} 
\AtBeginEnvironment{appendices}{\crefalias{section}{appendix}} 

\crefformat{enumi}{#2#1#3}
\crefrangeformat{enumi}{#3#1#4 to~#5#2#6}
\crefmultiformat{enumi}{#2#1#3}%
{ and~#2#1#3}{, #2#1#3}{ and~#2#1#3}

\usepackage[color,final]{showkeys} 

\colorlet{refkey}{orange!20}
\colorlet{labelkey}{blue!30}

\crefname{algocf}{Algorithm}{Algorithms}

\numberwithin{equation}{section}
\newtheorem{theorem}{Theorem}[section]
\newtheorem{corollary}[theorem]{Corollary}
\newtheorem{proposition}[theorem]{Proposition}
\newtheorem{lemma}[theorem]{Lemma}

\crefname{claim}{Claim}{Claims}
\newtheorem{conjecture}[theorem]{Conjecture}
\newtheorem*{question*}{Question}

\theoremstyle{definition}
\newtheorem{definition}[theorem]{Definition}

\newtheorem*{definition*}{Definition}

\theoremstyle{remark}
\newtheorem*{remark}{Remark}
\newtheorem*{remarks}{Remarks}

\newcommand{\snorm}[1]{\lVert#1\rVert}

\newcommand{\mb}{\mathbb}

\newcommand{\mbm}{\mathbbm}
\newcommand{\mc}{\mathcal}
\newcommand{\mf}{\mathfrak}

\newcommand{\ol}{\overline}
\newcommand{\on}{\operatorname}

\newcommand{\eps}{\varepsilon}
\newcommand{\imod}[1]{~\mathrm{mod}~#1}
\renewcommand{\le}{\leqslant}
\renewcommand{\ge}{\geqslant}
\renewcommand{\Re}{\on{Re}}
\renewcommand{\Im}{\on{Im}}

\newcommand\Z{\mathbf{Z}}
\newcommand\Q{\mathbf{Q}}
\newcommand\C{\mathbf{C}}
\newcommand\R{\mathbf{R}}
\newcommand\E{\mb{E}}
\renewcommand\P{\mb{P}}

\newcommand{\legendre}[2]{(#1 \mid #2)}

\newcommand{\md}[1]{\ensuremath{(\operatorname{mod}\, #1)}}
\newcommand{\mdsub}[1]{\ensuremath{(\mbox{\scriptsize mod}\, #1)}}

\newcommand{\mdsublem}[1]{\ensuremath{(\mbox{\scriptsize \textup{mod}}\, #1)}}

\newcommand\siegel{{\textnormal{Siegel}}}
\newcommand\cramer{{\textnormal{Cram\'er}}}

\newcommand\main{{\textnormal{main}}}
\newcommand\qq{Q} 

\newcommand\Unif{\operatorname{U}} 
\DeclareMathOperator*{\bigavg}{\scalerel*{\mbm{E}}{(}}
\usepackage{scalerel}

\newcommand\classgroup{\operatorname{Cl}(K)}

\renewcommand\main{\operatorname{main}}

\newcommand\Ideals{\operatorname{Ideals}}

\newcommand\GP{\operatorname{GP}}

\newcommand{\conv}{\mathop{\scalebox{1.5}{\raisebox{-0.2ex}{$\ast$}}}}

\renewcommand\O{\mathcal{O}}

\allowdisplaybreaks

\title{Primes of the form $p^2 + nq^2$}

\author[A1]{Ben Green}
\address{Mathematical Institute, Andrew Wiles Building, Radcliffe Observatory Quarter, Woodstock Rd, Oxford OX2 6QW, UK}
\email{ben.green@maths.ox.ac.uk}

\author[A2]{Mehtaab Sawhney}
\address{Department of Mathematics, Columbia University, New York, NY 10027}
\email{m.sawhney@columbia.edu}

\begin{document}

\begin{abstract}
Suppose that $n \equiv 0$ or $n \equiv 4 \imod{6}$. We show that there are infinitely many primes of the form $p^2 + nq^2$ with both $p$ and $q$ prime, and obtain an asymptotic for their number. In particular, when $n = 4$ we verify the `Gaussian primes conjecture' of Friedlander and Iwaniec.

We study the problem using the method of Type I/II sums in the number field $\Q(\sqrt{-n})$. The main innovation is in the treatment of the Type II sums, where we make heavy use of two recent developments in the theory of Gowers norms in additive combinatorics: quantitative versions of so-called concatenation theorems, due to Kuca and to Kuca--Kravitz-Leng, and the quasipolynomial inverse theorem of Leng, Sah and the second author.
\end{abstract}

\maketitle

\setcounter{tocdepth}{1}
\tableofcontents

\section{Introduction}\label{sec:introduction}

 Our main result establishes an asymptotic count for pairs of primes $x$ and $y$ such that $x^2 + ny^2$ is prime. 

\begin{theorem}\label{thm:main}
Suppose that $n$ is a positive integer with $n \equiv 0$ or $n\equiv 4\imod 6$. Let $W \in C^{\infty}_0(\R^2)$. Let $N > 1$ be a parameter. Then
\[ \sum_{x,y \in \Z} \Lambda(x) \Lambda(y) \Lambda(x^2 + ny^2) W\bigg(\frac{x}{N}, \frac{y}{N}\bigg) = \kappa_n N^2 \bigg(\int_{\R^2} W\bigg) + O_{W,n}\bigg( \frac{N^2(\log \log N)^{2}}{\log N}\bigg),\]
where
\begin{equation}\label{kappa-def}
\kappa_n :=   \lim_{X \rightarrow \infty} \prod_{\substack{p \le X \\ \legendre{-n}{p} = 1}}\frac{p(p-3)}{(p-1)^2}\prod_{\substack{p \le X \\ \legendre{-n}{p} \neq 1}}\frac{p}{p-1},\end{equation}
with $\legendre{\cdot}{p}$ the Legendre symbol. In particular, if $n \equiv 0$ or $n \equiv 4 \md{6}$ then there are infinitely many primes of the form $x^2 + ny^2$ with $x, y$ prime.
\end{theorem}

\begin{remarks} 
Here, and throughout the paper, the von Mangoldt function is defined by $\Lambda(n) = \log p$ if $n = p^k$ is a prime power, $\Lambda(n) = 0$ if $n$ is a non-negative integer other than a prime power, and $\Lambda(-n) = \Lambda(n)$ if $n < 0$. 

From the statement of \cref{thm:main} one can easily derive asymptotics for the number of $x,y$ in a box such as $[0,N] \times [0,N]$ for which $x,y$ and $x^2 + ny^2$ are all prime, by approximating $1_{[0,1]^2}$ by a smooth function $W$. We leave the details to the reader. It can be seen from the arguments of \cref{polar-sec} that, under the assumption that $W$ is supported in $B_{10}(0)$, the implied constant in the $O_W( \cdot )$ in \cref{thm:main} can be taken to be $\ll \sup_{0 \le j \le 3}\Vert \partial^{(j)}W  \Vert_{\infty}$, by which we mean the maximum over all derivatives of $W$ of order at most $3$. 
\end{remarks}

If one wanted an asymptotic over the range $x^2 + ny^2 \le N^2$ then one could use \cref{prop:main} directly, without having to pass to a smooth weight $W$. This case, when further specialized to $n=4$, recovers the `Gaussian primes conjecture' in the form stated by Friedlander and Iwaniec \cite[Conjecture 1.1]{FI22}.  This (up to a trivial change in notation) states the following.

\begin{conjecture}[{Gaussian Primes Conjecture}]\label{gpc}
There holds
\[ \sum_{\substack{x,y \\ x^2 + 4y^2 \le N}} \Lambda(x) \Lambda(y) \Lambda(x^2 + 4 y^2) \sim cN,\] where $c = \prod_{p \equiv 1 \mdsublem{4}} (1 - \frac{3}{p})(1 - \frac{1}{p})^{-3} \prod_{p \equiv 3 \mdsublem{4}} (1 - \frac{1}{p^2})^{-1}$.
\end{conjecture}

To see that this does indeed follow from \cref{thm:main}, one can take $W = W_{\eps}^{\pm}$ where $W_{\eps}^+$ is a smooth majorant for $1_{|x^2 + 4y^2| \le 1}$, differing from it by at most $\eps$ in the $L^1(\R^2)$-norm, and $W_{\eps}^-$ is a minorant with the same property. As $\eps \rightarrow 0$, this tends to $4$ times the quantity of interest in \cref{gpc} (noting here that positive and negative values of $x,y$ are included). Now note that $\lim_{\eps \rightarrow 0} \int_{\R^2}  W^{\pm}_{\eps} = \pi/2$. Moreover, the class number formula gives $\frac{\pi}{4} =  L(1,\chi_{-4}) = \lim_{X \rightarrow \infty}\prod_{p \le X : p \equiv 1 \mdsub{4}} \big(1 - \frac{1}{p}\big)^{-1} \prod_{p \le X: p \equiv 3 \mdsub 4}\big(1 + \frac{1}{p}\big)^{-1}$ which, after some calculation, leads to $c = \pi \kappa_4/8$, and so \cref{gpc} follows from \cref{thm:main} as claimed.

In general the infinite product \cref{kappa-def} is only conditionally convergent. To show that it does converge, by taking logs and Taylor expansion it suffices to show that $\lim_{X \rightarrow \infty} \sum_{p \le X} \legendre{-n}{p} p^{-1}$ exists, which is typically shown in a standard proof of Dirichlet's theorem on primes in progressions, noting here that $\legendre{-n}{p} = \chi_n(p)$ for some Dirichlet character $\chi_n$ to modulus dividing $4n$.

It has the form predicted by heuristics of Bateman--Horn type, that is to say a product of natural archimedean and $p$-local factors. One can state an equivalent result with an absolutely convergent product by dividing through by the class number formula for the field $\Q(\sqrt{-n})$, as Friedlander and Iwaniec did for $n = 4$ in their formulation of \cref{gpc}.

Finally, the argument could surely be modified in a straightforward manner to handle arbitrary positive definite binary quadratic forms over $\Z$ in place of $x^2 + ny^2$, but this would introduce even more notation to a paper which is already fairly heavy going. It should also not be difficult to modify the argument to prove, for example, that there are infinitely many primes $x^2 + y^2$ with $x$ and $y - 1$ prime.

\subsection{Prior results}

Famously, it was stated by Fermat and proven by Euler that every prime $p \equiv 1 \md{4}$ can be written as $x^2 + y^2$, and in particular there are infinitely many primes of the form $x^2 + y^2$. Fermat also asked about prime values of $x^2 + ny^2$ for various $n \ge 2$, and made analogous conjectures when $n = 2,3,5$. The first two of these were also proven by Euler, and the third by Lagrange. The book \cite{cox-book} (which inspired the title and some of the content of this paper) is devoted to primes of the form $x^2 + ny^2$ from an algebraic perspective, linking this to the development of central topics in number theory such as class field theory. We note in particular that it was shown by Weber in 1882 \cite{weber} that there are infinitely many primes of the form $x^2 + ny^2$ for any $n \ge 1$.

Our main result falls into a long line of work considering refinements of such statements in which some restriction is placed on the coordinates $x,y$. Most of these have been concerned with the case $n = 1$, that is to say with the form $x^2 + y^2$. Work of Fouvry and Iwaniec \cite{FI97} handled the case where $x$ is prime. Celebrated work of Friedlander and Iwaniec \cite{FriIw1,FriIw2}, showed that one may restrict $x$ to be a perfect square (the first polynomial taking values in polynomially thin set shown to capture primes). This was refined in work of Heath--Brown and Li \cite{HL17}, to force $x$ to be a square of a prime. Pratt \cite{Pra20} has proven that $x$ may be forced to lie in a set where one excludes $3$ digits from its decimal expansion. We also mention two recent works of Merikoski \cite{mer1, mer2}; in particular, in \cite{mer2} the variable $x$ is restricted to an arbitrary set $S$ of integers satisfying $|S \cap \{1,\dots N\}| \gg N^{1 - \delta}$ for some $\delta > 0$. The work of Fouvry and Iwaniec was generalised to $x^2 + ny^2$ for arbitrary $n \ge 1$ in \cite{he-thesis}.

Short of proving that there are infinitely many primes of the form $x^2 + 4$ (with $x$ prime), for parity reasons one cannot hope to show that there are infinitely many primes of the form $x^2 + y^2$ with \emph{both} $x$ and $y$ prime, and for related reasons $\md 6$ one should restrict the search for primes of the form $x^2 + ny^2$ with $x, y$ prime to the case $n \in \{0,4\} \md{6}$.

Friedlander and Iwaniec \cite{FI22} considered the problem in the case $n = 4$ and proved a lower bound for the number of $(x,y)$ with $x$ prime, $y$ having at most $7$ prime factors, and $x^2 + 4y^2$ prime. 

\cref{thm:main} uses relatively little about the specific form of the functions $\Lambda(x),\Lambda(y)$, and similar techniques may be used for other problems connected with the coordinates of Gaussian primes. Perhaps the most interesting of these, which we will address in \cite{GS24}, is to give asymptotic counts for corners $(x + iy, (x+d) + iy, x + i(y + d))$ in the Gaussian primes. 

\subsection{Outline of the argument}\label{subsec:outline}

For most of the rest of the introductory discussion we specialise to the case $n = 4$. In this case, \cref{thm:main} is closely related to a statement involving Gaussian primes (that is to say primes in $\Z[i]$) using the fact that $z = x + iy$, $y \neq 0$, is a Gaussian prime if and only $N_{\Q(i)/\Q}(z) = x^2 + y^2$ is a rational prime. In particular, our main theorem implies that there are infinitely many Gaussian primes $x + 2iy$ with $x, y$ prime.

For the sake of this introductory discussion we will consider the problem of estimating
\begin{equation}\label{general-sum-1} \sum_{x^2 + 4y^2 \le X} f(x) f'(y) 1_{x^2 + 4y^2 \operatorname{prime}}.\end{equation} In the case $f= f'= \Lambda$, this is a special case of \cref{thm:main} with $N = X^{1/2}$, $W$ (a smooth approximation to) the cutoff $1_{u^2 + nv^2 \le 1}$, and with the von Mangoldt weight $\Lambda(x^2 + 4y^2)$ (which is essentially always $\log X$) renormalised to $1$. 

Much of our argument is relevant for fairly general $f,f'$. Noting that \cref{general-sum-1} is bilinear in $f$ and $f'$, one can hope in any given situation of interest to write $f,f'$ as sums of relatively easily-understood `main terms' plus error terms for which one hopes that the sum \cref{general-sum-1} is small. We are then left with the task of finding verifiable criteria on $f,f'$ which allow us to assert that \eqref{general-sum-1} is small.

The first key idea is to observe that \cref{general-sum-1} is a sum
\begin{equation}\label{general-sum-2} \sum_{\gamma \in \operatorname{Primes}(\Z[i])} w(\gamma),\end{equation} where $w : \Z[i] \rightarrow \C$  is defined by $w(x + 2iy) = f(x) f'(y)$ and $w(x + iy) = 0$ for $y$ odd, and we assume throughout the following discussion that $w(\gamma)$ is supported where $|\gamma|^2 \le X$.

We then examine this sum over Gaussian primes using so-called Type I and Type II sums.  For a general introduction to this technique over the rational integers, see for instance \cite[Chapter 13]{IK-book} (where the authors do not use the terms Type I/II sum, but the relevant sums are (13.16) and (13.17)) or the comprehensive recent analysis \cite{FM24}. Adapting the relevant techniques from the rational integers $\Z$ to the Gaussian integers $\Z[i]$ presents only minor technical obstacles in our setting. See \cite[Section 3]{HB01} for a very similar type of argument in the ring $\Z[2^{1/3}]$, and \cite{hb-moroz,maynard-norm-forms} for similar arguments in fields not necessarily enjoying unique factorisation. 

For the purposes of this discussion, Type I sums are roughly of the form
\begin{equation}\label{typei-def} \sum_{\substack{\ell \in \Z[i] \\ |\ell|^2 \sim L}} \Big| \sum_{\substack{m \in \Z[i] \\ \ell | m}} w(m) \Big|\end{equation} for some positive $L \in \R$, which we informally call the \emph{level}.

Type II sums are bilinear sums of the form 
\begin{equation}\label{typeii-def}\sum_{\substack{\ell, m\in \Z[i]\\ |\ell|^2\sim L}} \alpha_{\ell} \beta_m w(\ell m),\end{equation} where here the coeffients $\alpha_{\ell}, \beta_m$ will satisfy reasonable boundedness conditions but can otherwise be fairly arbitrary.

The estimates we seek for these sums are bounds of the form $O(X (\log X)^{-B})$ for some large $B$, which represent savings over the trivial bounds by large powers of $\log X$. If we have such bounds, we informally say that we have Type I (or Type II) information at level $L$. (Precise statements may be found in \cref{section3}.)

If one has Type I and Type II information at sufficient levels, it is possible to obtain asymptotics for sums over primes \cref{general-sum-2} by sieve methods, details of which may be found in \cref{section3}. Any such argument needs Type I/II information in sufficiently large ranges. The question of exactly what is necessary or sufficient for this purpose is rather complicated (and, over the integers, has been comprehensively addressed in the recent work \cite{FM24}). 

We will use an argument of Duke, Friedlander, and Iwaniec \cite[Section~6]{DFI95} which (adapted to the Gaussian primes) shows that, at least for $1$-bounded functions $f,f'$, it is enough to have Type I information up to level $X^{1/2 - o(1)}$, and Type II information up to $X^{1/3 - o(1)}$.

The task, then, is to obtain the relevant Type I and II information. We are able to do this under the assumption that appropriate \emph{Gowers norms} of $f$ or $f'$ are small. The Gowers norms are central objects in modern additive combinatorics; see \cref{gowers-norms-intro} for an introduction. Assuming bounds for suitable Gowers norms of $f$ or $f'$, we are able to obtain Type I information up to level $X^{1/2 - o(1)}$ and Type II information for levels between $X^{o(1)}$ and $X^{1/2 - o(1)}$, which is sufficient for variants of the sieve arguments of Duke, Friedlander and Iwaniec to apply in order to show that \cref{general-sum-2} is small. One point to note here is that, while we have comfortably enough Type II information, the Type I range is just barely enough for the requirements of the Duke--Friedlander--Iwaniec sieve. Worse, the functions $f$ and $f'$ of interest to us for our main theorem are not $1$-bounded but come from the von Mangoldt function (with a suitable main term subtracted), and the most na\"{\i}ve application of the Duke--Friedlander--Iwaniec machinery fails to give a usable bound. To get around this issue we require an additional upper bound sieve (of dimension 3). We remark that obtaining Type I information at level $X^{1/2}$ in our setting remains an interesting open question. In fact, even obtaining bounds at this level without the absolute values in \cref{typei-def} seems challenging. (A bound of this latter type would be interesting, being a simple case of what is called Type I${}_2$ information. Sufficient information in this direction would allow one to obtain an arbitrary logarithmic savings in \cref{thm:main}.)

The most novel part of our work lies in showing that the stated Type I and Type II information can indeed be obtained from suitable Gowers norm control on the functions $f$ and $f'$. Precise statements of these results are \cref{prop:typei-to-gowers} and \cref{prop:typeii-to-gowers} respectively. The proof of \cref{prop:typeii-to-gowers} implicitly uses a large number of applications of the Cauchy--Schwarz inequality, organised via what are known as \emph{concatenation} theorems for Gowers-type norms. The first concatenation results were obtained by Tao and Ziegler \cite{tao-ziegler}, but those results are rather qualitative. More recently, starting with the work of Peluse and Prendiville \cite{PP19} and Peluse \cite{Pel20}, more quantitative variants have been developed. For our application, we need the very recent work of Kuca \cite{kuca} and Kravitz--Kuca--Leng \cite{KKL24}.

In order to apply these results to prove Theorem \ref{thm:main}, we need to show that for suitable functions $f$ and $f'$ related to the von Mangoldt function, the relevant Gowers norms are small. We work with $f = f' = \Lambda - \Lambda_{\cramer}$, where $\Lambda_{\cramer}$ is a certain low-complexity approximant to the von Mangoldt function defined in \cref{cramer-def} below. Statements of this type were first shown in a series of works of the first author with Tao \cite{green-tao-linear,GT12,GT12-orthog} and with Tao and Ziegler \cite{GTZ12}. More recent work of Tao and Ter\"av\"ainen \cite{TT21} gave quantitative bounds relying on effective bounds for the inverse theorem for the Gowers norms established by Manners \cite{Man18}. For our application, however, we need savings of large powers of $\log X$ over the trivial bound. Bounds of this strength were only recently established by Leng \cite{Len23b}, whose work relies on the quasipolynomial inverse theorem for the Gowers norms due to Leng, Sah and the second author \cite{LSS}. 

The final task of the paper is then to obtain the main term in \cref{thm:main} by evaluating \cref{general-sum-1} with $f = f' = \Lambda_{\cramer}$. This is a task for classical methods of multiplicative number theory, the main input being the prime ideal theorem. However, it takes a little work to get all the details in order.

In the above outline we had $n = 4$. To deal with more general values of $n$, we use the field $K = \Q(\sqrt{-n})$ in place of the Gaussian field. In general, these fields do not enjoy unique factorisation and in the above discussion one must work with ideals rather than with elements of the ring of integers $\O_K$. In fact, for technical reasons related to units it is best to work in this language even in the Gaussian setting (which does have unique factorisation). The fact that not all ideals are principal creates some extra difficulties, but they turn out not to be too serious and are resolved in much the same way that similar difficulties were overcome in \cite{hb-moroz,maynard-norm-forms} (for instance).

\subsection{Organization}
In \cref{section2} we review some of the basic arithmetic of $K= \Q(\sqrt{-n})$ and recall some simple lemmas about the ring of integers $\O_K$ and the ideals in this ring.

In \cref{section3}, we develop the relevant sieve machinery allowing us to estimate sums over primes such as \cref{general-sum-2} using Type I/II information. Here we follow the arguments of Duke, Friedlander and Iwaniec rather closely, albeit in the number field setting. As explained above, we will additionally require the input of an upper bound sieve of dimension 3 in two variables. 

In \cref{section4} we begin by giving a technical reduction of our main theorem, which is the form we shall actually prove; essentially, this consists in converting from Cartesian to polar coordinates. Then, we recall the definitions and basic properties of the Gowers $U^{k}$-norms, which will feature heavily in the paper from this point on. We will also give the outline proof of our main theorem, leaving the proof of the key Type I and II estimates, as well as the main input from analytic number theory, to later sections.

In \cref{section5}, we recall a variant of the Gowers norms due to Peluse \cite{Pel20}, which we call the Gowers--Peluse norms. These are the appropriate generalisations of the Gowers norms needed for the concatenation machinery we will require. In this section we also state the main concatenation estimate of Kuca, Kravitz, and Leng \cite{KKL24}. 

In \cref{section6}, we show that Type I information is controlled by Gowers norms and in \cref{section7} we do the same for Type II information. 

In \cref{section8}, we complete the proof of \cref{thm:main} by evaluating the main term in the asymptotic. 

Finally, we provide four appendices. In \cref{appendixA}, we prove various properties of Gowers box norms. In \cref{appendixB}, we provide self-contained proofs of the main concatenation results. In \cref{appendixC}, we briefly provide details of the large sieve in higher dimensions. In \cref{appendixD} we collect some standard arithmetic estimates.

\subsection{Notation}\label{sec:notation} There will be a small number of fixed objects throughout the paper, as follows:
\begin{itemize}
\item $n$ will always be the same $n$ as in the main theorem, and we think of it as fixed;
\item $K$ will always be $\Q(\sqrt{-n})$, regarded as embedded into $\C$, except in the self-contained \cref{DFI-sec} where it could be any number field (but for our application will be $\Q(\sqrt{-n})$);
\item $X$ is some positive integer parameter, which we will always assume sufficiently large in terms of $n$ without further comment.
\end{itemize}

For real $N \ge 1$ we write $[N] = \{ x \in \Z : 1 \le x \le N\}$ and $[\pm N] = \{ x \in \Z : -N \le x \le N\}$.

It is convenient to use a very minor modification of the von Mangoldt function which ignores the prime powers and is symmetric. Thus define $\Lambda'(n) = \log |n|$ if $|n|$ is a prime, and $\Lambda'(n) = 0$ otherwise.

As mentioned above, we will use a `low-complexity' model for the primes, whose main task is to encode the fact that the primes are essentially supported only on residue classes coprime to $q$. Specifically, for $x\in \Z$ and $Q \in \Z^{+}$, we define the Cram\'er model at level $\qq$ to be  
\begin{equation}\label{cramer-def} \Lambda_{\cramer, \qq}(x) = \prod_{p\le \qq}\bigg(1 - \frac{1}{p}\bigg)^{-1} 1_{(x,p) = 1}.\end{equation}
Throughout the paper we take 
\begin{equation}\label{q-choice} \qq := \exp(\log^{1/10} (X^{1/2})).\end{equation} Note that this is the same choice made in \cite[Equation (1.1)]{TT21} and in \cite[Section 5]{Len23b} if one takes $N = X^{1/2}$ in those papers. This is the reason for the inclusion of $X^{1/2}$ in \cref{q-choice}, which otherwise looks peculiar.
We will write $\Lambda_{\cramer} = \Lambda_{\cramer,\qq}$ from now on.

As is standard, we let $e(\theta) = e^{2\pi i\theta}$. 

We use standard asymptotic notation throughout. We use $A\sim B$ to denote that $B/2\le A<2B$ (i.e. $A$ lies within a factor of $2$ of $B$). Given functions $f=f(x)$ and $g=g(x)$, we write $f=O(g)$, $f\ll g$, $g=\Omega(f)$, or $g\gg f$ to mean that there is a constant $C$ such that $|f(x)|\le Cg(x)$ for sufficiently large $x$. We write $f\asymp g$ or $f=\Theta(g)$ to mean that $f\ll g$ and $g\ll f$, and write $f=o(g)$ to mean $f(x)/g(x)\to0$ as $x\to\infty$. Subscripts indicate dependence on parameters.

From \cref{section5} onwards, we will use the language of probability measures on $\Z$. For us, this simply means a finitely-supported function $\mu : \Z \rightarrow [0,1]$ with $\sum_x \mu(x) = 1$. In this paper all our probability measures will be symmetric, meaning that $\mu(x) = \mu(-x)$. The convolution $\mu_1 \ast \mu_2$ of two probability measures is defined by $(\mu_1 \ast \mu_2)(x) := \sum_y \mu_1(y) \mu_2(x - y)$, and it is also a probability measure. If $\mu_1,\dots, \mu_{\ell}$ are probability measures then we write $\conv_{i = 1}^{\ell} \mu_i$ for $\mu_1 \ast \cdots \ast \mu_\ell$. If $\mu_1 = \cdots = \mu_{\ell} = \mu$, we write $\mu^{(\ell)}$ for this measure. Given a (multi)set $S\subseteq \Z$, we let $\Unif_S$ denote the uniform measure on this set. Finally, $\delta_0$ denotes the delta-mass at $0$, thus $\delta_0(0) = 1$ and $\delta_0(x) = 0$ for $x \neq 0$. Observe that $\delta_0$ is an identity for convolution in the sense that $\mu \ast \delta_0 = \mu$ for any probability measure $\mu$. The notation $\E_{x \sim \mu} f(x)$ means the same as $\sum_x f(x) \mu(x)$.

We remark that it is perhaps controversial to use the letter $\mu$ for something other than the M\"obius function in an analytic number theory paper. However, the M\"obius function only appears in \cref{section3} of our paper, where there are no probability measures, and so there should be no danger of confusion.

When working in number fields, we will use the term \emph{rational prime} to mean a prime number in $\Z$. This is to distinguish from other uses of the word prime (prime ideals or prime elements of the ring of integers).

Finally, if $\ell \in \Z$ then we define the character $\chi_{\infty}^{(\ell)} : \C^* \rightarrow \C^*$ by 
\begin{equation}\label{chi-infinity-def} \chi_{\infty}^{(\ell)}(z) := (z/|z|)^{\ell}.\end{equation}

\subsection{Acknowledgements}
We thank Sarah Peluse and Terry Tao for helpful conversations. We also thank Roger Heath-Brown, Emmanuel Kowalski and James Maynard for comments pertaining to the literature concerning the prime number theorem for ideals, and Noah Kravitz, James Leng, Jori Merikoski, C\'edric Pilatte, Joni Ter\"av\"ainen and Katharine Woo for comments on a draft of the paper. We thank the anonymous referee for a careful reading of the paper. Finally, we thank Ouyang Xuan for remarking on Math Stack Exchange that the partial summation step in the proof of \cref{main-term-sharps} was done incorrectly in the first version of the paper.

BG is supported by Simons Investigator Award 376201. This research was conducted during the period MS served as a Clay Research Fellow. Portions of this research were conducted while MS was visiting Oxford and Cambridge both of which provided excellent working conditions. Finally, the authors thank ICMS and the organisers of the workshop on Additive Combinatorics (July 2024) which provided the initial impetus for this collaboration.

\section{Number fields and weight functions}\label{section2}

In this section we collect some basic facts about number fields, particularly the imaginary quadratic field $K = \Q(\sqrt{-n})$. We also introduce the key notion of a weight function of product form, \cref{def-product-form}.

Given a number field $K$, denote by $\O_K$ its ring of integers. This is a lattice satisfying $\Z[\sqrt{-n}] \subseteq \O_K \subseteq \frac{1}{2n}\Z[\sqrt{-n}]$. Note in particular that 
\begin{equation}\label{int-contain} (2n) \subseteq \Z[\sqrt{-n}],\end{equation} where $(2n)$ is the ideal generated by $2n$. Write $\O^*_K$ for the units in $\O_K$. This will consist of $\pm 1$ unless the squarefree part of $n$ is 1 (in which case $\O_K^* = \{ \pm 1, \pm i\}$) or $3$ (in which case $\O_K^*$ consists of the sixth roots of unity).\vspace*{8pt}

\emph{Ideals.} Denote by $\Ideals(\O_K)$ the semi-group of non-zero ideals in $\O_K$ and by $\Ideals^0(\O_K)$ the non-zero principal ideals. We will also encounter fractional ideals. If $\alpha \in K^*$ then it is convenient to write $(\alpha) = \alpha \O_K$ and call this the fractional ideal generated by $\alpha$. Note that if $\alpha \in \O_K \setminus \{0\}$ then $(\alpha)$ is a principal ideal, and all principal ideals are of this form. 

Fraktur letters $\mf{a}, \mf{b},\mf{d}$ will always denote ideals in $\O_K$, with the letters $\mf{p}, \mf{q},\mf{r}$ being reserved for prime ideals. We reserve the letter $\mf{c}$ for particular fixed choices of fractional ideals in $K$, as described in \cref{lem:ideal-reps} below.

Recall that we may define the norm $N\mf{a}$ of any non-zero fractional ideal $\mf{a}$, and this takes values in $\Q^+$.

\emph{Ideal classes.} Denote by $C(K)$ the ideal class group of $K$. We will denote elements of $C(K)$ by $g$, and will write $[g]$ for the set of ideals in $\Ideals(\O_K)$ in the class represented by $g$. Sometimes it is necessary to consider fractional ideals and their classes in $C(K)$. Indeed it will be convenient throughout the paper to fix some fractional ideal representatives of each class, as detailed in the following lemma. 
\begin{lemma}\label{lem:ideal-reps}
Let $K = \Q(\sqrt{-n})$. We may choose fractional ideals $\mf{c}_g$, $\mf{c}'_g$ for $g \in C(K)$ with the following properties:
\begin{enumerate}
\item[\textup{(i)}]$\mf{c}_g, \mf{c}'_g$ are in the ideal class corresponding to $g$;
\item[\textup{(ii)}]$\mf{c}_g \mf{c}'_{g^{-1}}$ \textup{(}which is principal\textup{)} is generated by $\gamma_g \in \Q$, a positive rational of height $O_n(1)$;
\item[\textup{(iii)}] The set $\Pi_g := \{ z \in K^* : (z) \mf{c}_g \subseteq\O_K\}$ is a finite index sublattice of $\Z[\sqrt{-n}]$, and similarly for $\Pi'_g := \{ z \in K^* : (z) \mf{c}'_g \subseteq\O_K\}$.
\end{enumerate}
\end{lemma}
\begin{proof}
First of all, for each $g \in C(K)$ pick some $\mf{b}_g \in [g]$ with norm $\le \frac{4}{\pi}\sqrt{n}$; such ideals exist by the Minkowski bound. For each $g$, $\mf{b}_g \mf{b}_{g^{-1}}$ is a principal ideal. Set $\mf{b}'_{g^{-1}} := \overline{\mf{b}_g \mf{b}_{g^{-1}}} \mf{b}_{g^{-1}}$. Then $N\mf{b}'_{g^{-1}} \ll n^{3/2}$, and $\mf{b}_g \mf{b}'_{g^{-1}} = (\mf{b}_g \mf{b}_{g^{-1}})\overline{\mf{b}_g \mf{b}_{g^{-1}}}$ is generated by a rational integer of magnitude $\ll n^2$. Now, for each $g$, take $\mf{c}_g$ to be the fractional ideal $\frac{1}{M_g} \mf{b}_g$, where $M_g = O_n(1)$ is an integer such that $M_g\mf{b}_g^{-1} \subseteq (2n)$. Such an $M_g$ exists by clearing denominators in generators of $\mf{b}^{-1}_g$. Define $\mf{c}'_g = \frac{1}{M'_g}\mf{b}'_g$ similarly. Properties (i) and (ii) are then immediate from the above discussion. (We remark that the key point in (ii) is that the generator $\gamma_g$ lies in $\Q$, rather than merely in $K$; the second point is in a sense vacuous since there are only $O_n(1)$ ideal classes, however the proof yields an explicit value for this constant if desired.)

We turn to point (iii). It is clear that $\Pi_g$ is an abelian group under addition. Also, since $\mf{c}_g$ has an integral basis, we see by clearing denominators that there is $M = O_n(1)$ such that $M, M\sqrt{-n} \in \Pi_g$.
In the other direction, note that by construction we have $\mf{c}_g^{-1} \subseteq (2n)$. Therefore if $(z) \mf{c}_g \subseteq \O_K$ then $(z) = (z) \mf{c}_g \mf{c}_g^{-1} \subseteq \O_K \mf{c}_g^{-1} \subseteq (2n) \subseteq \Z[\sqrt{-n}]$, the second inclusion being \cref{int-contain}. Hence $\Pi_g \subseteq \Z[\sqrt{-n}]$, and the claim follows.
\end{proof}

\emph{Weight functions.} Throughout the paper we will be dealing with weight functions $w$ on $\Ideals(\O_K)$. In \cref{DFI-sec} these can be quite general, but for most of the paper we will take $K = \Q(\sqrt{-n})$ and the weight functions will be of a special product form.

\begin{definition}\label{def-product-form}
Suppose that $K = \Q(\sqrt{-n})$. Suppose that $f, f' : \Q \rightarrow \C$ are two functions, both supported on $\Z$, and let $\ell \in \Z$. Then we define a function $f \boxtimes_{\ell} f'$ on ideals as follows:
\begin{equation}\label{product-form} f \boxtimes_{\ell} f'(\mf{a}) = \sum_{(x + y \sqrt{-n}) = \mf{a}} \chi_{\infty}^{(\ell)}(x + y \sqrt{-n}) f(x) f'(y),\end{equation} 
where $\chi_{\infty}^{(\ell)}(z) = (z/|z|)^{\ell}$.
Any function $w : \Ideals(\O_K) \rightarrow \C$ having this form is said to be in \emph{product form} with \emph{frequency $\ell$}. We say that $w$ is $1$-bounded if $f, f'$ are.
\end{definition}
\begin{remarks}
Note that the sum is empty unless $\mf{a}$ is principal, so $f \boxtimes_{\ell} f'$ is supported on principal ideals. If $\mf{a}$ is principal the the sum is over the $|\O^*_K|$ associates of any generator for $\mf{a}$.
\end{remarks}

\emph{Further facts about $\Q(\sqrt{-n})$}. The following further definitions and facts will be required only in \cref{section8}. Factor
\begin{equation}\label{n-nstar} n = n_* r^2 \end{equation} where $n_*$ is squarefree and $r \in \mathbf{N}$. It is convenient to define a quantity $\omega$ by
\begin{equation}\label{omega-def} \omega := \left\{ \begin{array}{ll} 1 & \mbox{if $n_* \equiv 1,2 \md{4}$ and} \\ \frac{1}{2} & \mbox{if $n_* \equiv 3 \md{4}$}.\end{array}\right.\end{equation}
Then it is well-known that $\O_K = \Z[\sqrt{-n_*}]$ if $\omega = 1$, whereas $\O_K = \Z[\frac{1}{2}(1 + \sqrt{-n_*})]$ if $\omega = \frac{1}{2}$ (which perhaps explains the notation). Recall that the discriminant $\Delta$ is given by the formula
\begin{equation}\label{disc-form} \Delta  = -4\omega^2 n_*.\end{equation}

There is a real quadratic character $\legendre{\Delta}{\cdot}$ associated to the field, where the symbol here is the Kronecker symbol, which coincides with the Legendre symbol on odd primes $p$. This symbol determines the splitting type of a prime $p$; $p$ splits, is inert or ramifies according to whether $\legendre{\Delta}{p} = 1, -1, 0$ respectively. We also let $h_K$ denote the class number of $K$.  Finally we recall the class number formula, which in our setting and language states that 
\begin{equation}\label{class-number} \frac{2\pi h_K}{|\O_K^*| |\Delta|^{1/2}} =  L(1, (\Delta | \, \cdot \,)) = \prod_p \bigg(1 - \frac{(\Delta | p)}{p}\bigg)^{-1}.\end{equation}

\section{Sieve setup-reduction to Type I and Type II statements}\label{section3}

In this section we detail the main sieve-theoretic parts of the argument which allow us to reduce to proving certain Type I/II estimates. 

\subsection{The Duke--Friedlander--Iwaniec sieve in a number field}
\label{DFI-sec}
Here we adapt the sieve arguments of Duke, Friedlander and Iwaniec to an arbitrary number field $K$. The generalisation is essentially completely routine and follows \cite[Section 6]{DFI95} very closely, with only minor technical modifications required to work with ideals rather than integers. We think of $K$ as fixed, and do not explicitly indicate $K$-dependence in instances of asymptotic notation.

\begin{definition}\label{def:input-estim} Fix a weight function $w : \Ideals(\O_K) \rightarrow \C$. Given a real number $B \ge 1$, $w$ satisfies Type I (with savings $(\log X)^{-B})$) at $L_1$ if given $L \le L_1$ then 
\begin{equation}\label{typeI-def} \sum_{\mf{d} : N\mf{d} \sim L} \Big|  \sum_{ \mf{d} | \mf{a}} w(\mf{a})\Big| \ll_B X(\log X)^{-B} \Vert w \Vert_{\infty}. \end{equation}
We say that $w$ satisfies Type II (with savings $(\log X)^{-B}$) between $[L_1,L_2]$ if given $L\in [L_1,L_2]$, and any $1$-bounded sequences $\alpha_{\mf{a}},\beta_{\mf{b}}$ then 
\begin{equation}\label{typeII-def} \sum_{\mf{a}, \mf{b}: N\mf{a} \sim  L } \alpha_{\mf{a}} \beta_{\mf{b}} w(\mf{a} \mf{b})\ll_{B} X(\log X)^{-B}\Vert w \Vert_{\infty}.\end{equation}  

\end{definition}
\begin{remarks}
The definition of course depends on our global parameter $X$. In practice, $w$ will be supported on ideals of norm $\ll X$, in which case both \cref{typeI-def} and \cref{typeII-def} represent savings of a large power of $\log X$ over the trivial bound. In our applications, the parameter $L_1$ in Type II estimates will be purely technical and chosen to grow slightly faster than any power of logarithm.
\end{remarks}

Here is our formulation of Duke, Friedlander and Iwaniec's sieve result in the number field setting.
\begin{lemma}\label{lem:sieve-input}
Let $A, C > 1$ be parameters, set $B := 2A + 4$, and suppose $C\ge B$. Furthermore suppose that $w : \Ideals(\O_K) \rightarrow \C$ satisfies Type I \textup{(}with savings $(\log X)^{-B}$\textup{)} at $X^{1/2}(\log X)^{-C}$ and Type II \textup{(}with savings $(\log X)^{-B}$\textup{)} between $[(\log X)^{C}, X^{3/8}]$. Suppose that $w$ is supported on ideals of norm $\le X$. Then
\[\sum_{\mf{p}} w(\mf{p}) + \sum_{\mf{p}_1,\mf{p}_2}  w(\mf{p}_1\mf{p}_2)\ll_{A,C} X(\log X)^{-A} \snorm{w}_{\infty} ,\] where the first sum ranges over prime ideals $\mf{p}$, and the second sum ranges over pairs $\mf{p}_1, \mf{p}_2$ of prime ideals with $N\mf{p}_i \ge X^{1/2}e^{-(\log\log X)^2}$ for $i = 1,2$.
\end{lemma}
\begin{remark}
To explain the meaning of this result, we note that the sum $\sum_{\mf{p}} w(\mf{p})$ is the one we are ultimately interested in. The second sum, which is over ideals $\mf{a} = \mf{p}_1\mf{p}_2$ with two similarly-sized prime ideal factors, will later be bounded using an upper bound sieve.
\end{remark}

\begin{proof} Since the statement is homogeneous in $w$, we may assume that $w$ is 1-bounded. By adjusting implicit constants in the statement, we may assume that $X$ is larger than an absolute constant depending on $B$. Let 
\begin{equation}\label{param-choices} D = X^{1/2}(\log X)^{-C}, ~u =(\log X)^{C}\text{ and }~z = X^{1/2}e^{-(\log\log X)^2}.\end{equation} We put a total ordering on the prime ideals in $\O_K$ by declaring that $\mf{p}_1 < \mf{p}_2$ if $N\mf{p}_1 < N\mf{p}_2$, and then ordering the prime ideals with the same norm arbitrarily. (The same technical device occurs in \cite[Section 4]{hb-moroz} or \cite[Section 6]{maynard-norm-forms}, for example.) Note that there are at most $[K : \Q]$ prime ideals of a given norm. By an \emph{up-set} we mean a set $I$ of prime ideals with the property that if $\mf{p}_1 \in I$, and if $\mf{p}_2 > \mf{p}_1$, then $\mf{p}_2 \in I$. Since the prime ideals are well-ordered by $<$, every up-set is of the form $I(\mf{p}) := \{ \mf{p}' : \mf{p}' \ge \mf{p}\}$. However, for technical reasons during the proof it is sometimes useful to consider up-sets specified in the form $I(t) := \{ \mf{p} : N \mf{p} \ge t\}$, where $t \in \R^+$. There should hopefully be no confusion over which definition is being used at any given point. 

Let $\mc{A} \subseteq\Ideals(\O_K)$ and let $I$ be an up-set. We define
\[S(\mc{A},I) = \sum_{\substack{\mf{a} \in \mc{A}\\ \mf{p} |\mf{a} \implies \mf{p} \in I}}w(\mf{a}).\]

Given a set $\mc{A}$ of ideals and an ideal $\mf{d}$, we denote by $\mc{A}_{\mf{d}}$ the subset of $\mc{A}$ consisting of ideals divisible by $\mf{d}$.

Now note that if $I_1, I_2$ are up-sets with $I_1 \subseteq I_2$ then we have the \emph{Buchstab identity}
\[ S(\mc{A}, I_1) = S(\mc{A}, I_2) - \sum_{\mf{p} \in I_2 \setminus I_1} S(\mc{A}_{\mf{p}}, I(\mf{p})).\]
To see this, observe that if $w(\mf{a})$ appears in the sum $S(\mc{A}, I_2)$ but not in $S(\mc{A}, I_1)$ then $\mf{a}$ is divisible by a unique smallest prime ideal $\mf{p}$ with $\mf{p} \in I_2 \setminus I_1$. Of course, $\mf{a}$ lies in $\mc{A}_{\mf{p}}$, and since $\mf{p}$ was chosen to be smallest, all other prime ideal divisors of $\mf{a}$ are $\ge \mf{p}$, or in other words lie in $I(\mf{p})$.

For notational brevity take $\mc{C}= \Ideals(\O_K)$. By two applications of the Buchstab identity we have
\begin{align} \nonumber
S(\mc{C},I(z)) &= S(\mc{C},I(u)) - \sum_{\mf{p} \in I(u) \setminus I(z)}S(\mc{C}_\mf{p},I(\mf{p}))\\
& \nonumber =S(\mc{C},I(u)) - \sum_{\mf{p} \in I(u) \setminus I(z)}\bigg(S(\mc{C}_\mf{p},I(u))-\sum_{\mf{q} \in I(u) \setminus I(\mf{p})}
S(\mc{C}_{\mf{p}\mf{q}},I(\mf{q}))\bigg)\\
& \label{two-buchs} =S(\mc{C},I(u)) - \sum_{u \le N\mf{p} < z}S(\mc{C}_\mf{p},I(u)) + \sum_{\substack{N\mf{p} \ge u \\ N\mf{q} < z \\ \mf{p} < \mf{q}}}S(\mc{C}_{\mf{p}\mf{q}},I(\mf{p})).
\end{align}
Note that in the last step here, we switched the roles of $\mf{p}$ and $\mf{q}$, and unpacked the definitions of the $I()$ terms.

On the other hand, note that $z^{3}>X$ and recall that if $w(\mf{a})$ is non-zero then $N\mf{a} \ll X$. Therefore the sum $S(\mc{C}, I(z))$ contains only two types of term $w(\mf{a})$: those with $\mf{a}$ prime, and those with $\mf{a}$ a product of two prime ideals, both of norm $\ge z$. That is,
\begin{equation}\label{sieved-sum} S(\mc{C},I(z)) = \sum_{\mf{p}: N\mf{p}\ge z} w(\mf{p}) + \sum_{\mf{p}_1,\mf{p}_2} w(\mf{p}_1\mf{p}_2),\end{equation}
where the second sum is over $\mf{p}_1,\mf{p}_2$ with $N\mf{p}_i \ge z = X^{1/2} e^{-(\log\log X)^2}$, that is to say the second sum here is exactly the one appearing in \cref{lem:sieve-input}. The constraint $N\mf{p} \ge z$ will easily be removed by trivial bounds, so our task is to bound $S(\mc{C}, I(z))$, which we shall do via \cref{two-buchs}.\vspace*{8pt}

\textbf{Step 1: Deploying Type I information.}

The first step is to bound the first two terms in \eqref{two-buchs}, $S(\mc{C},I(u)) - \sum_{u \le  N\mf{p}<z}S(\mc{C}_{\mf{p}},I(u))$, via Type I information, or in other words by use of \cref{typeI-def}. Define $P(u) := \prod_{\mf{p} : N\mf{p} < u} \mf{p}$ and let $\Pi_u$ be the set of all $\mf{a}$ with at most one prime factor of norm $\ge u$. Then, unwinding definitions, we have that 
\begin{equation}\label{scu}S(\mc{C},I(u)) - \sum_{u \le N \mf{p}<z}S(\mc{C}_\mf{p},I(u)) = \sum_{(\mf{a}, P(u)) = 1}w(\mf{a})\bigg(1 - \sum_{\substack{u \le  N\mf{p}<z\\ \mf{p}|\mf{a}}} 1\bigg).\end{equation}
Denote by $\mu_K$ the M\"obius function on (ideals of) $K$. Then we have
\[ \sum_{\substack{\mf{d} | (P(z), \mf{a}) \\ \mf{d} \in \Pi_u }} \mu_K(\mf{d}) = 1_{(\mf{a}, P(u)) = 1} \cdot \bigg(1 - \sum_{\substack{u \le N\mf{p}<z\\ \mf{p}|\mf{a}}} 1\bigg).\]
Substituting into \cref{scu} yields
\begin{equation}\label{subst3} S(\mc{C},I(u)) - \sum_{u \le N\mf{p}<z}S(\mc{C}_\mf{p},I(u)) = \sum_{\substack{\mf{d}|P(z)\\\mf{d}\in \Pi_u}}\mu_K(\mf{d})\sum_{\mf{d}|\mf{a}}w(\mf{a}).\end{equation}
We now break the sum into two parts based on where $N \mf{d}\le D$ and else when $N \mf{d}>D$. For the first of these we have that 
\begin{equation}\label{sum-small-d}
\Big|\sum_{\substack{\mf{d}|P(z)\\ N \mf{d} \le D\\ \mf{d}\in \Pi_u}}\mu_K(\mf{d})\sum_{\mf{d}|\mf{a}}w(\mf{a})\Big| \le \sum_{N\mf{d}\le D}\Big|\sum_{\mf{d}|\mf{a}} w(\mf{a})\Big|\ll X(\log X)^{-B+1} \ll X (\log X)^{-A};
\end{equation}
here we have applied dyadic summation in $N \mf{d}$ and the Type I estimate \cref{typeI-def} at each dyadic scale, and recall that $w$ is $1$-bounded.

For the sum over $N\mf{d} >D$, we use the trivial bound $|w(\mf{a})| \le 1_{N\mf{a} \le X}$ (that is, $w$ is $1$-bounded on supported on ideals of norm $\le X$), obtaining
\begin{equation}\label{209}
\Big|\sum_{\substack{\mf{d}|P(z)\\ N \mf{d} > D\\ \mf{d}\in \Pi_u}}\mu_K(\mf{d})\sum_{\mf{d}|\mf{a}}w(\mf{a})\Big|  \le \sum_{\substack{\mf{d} | P(z) \\ D < N\mf{d} \le X \\ \mf{d} \in \Pi_u}} \Big| \sum_{\substack{N\mf{a} \le X \\ \mf{d} | \mf{a}}} 1 \Big|    \ll   \sum_{\substack{\mf{d} | P(z) \\ D < N\mf{d} \le X \\ \mf{d} \in \Pi_u}} \frac{X}{N\mf{d}}.
\end{equation}
By dyadic decomposition, this is bounded above by

\begin{equation}\label{210}
\ll  X(\log X)\sup_{L\in [D,2X]}\sum_{\substack{\mf{d}|P(z)\\ N\mf{d}\sim L\\\mf{d}\in \Pi_u}} \frac{1}{L}.
\end{equation}
Note that every $\mf{d}$ appearing in the sum here is squarefree (since it divides $P(z)$) and has at most one prime factor of norm $\ge u$. The product of the (distinct) prime divisors of $\mf{d}$ of norm $< u$ must therefore have norm at least $D/z$, and therefore there must be at least $t := \lfloor \log(D/z)/\log u\rfloor$ of them. Using \cref{num-ideals-lem} and \cref{number-field-mertens}, this implies that 
\[
\sum_{\substack{\mf{d}|P(z)\\ N\mf{d}\sim L\\\mf{d}\in \Pi_u}} \frac{1}{L}  \ll \frac{1}{L} \sum_{ \substack{\mf{p}_1 < \cdots < \mf{p}_t \\ N\mf{p}_i < u}} \frac{L}{N (\mf{p}_1 \cdots \mf{p}_t)}  \ll \frac{1}{t!}\big(\sum_{\mf{p} : N\mf{p} < u}\frac{1}{N\mf{p}}\big)^t\le \frac{(2\log\log u)^t}{t!}\ll (\log X)^{-A-1}.
\]
For the final calculation here, it is helpful to note that $t = (\frac{1}{C} + o(1)) \log \log X$, which follows from the choice of $D, u, z$ specified in \cref{param-choices}.

Substituting into \cref{210}, we see that the contribution to \cref{subst3} from the sum over $N\mf{d} > D$ is $\ll X (\log X)^{-A}\Vert w \Vert_{\infty}$. Combining with the estimate \cref{sum-small-d} for the sum over $N\mf{d} \le D$, we obtain the bound
\begin{equation}\label{typei-portion}
S(\mc{C},I(u)) - \sum_{u \le  N\mf{p}<z}S(\mc{C}_{\mf{p}},I(u)) \ll X (\log X)^{-B + 1} \ll X (\log X)^{-A} .\end{equation}
This concludes our analysis of the first two terms in \eqref{two-buchs} by the use of Type I sums.\vspace*{8pt}

\textbf{Step 2: Splitting the $\mf{p}$ variable.}

We now proceed to bounding the third term in \cref{two-buchs}, that is to say $\sum_{N\mf{p} \ge u ,  N\mf{q} < z ,\mf{p} < \mf{q}} S(\mc{C}_{\mf{p}\mf{q}},I(\mf{p}))$. Here, we will use the Type II estimate \cref{typeII-def}, but first we make some preliminary man{\oe}uvres aimed at decoupling the condition $\mf{p} < \mf{q}$ into conditions on $\mf{p}, \mf{q}$ separately. We first estimate the contribution from large $\mf{p}$. To this end set
\begin{equation}\label{y-choice} y := X^{3/8}.\end{equation}
Then we have
\begin{equation}\label{large-p-term}\sum_{\substack{N\mf{p} \ge y \\ N\mf{q} < z \\ \mf{p} < \mf{q}}}S(\mc{C}_{\mf{p}\mf{q}},I(\mf{p})) = \sum_{\substack{N\mf{p} \ge y \\ N\mf{q} < z \\ \mf{p} < \mf{q}}}w(\mf{p}\mf{q})\ll z^2  \ll X(\log X)^{-A} .\end{equation}
In the first step, we observed that the only ideal $\mf{a}$ of norm $\ll X$, divisible by $\mf{p}$ and $\mf{q}$ and with all other prime ideal factors in $I(\mf{p})$, is $\mf{a} = \mf{p}\mf{q}$ itself. For the last step, we recall the choice \cref{param-choices} of $z$.  The estimate \cref{large-p-term} is clearly acceptable in \cref{lem:sieve-input}. The remaining task is to estimate the contibution from  smaller $\mf{p}$, namely
\begin{equation}\label{rem-to-est}\sum_{\substack{u \le N\mf{p} < y \\ N\mf{q} < z \\ \mf{p} < \mf{q}}}S(\mc{C}_{\mf{p}\mf{q}},I(\mf{p})).\end{equation}
Now we begin the decoupling man{\oe}uvre in earnest. Set 
\begin{equation}\label{m-choice} M := \lfloor (\log X)^{1 + B/2}\rfloor \end{equation} and choose spacing points
\[ y = y_0 >y_1>y_2>\cdots >y_M = u\] where $y_m = y(u/y)^{m/M}$. Then splitting according to the size of $N\mf{p}$ gives
\begin{align*}
\sum_{\substack{u \le N\mf{p} < y \\ N\mf{q} < z \\ \mf{p} < \mf{q}}}&S(\mc{C}_{\mf{p}\mf{q}},I(\mf{p})) =\sum_{m=0}^{M-1} \sum_{\substack{y_{m+1} \le N\mf{p} < y_m \\ N\mf{q} < z \\ \mf{p} < \mf{q}}}S(\mc{C}_{\mf{p}\mf{q}},I(\mf{p}))\\
&=\sum_{m=0}^{M-1} \sum_{\substack{y_{m+1} \le N\mf{p}, N\mf{q} < y_m \\ \mf{p} < \mf{q} }}S(\mc{C}_{\mf{p}\mf{q}},I(\mf{p})) + \sum_{m = 0}^{M-1}\sum_{y_{m+1}\le N\mf{p}<y_m\le N\mf{q}<z}S(\mc{C}_{\mf{p}\mf{q}},I(\mf{p})).
\end{align*}
Applying (assuming here the conditions $y_{m+1}\le N\mf{p}<y_m\le N\mf{q}<z$, as in the right-hand sum) the Buchstab identity
\[S(\mc{C}_{\mf{p} \mf{q}}, I(\mf{p})) = S(C_{\mf p \mf q}, I(y_{m+1})) - \sum_{\mf{r} \in I(y_{m+1}) \setminus I(\mf{p})}  S(\mc{C}_{\mf{p} \mf{q} \mf{r}}, I(\mf{r})),\]
one confirms that
\begin{equation}\label{buchstab-levels} \sum_{\substack{u \le N\mf{p} < y \\ N\mf{q} < z \\ \mf{p} < \mf{q}}}   S(\mc{C}_{\mf{p}\mf{q}},I(\mf{p})) 
 = \sum_{m=0}^{M-1} \sum_{i = 1}^3 E^{(i)}_m,\end{equation} where
 \[ E^{(1)}_m := 
\sum_{\substack{y_{m+1} \le N\mf{p} \\ N\mf{q} < y_m \\ \mf{p} < \mf{q} }}S(\mc{C}_{\mf{p}\mf{q}},I(\mf{p})), \qquad E^{(2)}_m := \sum_{\substack{y_{m+1}\le N\mf{p}<y_m \\ y_m\le N\mf{q}<z}}S(\mc{C}_{\mf{p}\mf{q}},I(y_{m+1})) \] and
\[ E^{(3)}_m := -\sum_{y_{m+1} \le N\mf{r} < N\mf{p} < y_m \le N\mf{q} < z} S(C_{\mf{p} \mf{q} \mf{r}}, I(\mf{r})).\]
What has been achieved here is that the first and third types of sum $E^{(1)}_m$ and $E^{(3)}_m$ correspond to atypical factorization patterns (with two prime factors very close in norm) and can ultimately be bounded fairly trivially, whereas in the term $E^{(2)}_m$ the conditions on $\mf{p}, \mf{q}$ are fully decoupled, which will allow us to relate these terms to Type II sums.

We turn now to the actual estimations. Looking first at $E^{(1)}_m$, note that 
\begin{align}\nonumber |E^{(1)}_m| &\le \sum_{\substack{y_{m+1}\le N\mf{p}, N\mf{q}<y_m \\ \mf{p} < \mf{q}}}\sum_{\mf{p}\mf{q}|\mf{a}}|w(\mf{a})|\ll X\sum_{\substack{y_{m+1}\le N\mf{p}, N\mf{q}<y_m \\ \mf{p} < \mf{q}}}(N\mf{p}\mf{q})^{-1}\\ 
&\le X\bigg(\sum_{y_{m+1}\le N\mf{p}<y_{m}}(N\mf{p})^{-1}\bigg)^2\ll X \bigg(\frac{y_m}{y_{m+1}} - 1\bigg)^2\ll X(\log X)^2/M^2.\label{e1-est}
\end{align}
Here, in the penultimate step we bounded the sum over primes by the sum over all ideals $\mf{a}$ in $\O_K$ satisfying $y_{m+1}\le N\mf{a}<y_{m}$, then used \cref{num-ideals-lem} to bound the resulting sum. The estimation of $E^{(3)}_m$ is almost identical. Since an ideal $\mf{a}$ with $N\mf{a} \le X$ has $O(\log X)$ prime factors, each ideal $\mf{a}$ with $\mf{p} \mf{r} | \mf{a}$ is counted by $\ll \log X$ choices of $\mf{q}$, and so we have
\begin{equation}\label{e3-est}
|E^{(3)}_m| \ll  (\log X)\sum_{y_{m+1}\le N\mf{r}< N\mf{p}<y_m}\sum_{\mf{p}\mf{r}|\mf{a}}|w(\mf{a})| \ll  X(\log X)^3/M^2,
\end{equation} finishing here in the same way as the estimation of $E^{(1)}_m$.\vspace*{8pt}

\textbf{Step 3: Deploying Type II information.}

We turn now to the heart of the matter, which is the estimation of $E^{(2)}_m$ using Type II information. The key point is observing that $E^{(2)}_m$ \emph{is} a Type II sum. Indeed we have
\begin{equation}\label{e2-typeii} E^{(2)}_m = \sum_{\mf{a}, \mf{b}} \alpha_{\mf{a}} \beta_{\mf{b}} w(\mf{a} \mf{b})\end{equation}
with
\[ \alpha_{\mf{a}} = 1_{\mf{a} \operatorname{prime}} 1_{y_{m+1} \le N\mf{a} < y_m} \qquad \mbox{and} \qquad \beta_{\mf{b}} = 1_{N\mf{b}\le X} \sum_{\substack{y_m \le N \mf{q} < z \\ \mf{q} | \mf{b}}} 1_{P^{-}(\mf{b}) \ge y_{m+1}},\]
where $P^{-}(\mf{b})$ denotes the norm of the smallest prime ideal factor of $\mf{b}$. Clearly $|\alpha_{\mf{a}}| \le 1$, and $|\beta_{\mf{b}}|$ is at most the number of prime ideal divisors of $\mf{b}$, which is $\ll \log X$. Note also that $y_m/y_{m+1} = (y/u)^{1/M} < 2$, and so $\alpha_{\mf{a}}$ is already supported on a dyadic range of $\mf{a}$, and so is of the form \cref{typeII-def} with $L = y_m$. Note (recalling the choice \cref{y-choice} of $y$) that $u \le L \le y = X^{3/8}$, and so for all values of $m$ our assumption on the Type II properties of $w$ gives, from \cref{e2-typeii}, that 
\begin{equation}\label{e2-est} |E^{(2)}_m| \ll X (\log X)^{-B + 1}.\end{equation}

We have now proven all of the necessary estimates; let us conclude the proof of \cref{lem:sieve-input} by putting them together.

From \cref{buchstab-levels}, \cref{e1-est}, \cref{e3-est}, \cref{e2-est} and the choice \cref{m-choice} of $M$ and the fact that $B = 2A + 4$, we have
\[ \sum_{\substack{u \le N\mf{p} < y \\ N\mf{q} < z \\ \mf{p} < \mf{q}}}   S(\mc{C}_{\mf{p}\mf{q}},I(\mf{p}))  \ll X (\log X)^3 M^{-1} + X(\log X)^{-B + 1}M  \ll X(\log X)^{-A}.\] The quantity on the left here arose as \cref{rem-to-est}. Together with \cref{large-p-term}, this gives the bound
\[ \sum_{\substack{N\mf{p} \ge u \\ N\mf{q} < z \\ \mf{p} < \mf{q}}}   S(\mc{C}_{\mf{p}\mf{q}},I(\mf{p})) \ll  X(\log X)^{-A}.\] The quantity on the left here is the third term in \cref{two-buchs}. Combining with the estimate \cref{typei-portion} for the first two terms in \cref{two-buchs}, we obtain
\[ S(\mc{C}, I(z)) \ll X (\log X)^{-A}.\]
Finally, from \cref{sieved-sum} and the trivial bound $|\sum_{\mf{p} : N\mf{p} < z} w(\mf{p})| \ll z$, we obtain the bound claimed in \cref{lem:sieve-input}.
\end{proof}

\subsection{Bounding the sum over semiprimes}
For \cref{lem:sieve-input} to be useful for sums over primes, we need a way to estimate the term $\sum_{\mf{p}_1, \mf{p}_2} w(\mf{p}_1\mf{p}_2)$ involving almost equally-sized semiprimes. We do this now under an assumption on $w$ suitable for our intended purpose (of proving \cref{thm:main}). We remark that in the work of Ford and Maynard \cite[Theorem~4.16]{FM24}, they point to the fact that if one only achieves Type I information below $1/2$ and $w$ is not bounded then one may not deduce any result regarding primes. However the weight function underlying \cite[Theorem~4.16]{FM24} places a substantial portion of weight on semiprimes; the role of the upper bound sieve is precisely to rule out this conspiracy.

At this point we specialise to the case of imaginary quadratic fields $K = \Q(\sqrt{-n})$ and to weight functions of product form \cref{product-form}. 

\begin{lemma}\label{lem:2primes-upper-bound}
Let $K = \Q(\sqrt{-n})$. Suppose that $w : \Ideals(\O_K) \rightarrow \C$ is in product form \cref{product-form} for some frequency $\ell \in \Z$. Suppose moreover that $f, f'$ satisfy the pointwise bound
\begin{equation}\label{w-cond} |f(x)|, |f'(x)|  \le \Lambda_{\cramer}(x) + \Lambda'(x).\end{equation} where $\Lambda_{\cramer}, \Lambda'$ are defined to be zero on non-integer arguments. Let $\eta \in [X^{-1/10} , \frac{1}{2})$ be a parameter. Then 
\[ \sum_{\substack{\mf{p}_1,\mf{p}_2\\ N\mf{p}_1\mf{p}_2 \le X \\ N\mf{p}_i \ge \eta X^{1/2}} }\big|w(\mf{p}_1\mf{p}_2)\big|\ll \frac{X \log(1/\eta)}{(\log X)^2}.\]
\end{lemma}
\begin{remark} We take the domain of definition of $\Lambda'$ to be $\Q$, though it is supported on $\Z$.
 The key point of the lemma is that it represents a nontrivial saving over the trivial bound of $O(X/\log X)$.
\end{remark}
\begin{proof}
We first dispense with the contribution from either $\mf{p}_1$ or $\mf{p}_2$ being generated by a rational prime. This is straightforward, since the number of such ideals with norm $\le M$ is $\ll M^{1/2}$. Using the trivial bound $|w(\mf{p}_1 \mf{p}_2)| \le \log^2 X$, the total contribution from these terms is $\ll (X^{1/2}/\eta)^{3/2} \log^2 X$, which is acceptable by a large margin.

It therefore suffices to establish for each fixed $\mf{p}$ not generated by a rational prime, and with $\eta X^{1/2} \le N\mf{p} \le X^{1/2}$, that 
\begin{equation}\label{to-est-23}\sum_{\mf{q}: N\mf{q} \le X/N\mf{p}}|w(\mf{p}\mf{q})| \ll \frac{X}{(N\mf{p})\log X},\end{equation} where the sum is over prime ideals $\mf{q}$ not generated by a rational prime.
Indeed, this implies the result using 
\[
\sum_{\substack{\mf{p}_1,\mf{p}_2\\ N\mf{p}_1\mf{p}_2 \le X \\ N\mf{p}_i\ge \eta X^{1/2}}}  \big|w(\mf{p}_1\mf{p}_2)\big| \ll \sum_{\eta X^{1/2}\le N\mf{p}\le X^{1/2}}\sum_{q: N\mf{q} \le X/N\mf{p}}|w(\mf{p}\mf{q})|
\]
and the fact, which follows using one of the Mertens' estimates in number fields (see \cref{number-field-mertens}), that
\[  \sum_{ \eta X^{1/2}\le N\mf{p} \le X^{1/2}}(N\mf{p})^{-1} \\
\ll \frac{\log (1/\eta)}{\log X} ,\] where in these sums all ideals are prime but not generated by a rational prime.

We turn to the proof of \cref{to-est-23}. An important subtlety here is that, while $w = f \boxtimes_{\ell} f'$ is supported on principal ideals, the prime ideals $\mf{p}$ and $\mf{q}$ may not be principal. To deal with this, we use the representatives $\mf{c}_g, \mf{c}'_g$ for the ideal classes as described in \cref{lem:ideal-reps}. Suppose that $\mf{p} \in [g]$ for some $g \in \classgroup$. Since $w$ is supported on principal ideals, we may restrict attention to $\mf{q} \in [g^{-1}]$. Henceforth, write $\mf{c} = \mf{c}_g$ and $\mf{c}' = \mf{c}'_{g^{-1}}$. Then by \cref{lem:ideal-reps} we have $\mf{p} = (a + b\sqrt{-n}) \mf{c}$ for some $a,b \in \Z$ and $\mf{q}$ is always of the form $(x + y \sqrt{-n}) \mf{c}'$ for some $x, y \in \Z$. Also, $\mf{c} \mf{c'} = (\gamma)$ for some rational $\gamma$ of height $O_n(1)$. 

The generators of $\mf{p}\mf{q}$ are the numbers $u \gamma (a + b\sqrt{-n})(x + y\sqrt{-n})$, $u \in \O^*_K$. Writing the units $u$ as $\zeta_j + \zeta'_j \sqrt{-n}$, $j = 1,\dots, r$, where $\zeta_j, \zeta'_j \in \frac{1}{2n}\Z$, one may compute that these numbers are $a_{1,j}x + b_{1,j}y + (a_{2,j} x + b_{2,j} y)\sqrt{-n}$, where $a_{1,j} = \gamma( a\zeta_j - nb \zeta'_j)$,  $b_{1,j} = -\gamma n( a \zeta'_j + b \zeta_j)$, $a_{2,j} = \gamma (a\zeta'_j + b \zeta_j)$, and $b_{2,j} = \gamma (a\zeta_j-nb\zeta'_j)$. Note here that we have $a_{i,j}, b_{i,j} \in \frac{\gamma}{2n} \Z$ for all $i,j$. Thus from the assumption that $w$ has the product form \cref{product-form} and the bound \cref{w-cond} we have 
\begin{equation}\label{eq277} |w(\mf{p}\mf{q})| \le \sum_{j = 1}^r (\Lambda_{\cramer} + \Lambda') (a_{1,j}x + b_{1,j}y) (\Lambda_{\cramer} + \Lambda')( a_{2,j}x + b_{2,j}y ). \end{equation}
Here, $\Lambda_{\cramer}$ and $\Lambda'$ take the value zero if the arguments are not integers. Observe also that 
\begin{equation}\label{non-deg} \frac{1}{n}(a_{1,j}^2 n + b_{1,j}^2) = a_{2,j}^2 n + b_{2,j}^2 =  a_{1,j} b_{2,j} - a_{2,j} b_{1,j} = \gamma^2 (a^2 + nb^2) \neq 0.\end{equation}
Let $W :=  X^{1/100}$. Then, recalling the definition \cref{cramer-def}, we have the pointwise bound $\Lambda'(x) \ll \Lambda_{\cramer, W}(x)$ for $W < |x| \ll X^2$, and so from \cref{eq277} we have
\begin{equation}\label{eq277a} |w(\mf{p}\mf{q})| \ll \sum_{j = 1}^r \max_{Q_1, Q_2 \in \{Q, W\}} \Lambda_{\cramer, Q_1} (a_{1,j}x + b_{1,j}y) \Lambda_{\cramer, Q_2}( a_{2,j}x + b_{2,j}y ). \end{equation}
unless $|a_{i,j} x + b_{i,j} y| \le W$ for some $i,j$.

We turn to the task of proving \cref{to-est-23}.  The constraint $N\mf{q} \le X/N\mf{p}$ is contained in the constraint $|x|, |y| \ll (X/N\mf{p})^{1/2}$ (note here that $\mf{c}'$ is simply a fixed fractional ideal with norm $\asymp_n 1$). The contribution from $|a_{i,j} x + b_{i,j} y| \le W$ will be negligible. We sum \cref{eq277a} over all other $x,y$ with $|x|, |y| \ll (X/N\mf{p})^{1/2}$, ignoring any congruence conditions necessary for $\mf{q}$ to actually be an ideal in $\O_K$, as we may by the positivity of all the terms involved. In order for $\mf{q}$ to be a prime ideal, we must have $N\mf{q} = (x^2 + ny^2) N \mf{c}'$ equal to a rational prime or the square of a rational prime.

Thus, writing $\gamma' := N\mf{c}'$ and $a_1 = a_{1,j}$ etc, it is enough to prove that
\[  \sum_{|x|, |y| \ll (X/N\mf{p})^{1/2}} 1_{\operatorname{prime}}((x^2 + ny^2)\gamma')\Lambda_{\cramer, Q_1}(a_1 x +  b_1 y)\Lambda_{\cramer,Q_2}(a_2 x + b_2 y) \ll \frac{X }{(N\mf{p}) \log X},\] for any $Q_1, Q_2 \in \{Q, W\}$, where recall that $Q$ is given by \cref{q-choice} and $W = X^{1/100}$. (We omit the contribution from $(x^2 + ny^2)\gamma'$ the square of a prime, which is easily shown to be considerably smaller.) Since $1_{\operatorname{prime}}(t) \ll \frac{1}{\log X} \Lambda_{\cramer, W}(t)$ for $|t| > W$, it is enough to prove that 
\begin{equation}\label{eq177} \sum_{|x|, |y| \ll N} \Lambda_{\cramer,W}((x^2 + ny^2)\gamma')\Lambda_{\cramer, Q_1}(a_1 x +  b_1 y)\Lambda_{\cramer,Q_2}(a_2 x + b_2 y) \ll N^2 \end{equation} for any $Q_1, Q_2 \in \{Q, W\}$, where here we have omitted the negligible contribution from $|(x^2 + ny^2)\gamma'| \le W$, and now have written
\begin{equation}\label{N-def} N := (X/N\mf{p})^{1/2}. \end{equation} 
Note that since $\eta X^{1/2} \le N\mf{p} \le X^{1/2}$ we have $X^{1/4} \ll N \ll X^{1/4}\eta^{-1/2}$. Recalling that $\eta^{-1}\ll X^{1/10}$, it follows that $N \le X^{1/3}$, so certainly $W < N^{1/30}$.

The desired estimate \cref{eq177} is a standard type of sieve bound, but it is not easy to find a suitable reference because (a) the sieving ranges $W, Q_1, Q_2$ can be different, so the underlying sieve does not have a well-defined dimension; (b) we need to track uniformity in the parameters $a_i, b_i$; and (c) there are two variables $x,y$. We have found an approach based on an application of the large sieve to be the most convenient, though `small sieve' arguments could also be made to work. The result we need is \cref{large-sieve-higher}, which is a 2-dimensional variant of \cite[Equation (7.38)]{IK-book}. Whilst generalising this kind of result from 1 to 2 or more dimensions presents no difficulty, we do not know of a reference for exactly what we need, and so we provide some details in \cref{appendixC}.

Returning to the proof of \cref{eq177}, suppose that $Q_1 \le Q_2 \le W$ (the other case $Q_2 \le Q_1 \le W$ is essentially identical). For rational primes $p$, we define sets $\Omega_p$ as follows. Denote by $\mc{P}_*$ the set of primes dividing $2n$, $a^2 + nb^2$ or the numerator or denominator of $\gamma$ or $\gamma'$. Note that $\mc{P}_*$ is a set of $O_n(1)$ primes, since $(a^2 + nb^2)\gamma^2/\gamma' = N\mf{p}$ and $\mf{p}$ is a prime ideal. 

We will be using \cref{large-sieve-higher}; let us now define the sets of residues $\Omega_p$ with which we will apply that result. If $p \in \mc{P}_*$, set
\[ \Omega_p = \emptyset.\] If $p \notin \mc{P}_*$, define $\Omega_p$ as follows.
For $p \le Q_1$, define 
\[ \Omega_p := \{ (u,v) \in (\Z/p\Z)^2 : (u^2 + nv^2)(a_1 u + b_1 v)(a_2 u + b_2 v) \equiv 0 \md{p}\};\]
for $Q_1 < p \le Q_2$ (which may be an empty range) define
\[ \Omega_p := \{ (u,v) \in (\Z/p\Z)^2 : (u^2 + nv^2)(a_2 u + b_2 v) \equiv 0 \md{p}\};\]
for $Q_2 < p \le W$ (which may also be an empty range) define
\[ \Omega_p := \{ (u,v) \in \Z/p\Z : u^2 + nv^2 \equiv 0 \md{p}\},\] and define $\Omega_p = \emptyset$ for $p > W$. The purpose of making these definitions is that we have 
\begin{align}\nonumber
\Lambda_{\cramer,W}((x^2 + ny^2)\gamma') \Lambda_{\cramer, Q_1}  (a_1 x + & b_1 y)\Lambda_{\cramer,Q_2}(a_2 x + b_2 y)  \\ & \ll (\log Q_1)(\log Q_2)(\log W) \prod_{p} 1_{\Omega_p^c}(x,y) \label{cramer-sieve}
\end{align}
(recall here that $a_1,a_2,b_1, b_2 \in \frac{\gamma}{2n} \Z$).

In order that we may apply \cref{large-sieve-higher}, we compute the densities $\alpha_p = |\Omega_p|/p^2$ in the different ranges. Here, note that if $p \notin \mc{P}_*$ then, from \cref{non-deg}, we have
 
 \begin{equation}\label{bad-moduli} p \nmid  (a_1 b_2 - a_2 b_1)(na_1^2 + b_1^2)(na_2^2 + b_2^2) .\end{equation} From this, we may compute that 
 \begin{equation}\label{density} \alpha_p = \frac{1}{p}\big(d(p) + \legendre{-n}{p}\big) + O\bigg(\frac{1}{p^2}\bigg),\end{equation}
 where $d(p) = 3$ if $p \le Q_1$, $d(p) = 2$ if $Q_1 < p \le Q_2$, and $d(p) = 1$ if $Q_2 < p \le W$. (Here $d$ is essentially the dimension of the sieve on the corresponding ranges.) 
 In particular, for all $p$ we have $\alpha_p \le 4/p$.

Now we bound the sum $\sum_{q \le N^{1/2}}\mu^2(q) h(q)$ appearing in \cref{large-sieve-higher} from below. Recall that $h$ is the multiplicative function satisfying $h(p) = \alpha_p/(1 - \alpha_p)$. To do this we use essentially Rankin's trick, but it is instructive to phrase this in a probabilistic language. Define a random variable $Z := \prod_{p \le W} Z_p$, where the $Z_p$ are independent random variables with $Z_p = p$ with probability $\alpha_p$, and $Z_p = 1$ otherwise. Thus for squarefree $q$ we have
\[ \P(Z = q) = h(q) \prod_{p \le W} (1 - \alpha_p).\]
We give an upper bound for the quantity in \cref{large-sieve-est}. We have
\begin{equation}\label{zn-half} \sum_{q \le N^{1/2}}\mu^2(q) h(q) = \prod_{p \le W} (1 - \alpha_p)^{-1}\P \big( Z \le N^{1/2}  \big).\end{equation}
On the other hand,
\[
  \E \log Z = \sum_{p \le W} \E \log Z_p   = \sum_{p \le W} \alpha_p \log p  \le 4 \sum_{p \le W} \frac{\log p}{p} < \frac{\log N}{4}.
\]
By Markov's inequality it follows that $\P(Z > N^{1/2}) < \frac{1}{2}$, and so from \cref{zn-half}
\begin{equation}\label{sumhq} \sum_{q \le N^{1/2}}\mu^2(q) h(q) \ge \frac{1}{2} \prod_{p \le W} (1 - \alpha_p)^{-1}.\end{equation}
Now from \cref{density} we have (since the product of $1 + O(p^{-2})$ terms converges)
\begin{equation}\label{prod-without-D} \prod_{p \le W} (1 - \alpha_p) \asymp \prod_{p \le W}\bigg(1 - \frac{d(p)}{p}\bigg)  \prod_{p \le W} \bigg(1 - \frac{\legendre{-n}{p}}{p}\bigg).\end{equation} However we have
\[ \prod_{p \le W}\bigg(1 - \frac{\legendre{-n}{p}}{p}\bigg) \asymp 1 \]
since $L(1, \legendre{-n}{\, \cdot \,}) \ne 0$. (See, e.g, \cite[Theorem 4.11]{MV-book} for a full discussion of the convergence of the infinite product here, which is only conditionally convergent.) Substituting into \cref{prod-without-D} and recalling the definition of the dimension function $d(\cdot)$ gives
\[ \prod_{p \le W} (1 - \alpha_p) \asymp \prod_{p \le W}\bigg(1 - \frac{d(p)}{p}\bigg)  \asymp \frac{1}{(\log Q_1)(\log Q_2)(\log W)}. \]

From this, \cref{sumhq,cramer-sieve,large-sieve-higher}, we obtain the stated bound \cref{eq177}.
This concludes the proof of \cref{lem:2primes-upper-bound}.
\end{proof}

Combining \cref{lem:sieve-input,lem:2primes-upper-bound} and taking $A = 4$ immediately leads to the following, which is the only result from this section that we will need in what follows.

\begin{proposition}\label{prop:main-sec3}
Fix $C\ge 28$. Let $K = \Q(\sqrt{-n})$, and suppose that $w : \Ideals(\O_K) \rightarrow \C$ is in product form \cref{product-form} with some frequency $\ell \in \Z$, and suppose moreover that $f, f'$ in that definition satisfy the pointwise bound $|f(x)|, |f'(x)|  \le (\Lambda_{\cramer} + \Lambda')(x)$. Define $\tilde w(\mf{a}) := w(\mf{a}) 1_{N\mf{a} \le X}$, and suppose that $\tilde w$ satisfies Type I \textup{(}with savings $(\log X)^{-12}$\textup{)} at $X^{1/2}(\log X)^{-C}$ and Type II \textup{(}with savings $(\log X)^{-12}$\textup{)} between $[(\log X)^{C}, X^{3/8}]$. Then
\[\sum_{\mf{p} : N\mf{p} \le X} w(\mf{p}) \ll_{C} \frac{X (\log \log X)^{2}}{(\log X)^2} .\] 
\end{proposition}

\section{Gowers norms and main proof framework}\label{section4}

We are now in a position to give the overview of the proof of Theorem \ref{thm:main} in more detail, reducing it to the proof of three main results: \cref{prop:typei-to-gowers,prop:typeii-to-gowers,prop:main-term-comp}.

\subsection{Reduction to polar coordinates}\label{polar-sec}

We begin with an exercise in partial summation, reducing \cref{thm:main} to the following slightly more technical-looking statement, which removes the von Mangoldt weight on $x^2 + ny^2$ and replaces the rectangular cutoffs of \cref{thm:main} with a norm cutoff $x^2 + ny^2 \le X$ and an angular frequency; this may be thought of as conversion to a kind of polar coordinates. This is essential in order to bring in the prime ideal theorem (\cref{main-analytic-input}).

Recall from \cref{sec:notation} the definitions of $\Lambda'$ and of $\chi_{\infty}^{(\ell)}$.

\begin{proposition}\label{prop:main}
Suppose that $n \equiv 0$ or $4 \imod 6$. Let $\ell\in \Z$, $|\ell| \le (\log X)^{10}$ be a frequency. Then we have that 
\[\sum_{x, y \in \Z: x^2 + ny^2 \le X}\chi_{\infty}^{(\ell)}(x + y \sqrt{-n})\Lambda'(x)\Lambda'(y)1_{x^2 + ny^2 \operatorname{prime}} = \frac{\pi \kappa_n 1_{\ell = 0} X}{\sqrt{n} \log X} + O\bigg(\frac{X(\log\log X)^{2}}{(\log X)^2}\bigg)\]
where $\kappa_n$ is the constant in \cref{kappa-def}.
\end{proposition}
\begin{remark}
The exponent $10$ is somewhat arbitrary, being chosen so that \cref{prop:main} suffices for the deduction of \cref{thm:main}. If desired it could be replaced by any fixed constant $A$, at the expense of requiring the $O()$ term to depend on $A$. 
\end{remark}

\begin{proof}[Proof of \cref{thm:main}, assuming \cref{prop:main}] 
Let $W \in C_0^{\infty}(\R^2)$ be as in the statement of \cref{thm:main}. Rescaling $N$ by a factor of $O_W(1)$ if necessary, we may assume that $W$ is supported on $B_{10}(0)$. It is convenient to write $\Vert W \Vert$ for the maximum $L^{\infty}$-norm of any partial derivative of $W$ of order at most $3$. Set $\eta := (\log N)^{-10}$ (say), and suppose that $\eta \le R \le 10$. Setting $X = (RN)^2$ in the conclusion of \cref{prop:main} gives
\begin{equation}\label{polar-conclu} 
(2\log N )\sum_{x, y \in \Z} \Lambda'(x) \Lambda'(y) 1_{x^2 + ny^2  \operatorname{prime}} \psi_{R,\ell}\bigg(\frac{x}{N}, \frac{y}{N}\bigg) = \kappa_n N^2 \bigg(\int_{\R^2} \psi_{R,\ell}\bigg) + O\bigg(\frac{N^2(\log\log N)^{2}}{\log N}\bigg),\end{equation}
where 
\[ \psi_{R,\ell} (r \cos 2\pi \theta, n^{-1/2}  r \sin 2\pi \theta) := 1_{r \le R} \cdot e(\ell \theta),\] and this is valid for $|\ell| \le (\log N)^{10}$. Note that here we have replaced the factor $\log RN$ which arises in \cref{prop:main} by $\log N$; the error in doing this can be absorbed into the $O( \cdot)$ term on the RHS of \cref{polar-conclu}. 

For $x^2 + ny^2 \ge \eta N^2$, we have
\begin{equation}\label{W-decomp-1} W\bigg(\frac{x}{N}, \frac{y}{N}\bigg) = -\int^{10}_{\eta} \sum_{\ell \in \Z} c(R,\ell) \psi_{R,\ell}\bigg(\frac{x}{N}, \frac{y}{N}\bigg) dR, \end{equation}
where the $c(\cdot, \cdot )$ are defined in terms of the Fourier expansion
\[ \frac{\partial}{\partial R} W(R \cos 2\pi \theta, n^{1/2} R \sin 2\pi \theta) = \sum_{\ell \in \Z} c(R,\ell) e(\ell \theta).\]
(To verify this, substitute $\frac{x}{N} = r \cos 2\pi \theta$, $\frac{y}{N} = n^{-1/2} r \sin 2\pi \theta$.)
We claim that 
\begin{equation}\label{large-m} 
\int^{10}_{\eta} \sum_{|\ell| > (\log N)^{10}} |c(R,\ell)| dR \ll (\log N)^{-10} \Vert W \Vert
\end{equation}
and that 
\begin{equation}\label{ell-1-c}
\int^{10}_{\eta} \sum_{\ell \in \Z} |c(R,\ell)| dR \ll  \Vert W \Vert .   
\end{equation}
To prove these statements, observe that
\begin{equation}\label{crm-bd} |c(R,\ell)| \ll |\ell|^{-2}\int^{1}_0 \Big| \frac{\partial^2}{\partial \theta^2}\frac{\partial}{\partial R} W(R \cos 2\pi \theta, n^{-1/2} R \sin 2\pi \theta)\Big| d\theta \ll |\ell|^{-2} \Vert  W \Vert \end{equation} for $\ell \neq 0$. Here, the first bound follows by integrating the definition of Fourier transform twice by parts, and the second from the chain rule.
We also have
\begin{equation}\label{cr0-bd} |c(R,0)| \ll \Vert W \Vert .\end{equation}
Using \cref{crm-bd,cr0-bd} immediately leads to \cref{large-m,ell-1-c}.

From \cref{W-decomp-1,large-m} it follows that
\begin{equation}\label{W-decomp-2} W\Big(\frac{x}{N}, \frac{y}{N}\Big) = -\int^{10}_{\eta} \sum_{|\ell| \le (\log N)^{10}} c(R,\ell) \psi_{R,\ell}\Big(\frac{x}{N}, \frac{y}{N}\Big) dR + O((\log N)^{-10} \Vert  W \Vert). \end{equation}
when $x^2 + ny^2 \ge \eta N^2$. Now consider
\begin{equation}\label{S-def} S := (2\log N)\sum_{x, y \in \Z} \Lambda'(x) \Lambda'(y) 1_{x^2 + ny^2  \operatorname{prime}} W\Big(\frac{x}{N}, \frac{y}{N}\Big).\end{equation}
The contribution to $S$ from $x^2 + ny^2 < \eta N^2$ is $\ll N^2(\log N)^{-9}\Vert W \Vert_{\infty}$.
By \cref{W-decomp-2}, the sum over the remaining range is
\begin{align*} -\int^{10}_{\eta} dR \sum_{|\ell| \le (\log N)^{10}} c(R,\ell) \sum_{\substack{x,y \in \Z \\ x^2 + ny^2 \ge \eta N^2}} \Lambda'(x) \Lambda'(y)1_{x^2 + ny^2  \operatorname{prime}} & \psi_{R,\ell}\Big(\frac{x}{N}, \frac{y}{N}\Big) \\ & + O(N^2 (\log N)^{-9} \Vert W \Vert). \end{align*}
We now remove the condition $x^2 + ny^2 \ge \eta N^2$ from the sums; the error in doing so is (crudely)
\[ \ll \eta N^2 (\log N)^2 \int^{10}_{\eta} dR\sum_{\ell \in \Z} |c(R,\ell)| \ll N^2 (\log N)^{-8} \Vert W \Vert, \] by \cref{ell-1-c}. 
Therefore
\begin{align*} S = - (2 \log N)\int^{10}_{\eta} dR \sum_{|\ell| \le (\log N)^{10}} c(R,\ell) \sum_{x, y \in \Z}\Lambda'(x) \Lambda'(y)1_{x^2 + ny^2  \operatorname{prime}}&  \psi_{R,\ell}\Big(\frac{x}{N}, \frac{y}{N}\Big)\\ & + O(N^2 (\log N)^{-8} \Vert W \Vert). \end{align*}
Applying \cref{polar-conclu} and then \cref{ell-1-c} gives
\begin{align*} S & = - \int^{10}_{\eta} dR \sum_{|\ell| \le (\log N)^{10}} c(R,\ell) \bigg(\kappa_n N^2 \int \psi_{R,\ell} + O\bigg( \frac{N^2(\log\log N)^{2}}{\log N} \bigg) \bigg) \\ & \qquad\qquad\qquad\qquad\qquad\qquad\qquad\qquad\qquad\qquad\qquad\qquad + O(N^2 (\log N)^{-8} \Vert W \Vert) \\ &  = - \int^{10}_{\eta} dR \sum_{|\ell| \le (\log N)^{10}} c(R,\ell) \bigg(\kappa_n N^2 \int_{\R^2} \psi_{R,\ell}  \bigg) + O\bigg(\frac{N^2(\log\log N)^{2}}{\log N} \Vert W \Vert\bigg) . \end{align*}
Finally, comparing this with the integral of \cref{W-decomp-2} we see that
\begin{equation}\label{main-preliminary}  S = \kappa_n N^2\bigg( \int_{\R^2} W \bigg) + O\bigg(\frac{N^2(\log\log N)^{2}}{\log N} \Vert  W \Vert\bigg)  .\end{equation}

Recall that $S$ is given by \cref{S-def}.
To deduce \cref{thm:main}, we add back in the contribution of prime powers to the two $\Lambda'$ terms and replace $(2\log N) 1_{x^2 + ny^2  \operatorname{prime}}$ with $\Lambda(x^2 + ny^2)$. The former replacement is trivial. The error in making the latter replacement is bounded by 
\[ \ll (\log N)^2\!\!\! \sum_{\substack{ x^2 + ny^2 \le N^2 \\ x, y, x^2 + ny^2 \operatorname{prime}}} \!\!\! W\Big(\frac{x}{N}, \frac{y}{N}\Big) \log \bigg( \frac{x^2 + ny^2}{N^2} \bigg) = -(\log N)^2\int^{N^2}_1 \frac{dt}{t} \!\!\! \sum_{\substack{ x^2 + ny^2 \le t  \\ x, y, x^2 + ny^2 \operatorname{prime}}} \!\!\! W\Big(\frac{x}{N}, \frac{y}{N}\Big). \]
The inner sum on the right is $\ll t(\log N)^{-3} \Vert W \Vert$ by a standard upper bound sieve. (For instance, one could use \cref{large-sieve-higher} with $k = 2$ and analyse as in \cref{sumhq}, noting that in this case one has $\alpha_p = \frac{4}{p} + O(\frac{1}{p^2})$ for $p$ with $\legendre{-n}{p} = 1$ and $\alpha_p = \frac{2}{p} + O(\frac{1}{p^2})$ for all $p$ with $\legendre{-n}{p} = -1$.) Thus we see that the error in passing from \cref{main-preliminary} to \cref{thm:main} is $O(N^2 \Vert W \Vert/\log N)$, which may be absorbed into the error term in \cref{thm:main}.
\end{proof}

\begin{remark}
One of the main ingredients of our argument is the prime ideal theorem \cref{main-analytic-input}. For ease of reference we have stated that result with a sharp cutoff $N \mf{a} \le X$. If one used a smooth cutoff instead then one could probably improve the dependence on the derivatives of $W$ in the above analysis.
\end{remark}

\subsection{Gowers norms}\label{gowers-norms-intro}

The Gowers norms, introduced by Gowers \cite{gowers-szemeredi}, are a fundamental notion in additive combinatorics. A discussion of their basic properties may be found in many places; see for instance the introductions to \cite{LSS,TT21}. We briefly recall the definitions. Let $f : \Z \rightarrow \C$ be a function with finite support. We first define, for $h \in \Q$, the difference operators
\begin{equation}\label{difference}\Delta_h f(x) = f(x)\overline{f(x+h)}.\end{equation} (It is convenient for technical reasons later on to allow $h \in \Q$, but of course $\Delta_h f$ will be zero unless $h \in \Z$). Now let $k \ge 2$ be a parameter. Then we first define
\begin{equation}\label{gowers-def} \Vert f \Vert^{2^k}_{U^k(\Z)} := \sum_{x, h_1, \dots, h_k \in \Z} \Delta_{h_1} \cdots\Delta_{h_k} f(x).\end{equation}
The Gowers $U^k(\Z)$-norm $\Vert f \Vert_{U^k(\Z)}$ is then defined to be the unique non-negative $2^k$-th root of this quantity. Given a real number $N \ge 1$, we then define 
\begin{equation}\label{gowers-norm-def-2} \Vert f \Vert_{U^k[N]} := \Vert f \Vert_{U^k(\Z)}/\Vert 1_{[N]} \Vert_{U^k(\Z)}.\end{equation}
This definition will be relevant when $f$ is supported on $[\pm O(N)]$. It is easy to see that this is completely equivalent to the definition given in, for instance, \cite[Definition 1.1]{LSS}, which is specifically geared to functions $f$ which are supported on $[N]$. The normalising factor $\Vert 1_{[N]} \Vert_{U^k(\Z)}$ has the effect that the maximum value of $\Vert f \Vert_{U^k[N]}$ is exactly 1 if $f$ is supported on $[N]$. For us, precise constants will never be relevant and it will suffice to know that 
\begin{equation}\label{normalisation-gowers-x} \Vert 1_{[N]} \Vert_{U^k(\Z)}^{2^k} \asymp_k N^{k+1},\end{equation} which is clear from the definition. That the Gowers norms are well-defined and are norms are well-known; proofs of this statement may be found in many places, for instance \cite[Appendix B.5]{green-tao-linear}. We have the nesting property
\begin{equation}\label{gowers-nesting} \Vert f \Vert_{U^2[N]} \ll \Vert f \Vert_{U^3[N]} \ll \cdots \ll \Vert f \Vert_{U^k[N]} \ll \cdots\end{equation}
Further properties of the Gowers norms relevant to our paper may be found in \cref{appendixA}. As a final remark, we note that all Gowers norms in the main part of our paper are at scale $X^{1/2}$, where $X$ is our main global parameter.

\subsection{Type I and II sums with product weights and Gowers norms}

In this section we state our two main results linking Gowers norms and Type I/II estimates for weight functions in product form. We begin with the Type I statement.

\begin{proposition}\label{prop:typei-to-gowers}
There is an absolute constant $C_{\ref{prop:typei-to-gowers}}$ with the following property. Let $\delta \in (0,1)$ be a parameter. Let $w : \Ideals(\O_K) \rightarrow \C$ be a weight function in product form \cref{product-form} with frequency $\ell$. Suppose that $f, f'$ are $1$-bounded and supported on $[\pm 2n X^{1/2}]$. Suppose $L \le \delta^{C_{\ref{prop:typei-to-gowers}}} X^{1/2}$ and that 
\begin{equation}\label{large-typei-assump} \sum_{\mf{d} : N\mf{d} \sim L} \Big| \sum_{\substack{\mf{d} | \mf{a} \\ N \mf{a} \le X}}  w(\mf{a})\Big| \ge \delta X.\end{equation}
Then \[ \Vert f \Vert_{U^3[X^{1/2}]}, \Vert f' \Vert_{U^3[X^{1/2}]} \gg \Big(\frac{\delta}{(|\ell|+1) \log X} \Big)^{O(1)} .\]
\end{proposition}
\begin{remark}
Jori Merikoski has indicated to us a rather different argument using the multiplicative large sieve, which allows one to show that in fact $f, f'$ correlate with linear phases $e(\theta x)$ with $\theta$ close to a rational with denominator $\ll (\log X/\delta)^{O(1)}$. We have opted to retain our original argument so as to give a unified treatment of the Type I and II sums. With some additional argument it may also be used to give the aforementioned stronger conclusion, as we will discuss in \cite{GS24}. 
\end{remark}
Now we give the Type II statement.

\begin{proposition}\label{prop:typeii-to-gowers}
There exists an absolute constant $C_{\ref{prop:typeii-to-gowers}}$ and a \textup{(}large\textup{)} positive integer $k$ such that the following holds. Let $w : \Ideals(\O_K) \rightarrow \C$ be a weight function in product form \cref{product-form} with frequency $\ell$, and suppose that $f, f'$ are $1$-bounded and supported on $[\pm 2n X^{1/2}]$. Suppose that $\delta^{-C_{\ref{prop:typeii-to-gowers}}} \ll L \le X^{3/8}$, and suppose that for some $1$-bounded weights $\alpha, \beta : \Ideals(\O_K) \rightarrow \C$ we have
\begin{equation}\label{typeii-assump-cont} \Big| \sum_{\substack{\mf{a}, \mf{b}: N\mf{a} \sim  L \\ N \mf{a}\mf{b} \le X}} \alpha_{\mf{a}} \beta_{\mf{b}} w(\mf{a} \mf{b}) \Big| \ge \delta X.\end{equation}
Then we have
\[ \Vert f \Vert_{U^k[X^{1/2}]}, \Vert f' \Vert_{U^k[X^{1/2}]} \ge \Big(\frac{\delta}{(|\ell|+1) \log X} \Big)^{O(1)} .\] 
\end{proposition}
\begin{remarks}
A trivial modification to the proof gives a similar statement in the range $L \le X^{1/2 - \kappa}$, with $k = O_{\kappa}(1)$. The value of $k$ is extremely large; for further comments see \cref{rmk:height}.
\end{remarks}

\subsection{Approximations to \texorpdfstring{$\Lambda'$}{} in Gowers norm}

Define $\Lambda_{\cramer}$ as in \cref{cramer-def}. Extend the domain of definition to $\Q$, with the convention that $\Lambda_{\cramer}(t) = 0$ if $t \notin \Z$. A crucial ingredient in our arguments is the fact that these functions are good approximants to $\Lambda'$ in the Gowers norms. 

\begin{proposition}\label{prop:primes-gowers-uniform}
Fix $A\ge 1$ and $k\ge 2$.  Then we have that 
\[\snorm{(\Lambda' - \Lambda_{\cramer})1_{[X^{1/2}]}}_{U^k[X^{1/2}]}^{2^k}\ll_{A} (\log X)^{-A}.\]
\end{proposition}
\begin{proof}
We will apply the results of \cite{TT21} and \cite[Section 5]{Len23b}, taking $N = X^{1/2}$ in those papers. Note (cf. the comments following \cref{q-choice}) that our value of $Q$ is precisely the same as the one in those papers. Let $\Lambda_{\siegel}$ be the Siegel model for the primes, introduced in \cite[Definition 2.1]{TT21}; the definition in \cite[Section 5]{Len23b} is the same. Let $q_{\siegel}$ be the modulus of the character for which there is a $Q$-Siegel zero $\beta$, if one exists (again, see \cite[Definition 2.1]{TT21} for the definition).

Then by \cite[Theorem 7]{Len23b} (and ignoring the contribution of the prime powers to $\Lambda$, which is negligible) we have
\begin{equation}\label{primes-to-siegel} \Vert (\Lambda' - \Lambda_{\siegel})1_{[X^{1/2}]} \Vert_{U^k[X^{1/2}]} \ll e^{-(\log X)^{c_k}}\end{equation} for some $c_k > 0$.
On the other hand, by \cite[Theorem 2.5]{TT21} we have 
\begin{equation}\label{cramer-siegel} \Vert ( \Lambda_{\siegel} - \Lambda_{\cramer, Q})1_{[X^{1/2}]} \Vert_{U^k[X^{1/2}]} \ll q_{\siegel}^{-c_k}  \end{equation} for some $c_k > 0$. Note here that the exponent $c$ in \cite[Theorem 2.5]{TT21} does depend on $k$, as per the convention laid out on the first page of that paper. If there is no Siegel zero, the LHS of \cref{cramer-siegel} is zero. By Siegel's theorem we have the (ineffective) bound
\begin{equation}\label{sig-bd} q_{\siegel} \gg_A (\log X)^A.\end{equation}
(In more detail, \cite[Theorem~5.28 (2)]{IK-book} with $\eps = 1/10A$ implies that $1 - \beta \gg_A q_{\siegel}^{-1/10A}$; on the other hand, by the definition \cite[Definition 2.1]{TT21} we have $1 - \beta \ll \frac{1}{\log Q} \sim (\log X)^{-1/10}$.)

The result follows from this and \cref{primes-to-siegel,cramer-siegel} by the triangle inequality for the Gowers norms.
\end{proof}
\begin{remark}
This is the point in the paper where the quasi-polynomial bounds for the inverse theorem for the Gowers norms \cite{LSS} has been applied, this being the major ingredient in the proof of \cref{primes-to-siegel}.
\end{remark}

\subsection{Evaluating the main term}

The following will be the main result of \cref{section8}.

\begin{proposition}\label{prop:main-term-comp} Let $\ell \in \Z$ and suppose that $|\ell| \le e^{\sqrt{\log X}}$. Then we have
\begin{align*}(\log X) \sum_{x,y \in \Z : x^2 + ny^2\le X}
\chi_{\infty}^{(\ell)}(x + y \sqrt{-n})\Lambda_{\cramer}(x)& \Lambda_{\cramer}(y) 1_{x^2 + ny^2 \operatorname{prime}} \\ & = \frac{\pi \kappa_n 1_{\ell = 0}}{\sqrt{n}} X + O( X (\log X)^{-1}),\end{align*}
where $\kappa_n$ is the constant in \cref{kappa-def} and $\chi_{\infty}^{(\ell)}(z) = (z/|z|)^{\ell}$. 
\end{proposition}

\subsection{Proof of the main theorem}

Recall our introductory discussion in \cref{subsec:outline}, and in particular the assertion \cref{general-sum-2} that our main theorem in the case $n = 4$ could be considered in terms of a sum over Gaussian primes. We begin by formulating such a statement for general $n$.  

\begin{lemma}\label{lem:sums-to-ideals}
Let $\ell \in \Z$, and let $K = \Q(\sqrt{-n})$. Then we have
\begin{equation}\label{main-sum-wp} \sum_{x^2 + ny^2 \le X}\chi_{\infty}^{(\ell)}(x + y \sqrt{-n}) \Lambda'(x) \Lambda'(y) 1_{x^2 + ny^2 \operatorname{prime}} = \sum_{\mf{p} : N\mf{p} \le X} w(\mf{p}) + O(X^{1/2 + o(1)}),\end{equation} where $\mf{p}$ runs over prime ideals in $\O_K$ and 
\begin{equation}\label{w-von-m}  w = \Lambda' \boxtimes_{\ell} \Lambda' . \end{equation} \end{lemma}
\begin{proof}
Recalling the definition of $\boxtimes_{\ell}$ (see \cref{product-form}) we see that the sum on the right of \cref{main-sum-wp} is the sum of $\chi_{\infty}^{(\ell)}(x + y \sqrt{-n})\Lambda'(x)\Lambda'(y)$ over $(x,y) \in \Z^2 \setminus \{0\}$ for which $(x + y \sqrt{-n})$ is a prime ideal.

Now the prime ideals $\mf{p}$ in $\O_K$ are of two types: those which sit above a rational prime $p$ which splits (or is ramified), in which case $N\mf{p} = p$, or those which are equal to $(p)$ for some rational prime $p$ which does not split, in which case $N\mf{p} = p^2$. The contribution of the non-split terms to the LHS of \cref{main-sum-wp} is $\ll (\log X)^2 \sum_{p \le X^{1/2}}\sum_{x,y} 1_{x^2 + ny^2 = p^2} \ll X^{1/2 + o(1)}$, where the last bound follows from \cref{repdivisor} and the divisor bound. The contribution of the non-split terms to the RHS of \cref{main-sum-wp} is at most $(\log X)^2$ times the number of ideals with norm $p^2 \le X$ for some $p$, which is again $\ll X^{1/2 + o(1)}$.
\end{proof}

We may now prove the main theorem itself, assuming \cref{prop:typei-to-gowers,prop:typeii-to-gowers,prop:main-term-comp}.

\begin{proof}[Proof of \cref{prop:main}]
We start with the LHS of the expression in \cref{prop:main}, which we expressed as a sum over $w(\mf{p})$ in \cref{main-sum-wp}. Note that $w(\mf{a})$ as given in \cref{w-von-m} is in product form, which brings to mind \cref{prop:main-sec3}. However, it would not be a good idea to try and apply \cref{prop:main-sec3} immediately, since $\tilde w(\mf{a}) = w(\mf{a}) 1_{N\mf{a} \le X}$ does not satisfy Type I estimates (it is irregularly distributed mod 3, for instance). To get around this latter issue, we introduce a decomposition using the approximant $\Lambda_{\cramer}$, defined in \cref{cramer-def}. Thus we write
\begin{equation}\label{w-decomp} w = \sum_{j=1}^3 w_j, \qquad w_j = f_j \boxtimes_{\ell} f'_j \end{equation}
where 
\begin{equation}\label{fj-choices} f_1 = \Lambda', \quad f'_1 = f_2 = \Lambda' - \Lambda_{\cramer}, \quad  \mbox{and} \quad f'_2 = f_3 = f'_3 = \Lambda_{\cramer}.\end{equation}

A trivial change to the proof of \cref{lem:sums-to-ideals} (replacing $\Lambda'$ by $\Lambda_{\cramer}$ throughout) shows that 
\[ \sum_{\mf{p} : N\mf{p} \le X} w_3(\mf{p}) = \sum_{x^2 + ny^2 \le X}\chi_{\infty}^{(\ell)}(x + y \sqrt{-n})\Lambda_{\cramer}(x)\Lambda_{\cramer}(y)1_{x^2 + ny^2 \operatorname{prime}}  + O(X^{1/2 + o(1)}).\]

In view of \cref{prop:main-term-comp}, it therefore suffices to show that 
\begin{equation}\label{suff-w1w2} \sum_{\mf{p} : N\mf{p} \le X} w_1(\mf{p}), \sum_{\mf{p}: N\mf{p} \le X} w_2 (\mf{p}) \ll X \frac{(\log \log X)^{2}}{(\log X)^2}.\end{equation}
For this, we use \cref{prop:main-sec3}. Note that we do have the required boundedness condition $|f_j(x)|, |f'_j(x)| \le (\Lambda_{\cramer} + \Lambda')(x)$ needed for that result. Since the concern in that result is with $\tilde w_j(\mf{a}) = w_j(\mf{a}) 1_{N\mf{a} \le X}$, we lose nothing by assuming that $f_j, f'_j$ are supported on $[\pm 2n X^{1/2}]$.

The Type I/II information that we need for the $\tilde w_j$ comes from (the contrapositive of) \cref{prop:typei-to-gowers,prop:typeii-to-gowers} respectively, together with \cref{prop:primes-gowers-uniform}, noting that for $j = 1,2$, at least one of the functions $f_j, f'_j$ is $\Lambda' - \Lambda_{\cramer, \qq}$. In applying \cref{prop:typei-to-gowers,prop:typeii-to-gowers}, one should take $\delta = (\log X)^{-12}$; the Gowers norm conclusions of those propositions are then contradicted by \cref{prop:primes-gowers-uniform} for sufficiently large $A$. (In particular, we have that we may apply \cref{prop:main-sec3} with $C = 12\max(C_{\ref{prop:typei-to-gowers}},C_{\ref{prop:typeii-to-gowers}}) + 1$.)
Note that, although \cref{prop:typei-to-gowers,prop:typeii-to-gowers} were stated only for $1$-bounded functions, they apply equally well to functions bounded by $\log X$, simply by applying the results with $(\log X)^{-1} f$ and $(\log X)^{-1} f'$ and absorbing the extra factors of $\log X$ in the $(\log X)^{O(1)}$ part of those bounds.

This concludes the proof of \cref{prop:main}, assuming \cref{prop:typei-to-gowers,prop:typeii-to-gowers,prop:main-term-comp}.
\end{proof}

\section{Preliminaries on concatenation and Gowers--Peluse norms}\label{section5}

In this section we develop some preliminary material connected with concatenation theorems for Gowers norms. When dealing with concatenation, it is useful to introduce generalisations of Gowers norms first considered by Peluse \cite[Definition 2.1]{Pel20}. Peluse calls these norms the \emph{Gowers box norms}, but we prefer to see the term `box norm' reserved for purely combinatorial graph- and hypergraph- norms and so will instead refer to these norms as Gowers--Peluse norms. 

We first state the definition of the Gowers--Peluse norm. 

\begin{definition}\label{def:box}
Given a function $f: \Z \to \C$ with finite support and $h,h'\in \Q$, we define 
\begin{equation}\label{double-difference} \Delta_{(h,h')}f(x) = f(x+h)\overline{f(x+h')}.\end{equation}
Given a multiset $\Omega = \{ \mu_1,\dots, \mu_k\}$ of probability measures on $\Q$ and a positive integer scale $N \ge 1$, we define 
\begin{equation}\label{gp-explicit-def} \snorm{f}_{U_{\GP}[N; \Omega]}^{2^k} = \snorm{f}_{U_{\GP}[N; \mu_1,\dots, \mu_k]}^{2^k} := \frac{1}{N}\sum_{x \in \Z} \E_{h_i,h_i' \sim \mu_i}\Delta_{(h_1,h_1')}\cdots\Delta_{(h_k,h_k')}f(x).\end{equation}
(As shown in \cref{lem:GCS}, the RHS is a non-negative real number, and so $\snorm{f}_{U_{\GP}[N; \Omega]}$ is uniquely defined as a nonnegative real number.)
\end{definition}

\begin{remark}
 The only difference between our definition and that of Peluse is the fact that we work with general probability measures rather than uniform measures on sets as Peluse did. (The two settings are equivalent, at least if one passes from sets to multisets as in \cite{KKL24}, with the multiset situation being a special case of the measure one, and the measure one a limiting case of the multiset one.) 

The definition of the difference operators $\Delta_{(h,h')}$ should be compared with \cref{difference}, of which it is a mild variant. Once again, for minor technical reasons it is convenient to allow $h,h' \in \Q$, but $\Delta_{(h,h')} f$ will be identically zero unless $h' - h \in \Z$. Note that $\Delta_{(h,h')}$, $\Delta_{(k,k')}$ commute. Later on it will be convenient to adopt the shorthand $\Delta_{M(h,h')}f := \Delta_{(Mh, Mh')}f$.

If $f$ is supported on $[\pm N]$ and if the $\mu_i$ are all uniform on $[\pm N]$ then $\Vert f \Vert_{U_{\GP}[N; \mu_1,\dots, \mu_k]}$ is a kind of `smoothed' Gowers norm which is comparable to, but not exactly the same as, $\Vert f \Vert_{U^{k}[N]}$. We will not actually need this fact in the paper: for equivalent results see \cite[Lemma~C.3]{PSS23} or \cite[Lemma~A.3]{KKL24b}.
\end{remark}

We next record a certain general--purpose lemma and corollary which will be used in order to convert Gowers--Peluse norms whose underlying measures are somewhat close to uniform on $[\pm N]$ into genuine Gowers norms, albeit of a higher order.

\begin{lemma}\label{lem:replac-diff}
Let $\delta\in (0,1)$, $N, T\ge 1$ and consider a $1$-bounded function $f:\Z\to \C$, supported on $[\pm N]$. Furthermore let $\mu$ be a probability measure supported on $[\pm N]$ which is somewhat uniform in the sense that  $\Vert \mu \Vert_2^2 \le T/N$. Suppose that $\snorm{f}_{U_{\GP}[N;\mu]}^2\ge \delta$. Then $\snorm{f}_{U^2[N]}\gg (\delta^2/T)^{1/4}$.
\end{lemma}
\begin{proof}
The condition $\snorm{f}_{U_{\GP}[N;\mu]}^2\ge \delta$ is, written out,
\[\sum_{x\in \Z}\E_{h_1,h_1'\sim \mu}f(x+h_1)\ol{f(x+h_1')} = \sum_{x\in \Z}\E_{h_1,h_1'\sim \mu}f(x)\ol{f(x+h_1 - h_1')}\ge \delta N .\]
This implies that 
\[ \sum_{t\in \Z} \E_{h_1 , h'_1 \sim \mu} 1_{h_1 - h'_1 = t} \Big|\sum_{x\in \Z}f(x)\ol{f(x+t)}\Big|\ge \delta N.\]
Removing the averaging over $h'_1$, we see that there is some $h'_1$ such that 
\[\sum_{t\in \Z}\E_{h_1 \sim \mu}1_{h_1 - h_1' = t} \Big|\sum_{x\in \Z}f(x)\ol{f(x+t)}\Big|\ge \delta N, \] or in other words
\[ \sum_{t \in \Z} \mu(h'_1 + t) \Big|\sum_{x\in \Z}f(x)\ol{f(x+t)}\Big|\ge \delta N. \]
Cauchy--Schwarz and the assumption of the lemma give
\[\sum_{t\in \Z}\Big|\sum_{x\in \Z}f(x)\ol{f(x+t)}\Big|^2\ge \frac{(\delta N)^2}{T/N} = \delta^2 N^3/T.\]
Expanding the left-hand side gives
\[\sum_{x,x',t\in \Z}f(x)\ol{f(x+t)}\ol{f(x')}f(x'+t) = \sum_{x,t,t'\in \Z}\Delta_{t}\Delta_{t'}f(x)\ge \delta^2 N^3/T.\] However, the left-hand side is $\gg N^3$ times $\Vert f \Vert^4_{U^2[N]}$.
\end{proof}

The following corollary details how the above may be applied iteratively in order to replace a Gowers--Peluse norm with nearly uniform measures by a geunine Gowers norm (of twice the order).

\begin{corollary}\label{cor53x}
Let $\delta \in (0,1)$, let $N, T \ge 1$ be parameters with $N$ an integer, and let $f : \Z \rightarrow \C$ be a $1$-bounded function supposed on $[\pm N]$. Let $\mu_1,\dots, \mu_k$ be probability measures supported on $[\pm N]$, all of which are somewhat uniform in the sense that $\Vert \mu_i \Vert_2^2 \le T/N$. Suppose that $\Vert f \Vert_{U_{\GP}[N;\mu_1,\dots, \mu_k]} \ge \delta$. Then $\Vert f \Vert_{U^{2k}[N]} \gg (T^{-k} \delta^{2^k})^{1/2^{2k}}$.
\end{corollary}
\begin{proof} Let $\mu$ be the uniform measure on $[\pm N]$.
We will prove by downward induction on $j = k, k - 1,\dots,0$ that, for any $j$, 

\begin{equation}\label{inductive-nearly-uniform} \frac{1}{N} \sum_{x \in \Z} \E_{\substack{h_i, h'_i \sim \mu_i \\ 1 \le i \le j}} \E_{\substack{h_i, h'_i \sim \mu \\ j+1 \le i \le k}} \Delta_{(h_1, h'_1)} \cdots \Delta_{(h_{j}, h'_{j})} \Delta_{h_{j+1}} \Delta_{h'_{j+1}} \cdots \Delta_{h_{k}} \Delta_{h'_{k}} f (x) \ge  T^{j - k} \delta^{2^k} .\end{equation}
The assumption is  the case $j = k$.
First note that this may be rewritten as 
\[\E_{\substack{h_i, h'_i \sim \mu_i \\ 1 \le i \le j-1}} \E_{\substack{h_i, h'_i \sim \mu \\ j+1 \le i \le k}}\Vert \Delta_{(h_1, h'_1)} \cdots \Delta_{(h_{j-1}, h'_{j-1})} \Delta_{h_{j+1}} \Delta_{h'_{j+1}} \cdots \Delta_{h_{k}} \Delta_{h'_{k}} f \Vert^2_{U_{\GP}[N;\mu_j]} \ge  T^{j - k} \delta^{2^k} .\]

Write $\delta(h)$ (where $h = (h_1, h'_1,\dots, h_{j-1}, h'_{j-1}, h_{j+1}, h'_{j+1},\dots, h_k, h'_k)$) for the size of the inner $U_{\GP}[N; \mu_j]$-norm squared. Then, applying \cref{lem:replac-diff}, we obtain using Cauchy--Schwarz
\begin{align*} \E_{\substack{h_i, h'_i \sim \mu_i \\ 1 \le i \le j-1}} \E_{\substack{h_i, h'_i \sim \mu \\ j+1 \le i \le k}}\Vert \Delta_{(h_1, h'_1)} \cdots \Delta_{(h_{j-1}, h'_{j-1})} \Delta_{h_{j+1}} \Delta_{h'_{j+1}} & \cdots \Delta_{h_{k}} \Delta_{h'_{k}} f \Vert^4_{U_2[N]} \\ & \ge T^{-1} \E \delta(h)^2 \ge T^{j - 1- k} \delta^{2^k}.\end{align*} Writing in the definition of the $U^2$-norm using dummy variables $h_j, h'_j$, this is
\[ \frac{1}{N}\sum_{x \in \Z}\E_{\substack{h_i, h'_i \sim \mu_i \\ 1 \le i \le j-1}} \E_{\substack{h_i, h'_i \sim \mu \\ j \le i \le k}}\Delta_{(h_1, h'_1)} \cdots \Delta_{(h_{j-1}, h'_{j-1})} \Delta_{h_{j}} \Delta_{h'_{j}} \cdots \Delta_{h_{k}} \Delta_{h'_{k}} f (x) \ge T^{-1} \E \delta(h)^2 \ge T^{j-1-k} \delta^{2^k}.\]
This in turn is exactly the statement \cref{inductive-nearly-uniform} in the case $j -1$, so the induction goes through.

The case $j = 0$ of \cref{inductive-nearly-uniform} is the statement we want.\end{proof}

Another key corollary of \cref{lem:replac-diff} is that if $a$ and $b$ are coprime then an average over $ax+by$ being large can be converted to $U^2$-control. 

\begin{lemma}\label{lem:U^2-control}
Let $\eta \in (0,1)$ be sufficiently small and let $M = O(1)$ be a positive integer. Let $a,b\in \frac{1}{M}\Z \setminus \{0\}$ satisfy $|a| \ge \eta |b|$, $|b| \ge \eta |a|$ and $d := \gcd(Ma,Mb) \le 1/\eta$. Set $Q := \max(|a|, |b|)$, and suppose that $N \ge Q^2/\eta^3$. Let $I_1, I_2 \subseteq[\pm N/\eta Q]$ be subintervals of $\Z$. Suppose that $f:\Q \to \C$ is $1$-bounded and supported on $\Z \cap [\pm N]$ and that 
\[\Big|\sum_{\substack{x,y\in \Z\\x\in I_1, y\in I_2}}f(ax + by)\Big|\ge \eta N^2/Q^2.\]
Then $\snorm{f}_{U^2[N]}^4\gg \eta^{O(1)}$.
\end{lemma}
\begin{proof} We generally do not indicate dependence on $M$ in implied constants.
Write $I_1 := \ell_1 + [L_1]$ and $I_2 := \ell_2 + [L_2]$, where $L_i := |I_i|$ (and $[L] = \{1,\dots, L\}$). Set $t = a \ell_1 + b\ell_2$ and write $a' = Ma/d$ and $b' = Mb/d$. Then the assumption may be written as
\begin{equation}\label{eq380} \Big|\sum_{\ell \in  \frac{d}{M}(a'[L_1] + b'[L_2])}f(t + \ell)\Big|\ge \eta N^2/Q^2.\end{equation} (Note here that the sum is over the \emph{multi}set $\frac{d}{M}(a'[L_1] + b'[L_2])$.)
Let $H := \lfloor\eta^{3}N/Q\rfloor$. Note that $H > 1$ by the assumption on $N$. 
Suppose that $\ell' \in \frac{d}{M}(a'[H] + b'[H])$. Then 
\[ \Big|\sum_{\ell \in  \frac{d}{M}(a' [L_1] + b' [L_2])}f(t + \ell + \ell') - \sum_{\ell \in  \frac{d}{M}(a'[L_1] + b'[L_2])}f(t + \ell)\Big| \ll \max(L_1, L_2) H,\] since the two ranges of summation differ in only $O(\max(L_1, L_2) H)$ terms and $f$ is $1$-bounded. Averaging over $\ell'$, it follows from this and \cref{eq380} that 

\[\frac{1}{H^2} \Big| \sum_{\ell \in  \frac{d}{M}(a' [L_1] + b'[L_2])} \sum_{\ell'\in  \frac{d}{M}(a' [H] + b'[H])}f(t + \ell + \ell')\Big|\ge \eta N^2/Q^2 - O(\max(L_1, L_2) H) \gg \eta N^2/Q^2\] assuming $\eta$ small enough, using here that $\max(L_1, L_2) \ll N/\eta Q$.
By the triangle inequality, and multiplying through by $H^2$, 
\begin{equation}\label{eq381}\sum_{\ell \in  \frac{d}{M}(a' [L_1] + b'[L_2])}\Big|\sum_{\ell'\in  \frac{d}{M}(a' [H] + b'[H])}f(t + \ell + \ell')\Big|\gg \eta^{7}N^4/Q^4.\end{equation} 
At this point we note for future reference that the assumption of the lemma and the fact that $I_1, I_2 \subseteq [\pm N/\eta Q]$ imply that $L_1, L_2 \gg \eta^2 N/Q$. Therefore, since we are assuming $N \ge Q^2/\eta^3$, we have $L_1, L_2 \ge Q$.

Now the maximum number of ways to write any $u \in \Z$ as $a'x_1 + b'x_2$ with $x_1 \in [L_1]$, $x_2 \in [L_2]$ is $\ll 1 + L_1/Q \ll L_1/Q$, since if there is one solution then any other one is given by $x'_1 = x_1 + \lambda b'$, $x'_2 = x_2 - \lambda a'$, using here that $\gcd(a', b') = 1$. It follows from this observation and \cref{eq381} that
\[\sum_{x \in t + \Z }\Big|\sum_{\ell'\in  \frac{d}{M}(a'[H] + b'[H])}f(x + \ell')\Big|\gg 
\frac{\eta^{7}N^4}{Q^4} \cdot \frac{Q}{L_1} \gg \frac{\eta^8 N^3}{Q^2}.\]
Note here that since $|da' H| = |aH| \le QH < \eta^3 N$, and similarly for $db'H$, and since $|t| < 2N/\eta$, and since $\on{supp}(f) \subseteq[\pm N]$, we may restrict $x$ to the range $\{ x \in t + \Z : |x| \ll N/\eta\}$. Then by Cauchy--Schwarz we have
\[\sum_{x\in t + \Z}\sum_{\ell_1',\ell_2'\in \frac{d}{M}(a'[H] + b'[H])}f( x + \ell_1')\overline{f( x + \ell_2')}\gg \eta^{17}N^5/Q^4,\] or equivalently
\begin{equation} \label{eq344} \Vert f \Vert^2_{U_{\GP}[N; \mu]} \gg \eta^{17/4},\end{equation}
where $\mu$ is the uniform measure on the multiset $\big( \frac{d}{M}(a'[H] + b'[H]) - t\big)\cap \Z$. Since $t = \frac{d}{M} (a'\ell_1 + b'\ell_2)$, this multiset has size at least $\gg H^2/M$.

To apply \cref{lem:replac-diff} we need an upper bound on $\Vert \mu \Vert_2^2$. Similarly to the remark above, for any $u \in \Z$ the number of ways to write $u = a'x_1 + b'x_2$ with $x_1, x_2 \in [H]$ is $\ll 1 + H/Q \ll H/Q$, and so 
\[ \Vert \mu\Vert_2^2 = \sum_u \mu(u)^2 \le \sup_{u \in \Z} \mu(u) \ll \frac{H}{Q} \cdot \frac{M}{H^2} \ll 1/\eta^3 N.\] 

Applying \cref{lem:replac-diff} (and recalling \cref{eq344}) then gives $\Vert f \Vert_{U^2[N]} \gg \eta^3$, as required.
\end{proof}

The key ingredient in our Type II estimate is the following result of Kravitz, Kuca, and Leng \cite[Corollary~6.4]{KKL24}. We state it now, and for the sake completeness (and because some of our notation is a little different) we reproduce a proof of \cref{lem:concat-main} in \cref{appendixB}. 

\begin{lemma}\label{lem:concat-main}
Fix $\delta\in (0,1/2)$, $s\ge 1$, $I$ an indexing set and $m$ a power of $2$. Suppose that for $i\in I$ and $j\in [s]$, we have symmetric probability measures $\mu_{ji}$ supported on $[\pm N]$. Let $f:\Z\to \C$ be $1$-bounded and such that 
\[\E_{i\in I}\snorm{f}^{2^s}_{U_{\GP}[N;\mu_{1i},\ldots,\mu_{si}]}\ge \delta.\]
Then there exists $t = t_0(m, s)$ such that 
\[\E_{i_1,\ldots,i_t\in I}\snorm{f}_{ U_{\GP}[N; \Omega]}
^{2^d}\ge \delta^{O_{m,s}(1)},\]
where $\Omega$ is the collection of probability measures
\[ \Omega = (\mu_{ji_{k_1}} \ast \cdots \ast \mu_{ji_{k_m}})_{j\in [s], 1\le k_1<k_2<\cdots<k_{m}\le t}\]
where $d = s\cdot \binom{t}{m}$.
\end{lemma}

\section{Type I up to \texorpdfstring{$X^{1/2}/(\log X)^{A}$}{}}\label{section6}

In this section we prove \cref{prop:typei-to-gowers}. We remark that the crucial manipulation, which occurs at the end of the proof, is essentially the base case for the concatenation proven by Peluse \cite{Pel20} (see \cite[Lemma~5.4]{Pel20}).

\begin{proof}
Replacing $\delta$ by $\min(\delta, |\ell|^{-1})$, we may assume throughout that $|\ell| \le 1/\delta$.

We start with the assumption \cref{large-typei-assump}, that is to say 
\begin{equation}\label{assump-rpt} \sum_{\mf{d} : N\mf{d} \sim L} \Big| \sum_{\substack{\mf{d} | \mf{a} \\ N\mf{a} \le X}} w(\mf{a})\Big| \ge \delta X.\end{equation}

We are working under the assumption that $w = f \boxtimes_{\ell} f'$ (see \cref{def-product-form}) and that $f, f'$ are supported on $\Z \cap [\pm 2n X^{1/2}]$. The aim is to prove that both $f$ and $f'$ have large Gowers $U^3[X^{1/2}]$-norm. We prove the bound for $f$; the proof of $f'$ is symmetric.

The initial stages of the argument consist in manipulating things so that we can work with the fact that $w$ is in product form.  First, split the sum over $\mf{d}$ according to which ideal class $\mf{d}$ belongs to. For some such class, the sum is $\gg \delta X$. Let $g \in C(K)$ be the inverse of this class, and let $\mf{c}_g,\mf{c}'_{g^{-1}}$ be as in \cref{lem:ideal-reps}. Since $w$ is supported on principal ideals, the inner sum over $\mf{a}$ is supported on ideals of form $\mf{a} = \mf{d}\mf{b}$, where $\mf{b} \in [g]$. By \cref{lem:ideal-reps}, we may parametrise these ideals by $\mf{b} = (z_2) \mf{c}_g$ where $z_2\in \Pi_g \setminus \{0\}$ and $\mf{d} = (z_1)\mf{c}'_{g^{-1}}$, where $z_1 \in \Pi'_{g^{-1}} \setminus \{0\}$, where $\Pi_g, \Pi'_{g^{-1}}$ are the sublattices of $\Z[\sqrt{-n}]$ described in \cref{lem:ideal-reps}. 

This parametrisation covers each $\mf{b}$ and $\mf{d}$ a total of $r^2$ times, since associate values of $z_1, z_2$ (i.e. differing up to a unit) give the same $\mf{b}$ and $\mf{d}$. Thus \cref{assump-rpt} gives
\begin{equation}\label{assump-new}
\sum_{\substack{z_1 \in \Pi'_{g^{-1}} \\ |z_1|^2 \sim \gamma_1 L}} \Big| \sum_{\substack{z_2 \in \Pi_g \setminus \{0\} \\ |z_1 z_2| \le X^{1/2}/\gamma} } w((\gamma z_1 z_2 ))\Big| \gg \delta X.
\end{equation}
Here, $\gamma_1 := (N\mf{c}'_{g^{-1}})^{-1}$ and $\gamma = \gamma_g$ is the generator of $\mf{c}_g \mf{c}'_{g^{-1}}$, a positive rational of height $O_n(1)$ (see \cref{lem:ideal-reps} (ii)). Using the fact that $w = f \boxtimes_{\ell} f'$ gives
\begin{equation}\label{636465}\sum_{\substack{z_1 \in \Pi'_{g^{-1}} \\ |z_1|^2 \sim \gamma_1 L}} \Big| \sum_{\substack{ z_2 \in \Pi_g \setminus \{0\} \\ |z_1 z_2| \le X^{1/2}/\gamma}} \Big(\frac{z_1 z_2}{| z_1 z_2|}\Big)^{\ell}f\big(\gamma \Re (z_1 z_2)\big) f'\big( n^{-1/2}\gamma  \Im ( z_1 z_2)\big) \Big| \gg \delta X. \end{equation} (Note here that all associates of $\gamma z_1 z_2$ are automatically included in the sum \cref{assump-new}, since the set $\Pi_g$ is invariant under multiplication by units.)
This may be written as 
\begin{equation}\label{eq528} \sum_{\substack{z_1 \in \Pi'_{g^{-1}} \\ |z_1|^2 \sim \gamma_1 L}} \Big| \sum_{\substack{z_2 \in \Pi_g \setminus \{0\} \\ |\Re(z_2)| \le  C(X/L)^{1/2} \\  n^{-1/2}|\Im(z_2)| \le C(X/L)^{1/2}}} \Psi_{\ell}\Big(\frac{\gamma z_1 z_2}{X^{1/2}}\Big) f\big(\gamma \Re (z_1 z_2)\big) f'\big( n^{-1/2}\gamma  \Im ( z_1 z_2)\big) \Big| \gg \delta X, \end{equation}
where 
\begin{equation}\label{r-sector} \Psi_{\ell}(z)  := \Big(\frac{z}{|z|}\Big)^{\ell} 1_{|z| \le 1}. \end{equation}
Note here that (for an appropriately large constant $C$) we can insert the two conditions $|\Re(z_2)| \le  C(X/L)^{1/2}$ and $n^{-1/2}|\Im(z_2)| \le C(X/L)^{1/2}$ without harm, since $f, f'$ are supported on $[\pm O(X^{1/2})]$ and so $|z_2| \ll (X/L)^{1/2}$ on the support of the sum. It will be useful to carry a condition of this type through the proof, and we have written it in a Cartesian form with later manipulations in mind.

The next step is to perform `cosmetic surgery' in the sense of Harman \cite{harman} to remove the cutoff involving $\Psi_{\ell}$, using Fourier analysis to replace this by products of phases in $\Re(z_1z_2)$ and $\Im(z_1 z_2)$. This allows us to fully decouple into functions of $\Re(z_1z_2)$ and $\Im(z_1z_2)$. 
\begin{lemma}\label{annulus-smooth} Let $\ell \in \Z$ and let $\eps > 0$. Then there is a smooth approximation $\tilde \Psi_{\ell} : \C \rightarrow \C$ with $\Vert \tilde \Psi_{\ell} \Vert_{\infty} \le 1$, such that $\tilde \Psi_{\ell}$ and $\Psi_{\ell}$ agree outside of the domain
\begin{equation}\label{d-eps-domain} \mc{D}_{\eps} := \{ z \in \C : |z| < \eps\} \cup \{ z \in \C : 1 - \eps < |z| < 1 + \eps\}\end{equation}
and such that for $z \in \C$ we have the Fourier expansion 
\begin{equation} \label{tilde1r-fourier}\tilde \Psi_{\ell}(z) = \int_{\R^2} W(\xi_1, \xi_2) e( \xi_1 \Re z +  \xi_2 \Im z) d \xi_1 d\xi_2\end{equation}
with 
\begin{equation}\label{ell-1-fourier-B} \Vert W \Vert_{L^1(\R^2)} \ll (1 + |\ell|)^4 \eps^{-4}.\end{equation}
\end{lemma}
\begin{proof}
Identify $\C$ with $\R^2$ in the usual way. Let $\psi \in C_0^{\infty}(\R^2)$ be supported on the unit ball and have $\int \psi = 1$, and set $\tilde\Psi_0 := (\Psi_0 \ast \psi(\frac{\cdot}{\eps})) (1 - \psi(\frac{\cdot}{\eps}))$. Since $\psi(\cdot / \eps)$ is supported on the ball of radius $\eps$, it follows that $\Psi_0$ and $\tilde \Psi_0$ agree outside of $\mc{D}_{\eps}$. By applications of the Leibniz rule, any $k$th derivative of $\tilde\Psi_0$ is bounded by $O_k(\eps^{-k})$ in $L^{\infty}(\R^2)$.

One can check that the $k$th derivatives of $z \mapsto (z/|z|)^{\ell}$ (considered as a function from $\R^2$ to $\C$) are bounded by $\ll_k (1 + |\ell|)^k \eps^{-k}$ in the domain $\mc{D}_{\eps}$. To see this note that $(z/|z|)^{\ell} = (z/\bar{z})^{\ell/2}$, and thus 
\[\Big|\frac{\partial^j}{\partial z^{j}}\frac{\partial^{j'}}{\partial \overline{z}^{j'}}(z/|z|)^{\ell}\Big| = \Big|z^{\ell/2-j}\ol{z}^{-\ell/2-j'}\prod_{i=0}^{j-1}\Big(\frac{\ell}{2} - i\Big)\prod_{i'=0}^{j'-1}\Big(\frac{\ell}{2} - i'\Big)\Big|\ll_{j,j''}(|\ell| + 1)^{j + j'}|z|^{-j - j'}.\]
As the standard coordinates on $\R^2$ are $(z+\ol{z})/2$ and $(z-\ol{z})/2$ the result follows easily. 

Now define $\tilde\Psi_{\ell}(z) := (z/|z|)^{\ell} \tilde\Psi_0(z)$. This has the desired support properties and, by the above derivative bounds and more applications of the Leibniz rule, has $k$th derivatives bounded by $\ll_k (1 + |\ell|)^k\eps^{-k}$. Since $\tilde \Psi_{\ell}$ is smooth it has a Fourier expansion \cref{tilde1r-fourier} with $W(\xi_1, \xi_2) = \widehat{\tilde\Psi_{\ell}}(\xi_1, \xi_2)$. Integrating the definition of Fourier transform twice by parts in each variable gives
\[ |W(\xi_1, \xi_2)| \ll |\xi_1|^{-2} |\xi_2|^{-2} \Vert \partial_x^2 \partial_y^2 \tilde\Psi_{\ell} \Vert_{\infty} \ll (1 + |\ell|)^4 \eps^{-4} |\xi_1|^{-2} |\xi_2|^{-2}.\]
Combining this with the trivial bound $|W(\xi_1, \xi_2)| \ll 1$ gives that $\Vert W \Vert_{L^1(\R^2)} \ll (1+ |\ell|)^4\eps^{-4}$, which is the desired bound \cref{ell-1-fourier-B}.
\end{proof}

Returning to the main line of argument, we replace $\Psi_{\ell}$ in \cref{eq528} with the smoothed approximant $\tilde \Psi_{\ell}$ from \cref{annulus-smooth} for an appropriate value of $\eps$. Since $f, f'$ are $1$-bounded, the error in doing this is bounded by
\[ \# \{ (z_1, z_2) \in \Pi'_{g^{-1}} \times \Pi_g : |z_1| \asymp L^{1/2}, \gamma z_1 z_2/X^{1/2} \in \mc{D}_{\eps}\}, \] where $\mc{D}_{\eps}$ is the domain \cref{d-eps-domain}.
Since the lattices $\Pi_g, \Pi'_{g^{-1}}$ are both contained in $\Z[\sqrt{-n}]$, one may see that this error is $\ll \eps X$. Taking $\eps = c \delta$ for an appropriately small constant $c \asymp_n 1$, this error may be absorbed by the RHS of \cref{eq528}, and so we indeed have 
\begin{equation}\label{eq528-b} \sum_{\substack{z_1 \in \Pi'_{g^{-1}} \\ |z_1|^2 \sim \gamma_1 L}} \Big| \sum_{\substack{z_2 \in \Pi_g \setminus \{0\} \\ |\Re(z_2)| \le  C(X/L)^{1/2} \\  n^{-1/2}|\Im(z_2)| \le C(X/L)^{1/2}}} \tilde\Psi_{\ell}\Big(\frac{\gamma z_1 z_2}{X^{1/2}}\Big) f\big(\gamma \Re (z_1 z_2)\big) f'\big( n^{-1/2}\gamma  \Im ( z_1 z_2)\big) \Big| \gg \delta X. \end{equation}
Applying the Fourier expansion \cref{tilde1r-fourier}, we now have
\[ \sum_{\substack{z_1 \in \Pi'_{g^{-1}} \\ |z_1|^2 \sim \gamma_1 L}} \Big| \int_{\R^2} W(\xi_1, \xi_2) \sum_{\substack{z_2 \in \Pi_g \setminus \{0\} \\ |\Re(z_2)| \le  C(X/L)^{1/2} \\  n^{-1/2}|\Im(z_2)| \le C(X/L)^{1/2}} }f_{\xi_1}(\gamma\Re(z_1 z_2)) f'_{\xi_2}(n^{-1/2}\gamma  \Im ( z_1 z_2) )d\xi_1d\xi_2 \Big|  \gg \delta X. \]
where 
\[ f_{\xi_1}(t) := f(t) e\Big(\frac{\xi_1 t}{X^{1/2}} \Big), \qquad f'_{\xi_2}(t) := f' (t) e\Big(\frac{n^{1/2}  \xi_2 t}{ X^{1/2}} \Big).\]
By \cref{ell-1-fourier-B} (and recalling that $\max(1,|\ell|) \le 1/\delta$), there is some choice of $\xi_1, \xi_2$ such that 
\begin{equation}\label{decoupled} \sum_{\substack{z_1 \in \Pi'_{g^{-1}} \\ |z_1|^2 \sim \gamma_1 L}} \Big|  \sum_{\substack{z_2 \in \Pi_g \setminus \{0\} \\ |\Re(z_2)| \le  C(X/L)^{1/2} \\  n^{-1/2}|\Im(z_2)| \le C(X/L)^{1/2}} } f_{\xi_1}(\gamma\Re(z_1 z_2)) f'_{\xi_2}( n^{-1/2}\gamma \Im ( z_1 z_2) ) \Big|  \gg \delta^9 X. \end{equation}
Now we pass to real and imaginary parts, writing $z_1 =a + b \sqrt{-n}$ and $z_2 = x + y \sqrt{-n}$ where $a,b,x,y \in \Z$ (since, by \cref{lem:ideal-reps}, $\Pi_g, \Pi'_{g^{-1}}$ are sublattices of $\Z[\sqrt{-n}]$). 
Set
\begin{equation}\label{gammas-def} \Gamma := \{(u,v) : u + v \sqrt{-n} \in \Pi'_{g^{-1}}\}, \qquad \Gamma' := \{(u,v) : u + v \sqrt{-n} \in \Pi_g\};\end{equation}
thus $\Gamma, \Gamma'$ are sublattices of $\Z^2$, of index $O_n(1)$.
Then \cref{decoupled} implies that 
\begin{equation}\label{eq4700} \sum_{ \substack{(a,b) \in \Gamma \\ |a|, |b| \ll L^{1/2}}} \Big|\sum_{\substack{(x,y) \in \Gamma' \\ |x|, |y| \le C(X/L)^{1/2}}} f_{\xi_1}(\gamma(ax - nby)) f'_{\xi_2}(\gamma(bx + ay))   \Big| \gg \delta^9 X.\end{equation}
Recall that our task is to show that $f$ and $f'$ have large Gowers $U^3$-norms. This norm is invariant under multiplication by linear phases, as can immediately be seen from the definition \cref{gowers-def} since the second and higher derivatives of any linear phase equal $1$. Therefore, to achieve our task we may replace $f, f'$ by $f e(\lambda \cdot)$ and $f' e(\lambda' \cdot)$ for any $\lambda, \lambda'$. In other words, we may assume henceforth that $\xi_1 = \xi_2 = 0$. Changing variables to $a_1 := \gamma a$, $b_1 := -n\gamma b$, we have

\begin{equation}\label{eq476}
\sum_{ \substack{a_1,b_1 \in \frac{1}{M}\Z \\ |a_1|, |b_1| \ll L^{1/2}}} \Big|\sum_{\substack{(x,y) \in \Gamma' \\ |x|, |y| \le  C(X/L)^{1/2}}} f(a_1 x + b_1 y) f'\Big(-\frac{b_1}{n} x + a_1 y\Big)   \Big| \gg \delta^9 X.
\end{equation}
Here, $M = O_n(1)$ is the denominator of $\gamma$; note that, since $\Gamma \subseteq\Z^2$, we do indeed have $a_1, b_1 \in \frac{1}{M} \Z$ for all terms in the sum \cref{eq4700}, and there is no problem with any extra terms being included since the summands are non-negative.

Now we restrict to the situation where $Ma_1, Mb_1$ are essentially coprime.
Note that as there are $\ll \sum_{d\ge D}L/d^2\ll L/D$ pairs with $\gcd(\ell,\ell')\ge D$, we can further restrict the sum over $a_1,b_1$ to those terms with $\gcd(Ma_1,Mb_1) \ll \delta^{-9}$. By a further pigeonholing we may restrict to some fixed value $\gcd(Ma_1, Mb_1) = d$, $1 \le d \ll \delta^{-9}$, such that we have
\begin{equation}\label{typei-written-new-2}
\sum_{\substack{a_1, b_1 \in \frac{1}{M}\Z \\ |a_1|, |b_1| \ll L^{1/2} \\ \gcd(Ma_1, Mb_1) = d}} \Big| \sum_{\substack{(x,y) \in \Gamma' \\ |x|, |y| \le C(X/L)^{1/2}}} f(a_1 x + b_1 y) f'\Big(-\frac{b_1}{n} x + a_1 y\Big)\Big| \gg \delta^{18} X.
\end{equation}

Let $t = -\frac{b_1}{n}x + a_1y$; via triangle inequality and the $1$-boundedness of $f'$ and the fact that $f'$ is supported on $[\pm O(X^{1/2})]$ we have that 
\begin{equation}\label{eq472}\sum_{\substack{a_1, b_1 \in \frac{1}{M}\Z \\ |a_1|, |b_1| \ll L^{1/2} \\ \gcd(Ma_1,Mb_1) = d}}\sum_{\substack{t\in \frac{1}{Mn} \Z\\ |t| \ll X^{1/2}}} \Big|\sum_{\substack{(x,y) \in \Gamma'  \\ |x|, |y| \le C(X/L)^{1/2} \\  -\frac{b_1}{n}x + a_1y = t}}f ( a_1x + b_1y)\Big|\gg \delta^{18} X.\end{equation}

Now it is time to deal with the constraint $(x, y) \in \Gamma'$. For this we use the orthogonality relation
\begin{equation}\label{gam-orth} 1_{\Gamma'}(x,y) = \E_{(\zeta, \zeta') \in (\Gamma')^{\perp}/\Z^2} e(\zeta x + \zeta' y),\end{equation} where $(\Gamma')^{\perp}$ denotes the dual lattice of $\Gamma'$. Substituting into \cref{eq472}, it follows that there is some choice of $\zeta, \zeta'$ such that 

\[\sum_{\substack{a_1, b_1 \in \frac{1}{M}\Z \\ |a_1|, |b_1| \ll L^{1/2} \\ \gcd(Ma_1,Mb_1) = d}}\sum_{\substack{t\in \frac{1}{Mn} \Z\\ |t| \ll X^{1/2}}} \Big|\sum_{\substack{x, y \in \Z  \\ |x|, |y| \le C(X/L)^{1/2} \\ -\frac{b_1}{n}x + a_1y = t}}f ( a_1x + b_1y) e(\zeta x + \zeta' y)\Big|\gg \delta^{18} X.\]

Now we apply Cauchy--Schwarz, which gives
\[\sum_{\substack{a_1,b_1 \in \frac{1}{M}\Z \\ |a_1|, |b_1| \ll L^{1/2}\\\gcd(Ma_1,Mb_1) = d}}\sum_{\substack{x,y,x',y' \in \Z\\ |x|, |x'|, |y|, |y'| \le C(X/L)^{1/2}  \\ b_1x - a_1yn = b_1x' - a_1y' n}}f(a_1x + b_1y)\overline{f(a_1x' + b_1y')} e(\zeta (x - x') + \zeta'(y - y')) \gg \delta^{36} X^{3/2}/L.\]
Note that $b_1x - a_1yn = b_1x' - a_1y' n$ and the coprimality of $Ma_1/d,Mb_1/d$ imply that $x'-x = (Ma_1/d)h$ and $y'-y = (Mb_1 /dn) h$ for some $h \in \Z$; therefore the above gives
\[\sum_{\substack{a_1,b_1 \in \frac{1}{M}\Z \\ |a_1|, |b_1| \ll L^{1/2} \\ \gcd(Ma_1,Mb_1) = d}} \sum_{\substack{x,y,h \in \Z \\ |x|, |y| \le C(X/L)^{1/2} \\  |h| \ll X^{1/2}/L }} e\Big(\Big(a_1 \zeta + \frac{b_1 \zeta'}{n}\Big)\frac{M h}{d} \Big)\Delta_{M(na_1^2 + b_1^2)h/dn}f(a_1x + b_1y) \gg \delta^{36} X^{3/2}/L.\]
By the triangle inequality we may eliminate the phase, obtaining 
\[\sum_{\substack{a_1,b_1 \in \frac{1}{M}\Z \\ |a_1|, |b_1| \ll L^{1/2} \\ \gcd(Ma_1,Mb_1) = d}} \sum_{\substack{h \in \Z \\ |h| \ll X^{1/2}/L}} \Big| \sum_{\substack{x,y \in \Z \\ |x|, |y| \le C(X/L)^{1/2} }} \Delta_{M(na_1^2 + b_1^2)h/dn}f(a_1x + b_1y)\Big| \gg \delta^{36} X^{3/2}/L.\]

It follows that there are $\gg \delta^{36} X^{1/2}$ triples $(a_1,b_1, h)$ with $|a_1|, |b_1| \ll L^{1/2}$ and $|h| \ll X^{1/2}/L$ and with $\gcd(Ma_1, Mb_1) = d \ll \delta^{-9}$ such that 
\[ \Big|\sum_{\substack{x,y \in \Z \\ |x|, |y| \le C(X/L)^{1/2} }}\Delta_{M(na_1^2 + b_1^2)h/dn}f(a_1x + b_1y)\Big| \gg \delta^{36} X/L.\]
By discarding exceptions, the same is true with the additional condition that $|a_1|, |b_1| \gg \delta^{36} L^{1/2}$.

Our aim now is to apply \cref{lem:U^2-control} with $\eta = \delta^{C_1}$ and $Q = C_2L^{1/2}$ for appropriate constants $C_1, C_2$, and with $N = 2n X^{1/2}$. If $C_1 > 12$ and if $C_2$ is large enough then all the conditions of that lemma are satisfied. (The condition $N \ge Q^2/\eta^3$ is satisfied due to the assumption $L \le \delta^{C_{\ref{prop:typei-to-gowers}}} X^{1/2}$ in the statement of \cref{prop:typei-to-gowers}, assuming $C_{\ref{prop:typei-to-gowers}}$ large enough.)

Applying \cref{lem:U^2-control}, we have for each such $a_1,b_1,h$ the lower bound 
\[\snorm{\Delta_{M(a_1^2n + b_1^2)h/dn}f}_{U^2[X^{1/2}]}\gg \delta^{O(1)}.\]
Squaring and summing over $a_1,b_1, h$ (dropping the condition $\gcd(Ma_1, Mb_1) = d$, which has now served its purpose, which we may by positivity) yields
\begin{equation}\label{ready-for-concat}\sum_{\substack{a_1, b_1 \in \frac{1}{M} \Z \\ \delta^{12} L^{1/2} \ll |a_1|, |b_1| \ll L^{1/2}}}\sum_{|h| \ll X^{1/2}/L}\snorm{\Delta_{M(na_1^2 + b_1^2)h/dn}f}_{U^2[X^{1/2}]}^{2}\gg \delta^{O(1)}X^{1/2}.\end{equation}
The idea now that that, as $a_1, b_1, h$ range over $a_1, b_1 \in \frac{1}{M} \Z$ with $\delta^{12} L^{1/2} \ll |a_1|, |b_1| \ll L^{1/2}$ and $h \in \Z$ with $|h| \ll X^{1/2}/L$, then $M(a_1^2n + b_1^2)h/dn$ ranges somewhat uniformly over its range. If it was \emph{exactly} uniform then the LHS would be essentially $X^{1/2}$ times the $U^3[O(X^{1/2})]$-norm of $f$. Making this heuristic rigorous is the first instance of ``concatenation'' in the paper.

The contribution from $h = 0$ is $\ll L$ and can be ignored by the assumption that $L \le  \delta^{C_{\ref{prop:typei-to-gowers}}}X^{1/2}$. Writing, for $t \in \Z$, 
\begin{align*} \sigma(t) := \#\Big\{ (a_1,b_1,h) \in \frac{1}{M} \Z \times \frac{1}{M} \Z \times \Z : & \delta^{12} L^{1/2} \ll |a_1|, |b_1| \ll L^{1/2},\\ &  0 < |h| \ll (X/L)^{1/2}, M(a_1^2n + b_1^2)h/dn = t\Big\},\end{align*} we can rewrite \cref{ready-for-concat} (with the $h = 0$ term removed) as 
\begin{equation}\label{sigma-version} \sum_{\substack{t \in \Z \\ 0 < |t| \ll \delta^{-O(1)} X^{1/2}}} \sigma(t) \snorm{\Delta_{t}f}_{U^2[X^{1/2}]}^{2} \gg \delta^{O(1)}X^{1/2}.
\end{equation}
(We can restrict to $t \in \Z$, since $\Delta_t f \equiv 0$ otherwise on account of $f$ being supported on $\Z$, and from the definition we see that $\sigma(0) = 0$.)  

Now if $(a_1,b_1,h)$ is one of the triples being counted in $\sigma(t)$, we have $h | tdn$, and hence $\sigma(t) \le \sum_{h | dnt} r(Mdnt/h)$, where $r(m)$ is the number of representations of $m$ as $u^2n + v^2$ for integer $u,v$ (which, for us, are $Ma_1$ and $Mb_1$). Since $r \ll \tau$ (see \cref{repdivisor}) it follows that $\sigma(t) \ll \sum_{h | dnt} \tau (Mdnt/h) \le \tau(Mdnt)^2 \ll \tau(t)^2$. By \cref{lem:divisor-moment} with $m = 4$ we have \[ \sum_{0 < |t| \ll \delta^{-O(1)} X^{1/2}} \sigma(t)^2 \ll \delta^{-O(1)} X^{1/2} (\log X)^{15}.\] Applying Cauchy--Schwarz to \cref{sigma-version} then gives 
\[ \sum_{\substack{t \in \Z \\ 0 < |t| \ll \delta^{-O(1)} X^{1/2}}}  \snorm{\Delta_{t}f}_{U^2[X^{1/2}]}^{4} \gg \delta^{O(1)} (\log X)^{-15}X^{1/2}.\]
Expanding the left-hand side using the definitions \cref{gowers-def,gowers-norm-def-2}, we obtain
\[ \Vert f \Vert_{U^3(\Z)}^8 = \sum_{t, x, h_1, h_2 \in \Z} \Delta_{h_1} \Delta_{h_2} \Delta_t f(x) \gg \delta^{O(1)} (\log X)^{-15} X^2.\]
Finally, \cref{prop:typei-to-gowers} follows from this and \cref{gowers-norm-def-2}. 
\end{proof}

\section{Type II estimates up to \texorpdfstring{$X^{1/2-o(1)}$}{}}\label{section7}

In this section we prove \cref{prop:typeii-to-gowers}. As in the proof of \cref{prop:typei-to-gowers}, we may assume that $|\ell| \le 1/\delta$ throughout. We handle the bound involving $f'$, the one with $f$ being essentially identical. We begin, of course, with the assumption \cref{typeii-assump-cont}, that is to say
\[\Big| \sum_{\substack{ \mf{a}, \mf{b}: N\mf{a} \sim  L \\ N\mf{a}\mf{b} \le X}} \alpha_{\mf{a}} \beta_{\mf{b}} w(\mf{a} \mf{b}) \Big| \ge \delta X,\] and recall once again that $w = f \boxtimes_{\ell} f'$ is in product form (\cref{def-product-form}) and that $f,f'$ are supported on $[\pm 2nX^{1/2}]$.

By pigeonhole and the fact that $w$ is supported on principal ideals, there is some ideal class $g \in C(K)$ such that the contribution from $\mf{a} \in [g^{-1}]$ is $\gg \delta X$. By Cauchy, the triangle inequality and the $1$-boundedness of $\alpha, \beta$ it follows that
\begin{equation}\label{cauchyed-ideals} \sum_{\substack{\mf{a}, \mf{a}' \in [g^{-1}] \\ N\mf{a}, N\mf{a}' \sim  L}} \Big| \sum_{\substack{\mf{b}\in [g] \\ N\mf{a}\mf{b}, N\mf{a}'\mf{b} \le X}} w(\mf{a} \mf{b})\overline{w(\mf{a}' \mf{b})}\Big| \gg \delta^2 XL.\end{equation}
Now we parametrise ideals as in \cref{section6}. Let $\mf{c}_g,\mf{c}'_{g^{-1}}$ be as in \cref{lem:ideal-reps}, and parametrise $\mf{b}$ by $(z_2) \mf{c}_g$ with $z_2 \in \Pi_g \setminus \{0\}$, and $\mf{a}, \mf{a'}$ by $(z_1) \mf{c}'_{g^{-1}}$, $(z'_1) \mf{c}'_{g^{-1}}$ respectively, where $z_1 \in \Pi'_{g^{-1}} \setminus \{0\}$. Here $\Pi_g, \Pi'_{g^{-1}}$ are the sublattices of $\Z[\sqrt{-n}]$ described in \cref{lem:ideal-reps}. The parametrisation covers each triple $\mf{a}, \mf{a'}, \mf{b}$ a total of $r^3$ times due to associate values of $z_1, z'_1, z_2$ giving the same ideals. With these parametrisations, \cref{cauchyed-ideals} becomes 
\[ \sum_{\substack{z_1, z'_1 \in \Pi'_{g^{-1}} \\ |z_1|, |z'_1| \sim \gamma_1 L}}  \Big|  \sum_{\substack{z_2 \in \Pi_g \\ |z_1 z_2|, |z'_1 z_2| \le X^{1/2}/\gamma}}  w((\gamma z_1 z_2 ))\overline{w((\gamma z'_1 z_2 ))} \Big| \gg \delta^{2} XL.\]
Here, $\gamma_1 := (N\mf{c}'_{g^{-1}})^{-1}$ and $\gamma = \gamma_g$ is the generator of $\mf{c}_g \mf{c}'_{g^{-1}}$, a positive rational of height $O_n(1)$ (see \cref{lem:ideal-reps} (ii)). Using the fact that $w = f \boxtimes_{\ell} f'$ gives
\begin{align*} \sum_{\substack{z_1, z'_1 \in \Pi'_{g^{-1}} \\ |z_1|^2, |z'_1|^2\sim \gamma_1 L}}  \Big|  \sum_{\substack{z_2 \in \Pi_g \\ |\Re(z_2)| \le  C(X/L)^{1/2} \\  n^{-1/2}|\Im(z_2)| \le C(X/L)^{1/2}}} & \Psi_{\ell} \Big(\frac{\gamma z_1 z_2}{X^{1/2}}\Big)\overline{\Psi}_{\ell} \Big(\frac{\gamma z'_1 z_2}{X^{1/2}}\Big) f\big(\gamma \Re (z_1 z_2)\big) f'\big( n^{-1/2}\gamma  \Im ( z_1 z_2)\big) \times \\ & \times \overline{f\big(\gamma \Re (z'_1 z_2)\big) f'\big( n^{-1/2}\gamma  \Im ( z'_1 z_2)\big)} \Big| \gg \delta^{2} XL.\end{align*}
with $\Psi_{\ell}$ as in \cref{r-sector}. The two rough cutoffs $\Psi_{\ell}$ may be replaced by the smooth cutoffs $\tilde \Psi_{\ell}$ constructed in \cref{annulus-smooth}, now with $\eps = c \delta^2$. Analogously to \cref{decoupled}, we obtain 
\begin{align*} \sum_{\substack{z_1, z'_1 \in \Pi'_{g^{-1}} \\ |z_1|^2, |z'_1|^2\sim \gamma_1 L}}  \Big|  \sum_{\substack{z_2 \in \Pi_g \\ |\Re(z_2)| \le  C(X/L)^{1/2} \\  n^{-1/2}|\Im(z_2)| \le C(X/L)^{1/2}}} &  f_{\xi_1}\big(\gamma \Re (z_1 z_2)\big) f'_{\xi_2}\big( n^{-1/2}\gamma  \Im ( z_1 z_2)\big) \times \\ & \times \overline{f_{\xi'_1}\big(\gamma \Re (z'_1 z_2)\big) f'_{\xi'_2}\big( n^{-1/2}\gamma  \Im ( z'_1 z_2)\big)} \Big| \gg \delta^{26} XL,\end{align*} where $f_{\xi_1}, f'_{\xi_2}, f_{\xi'_1}, f'_{\xi'_2}$ are twists of $f, f', f, f'$ respectively by fixed additive characters (whose precise form is irrelevant). 

Now we pass to real and imaginary parts, writing $z_1 =a + b \sqrt{-n}$, $z'_1 =a' + b' \sqrt{-n}$ and $z_2 = x + y \sqrt{-n}$ where $a,b,a', b',x,y \in \Z$. This gives, with $\Gamma, \Gamma'$ the lattices defined in \cref{gammas-def},

\begin{align*} \sum_{ \substack{a,a', b,b' \in \Z \\ |a|,|a'|, |b|, |b'| \ll L^{1/2}}} \Big|\sum_{\substack{(x,y)\in \Gamma \\ |x|, |y| \le C(X/L)^{1/2}}} & f_{\xi_1}(\gamma(ax - nby)) f'_{\xi_2}(\gamma(bx + ay))  \times \\ & \times  \overline{f_{\xi'_1}(\gamma(a'x - nb'y)) f'_{\xi'_2}(\gamma(b'x + a'y))}  \Big| \gg \delta^{26} XL.\end{align*}
Henceforth, write $f_1 = f_{\xi_1}$, $f_2 = f'_{\xi_2}$, $f_3 = \ol{f_{\xi_1'}}$, $f_4 = \ol{f'_{\xi_2'}}$. Detecting the condition $(x,y) \in \Gamma'$ using \cref{gam-orth}, it follows by the triangle inequality that for some $\zeta, \zeta'$ we have

\begin{align} \nonumber \sum_{ \substack{a,a',b,b' \in \Z \\ |a|, |a'|,|b|, |b'| \ll L^{1/2}}} \Big| & \sum_{\substack{x,y  \in \Z\\ |x|, |y| \le C(X/L)^{1/2}}} f_1(\gamma(ax - nby)) f_2(\gamma(bx + ay)) \times  \\ & \times f_3(\gamma(a'x - nb'y)) f_4(\gamma(b'x + a'y)) e(\zeta x + \zeta' y)  \Big| \gg \delta^{26} XL.\label{before-sub}\end{align}
We now select a set $\mc{Q}$ of $\gg \delta^{26} L^2$ quadruples $(a,a', b,b')$ for which the inner sum is $\gg \delta^{26} X/L$, and satisfying some further gcd, size and nondegeneracy conditions, the need for which will become apparent later, but which we need to select for now. The conditions are
\begin{equation}\label{ab-gcd-conds} \gcd(a, b), \gcd(a',b') \ll \delta^{-26}, \quad \delta^{26}L^{1/2} \ll |a|, |a'|, |b|, |b'| \ll L^{1/2}, \quad a'b - ab', aa' + n bb' \ne 0.
\end{equation}
First note that each quadruple $(a,a',b,b')$ contributes at most $O(X/L)$ to be sum \cref{before-sub}. Therefore there are at least $\gg \delta^{26}L^2$ quadruples $\mc{Q}'$ each contributing at least $\gg \delta^{26}X/L$ to this sum. In order to construct $\mc{Q}$, we remove quadruples from $\mc{Q}'$ violating \cref{ab-gcd-conds}. To see this is possible leaving at least half of $\mc{Q}'$, note that the set of quadruples satisfying $\gcd(a,b) + \gcd(a',b')\gg \delta^{-26}$ is bounded by $\sum_{d\ge \delta^{-26}}(L^{1/2})^4/d^2\ll \delta^{26}L^2$. Similarly, the number of quadruples satisfying $\min(|a|,|a'|,|b|,|b'|)\le \delta^{26}L^{1/2}$ is bounded by $\ll \delta^{26} L^{2}$. Finally if $a,b$ are non-zero, then given $a,b,a'$ there are at most $2$ choices of $b'$ which cause $a'b-ab' = 0$ or $aa' +nbb' = 0$; therefore there are at most $L^{3/2}$ additional violating tuples here, which is $\ll \delta^{26} L^2$ provided that $L\gg \delta^{-104}$.

For the argument over the next couple of pages it is convenient to write
\[ a_1 := \gamma a, \quad a_2 := \gamma b, \quad a_3 := \gamma a', \quad a_4 := \gamma b', \] \begin{equation}\label{oct-subs} b_1 := -n\gamma b, \quad b_2 := \gamma a , \quad  b_3 := -n\gamma b', \quad b_4 := \gamma a'.\end{equation} It also saves on notation to write
\begin{equation}\label{i-def} I := [\pm C(X/L)^{1/2}].\end{equation} Thus, for $(a, a', b, b') \in \mc{Q}$ we have
\begin{equation}\label{typii-4} \Big|\sum_{x,y \in I}e(\zeta x + \zeta' y) \prod_{i=1}^4 f_i(a_i x + b_i y)\Big| \gg \delta^{26} XL.\end{equation}
 
We are now going to bound the LHS of \cref{typii-4} in terms of a Gowers--Peluse norm of $f_4$, in preparation for the application of concatenation results. Set \begin{equation}\label{H-def} H := \lfloor \delta^{C_1}X^{1/2}/L\rfloor\end{equation} for some $C_1$ to be specified. 
The important thing to note here is that if $L \sim X^{1/2 - \kappa}$ for some $\kappa > 0$ then $|H| \sim \delta^{O(1)} X^{\kappa}$. The fact that this is a small power of $X$ will be vital later (to apply \cref{lem:diophantine-upp} effectively), and is what constrains our Type II methods to the regime $L \ll X^{1/2 - \kappa}$ for some $\kappa > 0$.

For $h \in [\pm H]$, substitute $x := x' + b_1 h$, $y := y' - a_1 h$. Then 
\[ \Big| \sum_{\substack{x' \in I - b_1 h \\ y' \in I + a_1 h} }e(\zeta x' + \zeta' y' + (\zeta b_1 - \zeta' a_1) h) f_1(a_1 x' + b_1 y') \prod_{j = 2}^4 f_j (a_j x' + b_j y' + (a_j b_1 - a_1 b_j)h)\Big| \gg \delta^{26}X/L.\]
We have $|a_1 h|,|b_1 h| \ll  L^{1/2} H \ll \delta^{C_1} (X/L)^{1/2}$, so the error in replacing the shifted intervals $I - b_1 h$, $I + a_1 h$ by two copies of $I$ is $\ll \delta^{C_1} X/L$, which is negligible for $C_1 > 28$. Therefore we have (replacing the dummy variables $x', y'$ by $x, y$ respectively) 
\[ \Big| \sum_{\substack{x \in I \\ y \in I} }e(\zeta x + \zeta' y + (\zeta b_1 - \zeta' a_1) h) f_1(a_1 x + b_1 y) \prod_{j = 2}^4 f_j (a_j x + b_j y + (a_j b_1 - a_1 b_j)h)\Big| \gg \delta^{26}X/L.\]
Summing over $h \in [\pm H]$ gives
\begin{equation}\label{7382} \sum_{|h| \le H} c_h \sum_{x, y \in I}  e\big((\zeta b_1 - \zeta' a_1) h\big)\prod_{j = 2}^4 f_j (a_j x + b_j y + (a_j b_1 - a_1 b_j)h)  \gg \delta^{26}XH/L,\end{equation} where the $c_h$ are complex numbers with unit modulus, chosen to make each term in the sum over $h$ real and nonnegative. Switching the order of summation and applying Cauchy--Schwarz gives
\begin{align*} \sum_{x, y \in I} & \sum_{|h|, |h'| \le H} c_h \overline{c}_{h'} e\big((\zeta b_1 - \zeta' a_1) h\big)e\big({-(\zeta b_1 - \zeta' a_1) h'}\big) \times \\ & \times \prod_{j = 2}^4 f_j (a_j x + b_j y + (a_j b_1 - a_1 b_j)h) \prod_{j = 2}^4 \overline{f_j (a_j x + b_j y + (a_j b_1 - a_1 b_j)h')}  \gg \delta^{26}XH/L.\end{align*}
By the triangle inequality (and recalling the notation for difference operators as detailed at the beginning of \cref{section5}) we obtain

\[ \sum_{|h|, |h'| \le H} \Big| \sum_{x, y \in I} \prod_{j=2}^4 \Delta_{(a_j b_1 - a_1 b_j) (h, h')} f_j(a_j x + b_j y)  \Big| \gg \delta^{52} XH^2/L.
\]
Therefore for $\gg \delta^{52} H^2$ values of $(h,h')$ we have 
\[ \sum_{x, y \in I} \prod_{j=2}^4 \Delta_{(a_j b_1 - a_1b_j) (h,h')} f_j(a_j x + b_j y)   \gg \delta^{52} X/L.\]

Note that this statement is very similar to \cref{typii-4}, only with a product of three functions rather than four, and with the exponent $26$ replaced by $52$. Consequently we may repeat the argument two more times, obtaining at the end $\gg \delta^{208} H^6$ sextuples $(h_1,h'_1, h_2,h'_2,h_3,h'_3)$, $h_i, h'_i \in [\pm H]$ such that 

\[\sum_{x,y \in I}\Delta_{(a_4b_1-b_4a_1)(h_1,h'_1)}\Delta_{(a_4b_2-b_4a_2)(h_2,h'_2)}\Delta_{(a_4b_3-b_4a_3)(h_3,h'_3)}f_4(a_4x + b_4y) \ge \delta^{208} X/L.\] For these further two iterations of the argument to work, we need $C_1 > 104$, thus for definiteness set $C_1 := 200$. Henceforth we will not keep track of exponents explicitly.

We apply \cref{lem:U^2-control} to this statement, taking in that lemma $I_1 = I_2 = I$ with $I$ as in \cref{i-def}, $N = X^{1/2}$, $Q = L^{1/2}$, $M$ the denominator of $\gamma$ and $\eta \sim \delta^{204}$. Most of the conditions of that lemma follow from \cref{ab-gcd-conds}; the condition $N \ge Q^2/\eta^3$ follows from $X^{1/2} \ge L \delta^{-1000}$, which follows from the assumptions of \cref{prop:typeii-to-gowers} if $C_{\ref{prop:typeii-to-gowers}}$ is large enough. 

It follows that, for each of these sextuples, 
\[ \Vert \Delta_{(a_4b_1-b_4a_1)(h_1,h'_1)}\Delta_{(a_4b_2-b_4a_2)(h_2,h'_2)}\Delta_{(a_4b_3-b_4a_3)(h_3,h'_3)}f_4 \Vert_{U^2[X^{1/2}]}^{4} \gg \delta^{O(1)}.\] 
We have the above inequality for a set of $\gg \delta^{O(1)} H^6$ sextuples $(h_1,h'_1 ,h_2,h'_2, h_3,h'_3)$ from $[\pm H]^6$. However, $\Vert \Delta_{(a_4b_1-b_4a_1)(h_1,h'_1)}\Delta_{(a_4b_2-b_4a_2)(h_2,h'_2)}\Delta_{(a_4b_3-b_4a_3)(h_3,h'_3)}f_4 \Vert_{U^2[X^{1/2}]}$ is nonnegative and therefore we may average over all sextuples $(h_1,h'_1 ,h_2,h'_2, h_3,h'_3) \in [\pm H]^6$ to obtain 
\[\E_{(h_1,h'_1 ,h_2,h'_2, h_3,h'_3)\in[\pm H]^6} \Vert \Delta_{(a_4b_1-b_4a_1)(h_1,h'_1)}\Delta_{(a_4b_2-b_4a_2)(h_2,h'_2)}\Delta_{(a_4b_3-b_4a_3)(h_3,h'_3)}f_4 \Vert_{U^2[X^{1/2}]}^{4} \gg \delta^{O(1)}\]

Unpacking the definition of the $U^2[X^{1/2}]$ norm via  \cref{gowers-def}, this implies that 
\begin{align}\nonumber
X^{-3/2}\sum_{x,h,h'\in \mb{Z}} \E_{(h_1,h'_1 ,h_2,h'_2, h_3,h'_3)\in[\pm H]^6} & \Delta_{(a_4b_1-b_4a_1)(h_1,h'_1)}\Delta_{(a_4b_2-b_4a_2)(h_2,h'_2)}\\ & \Delta_{(a_4b_3-b_4a_3)(h_3,h'_3)}\Delta_{h} \Delta_{h'} f_4(x) \gg \delta^{O(1)}.\label{eq:gowers-unpack}
\end{align}
Via switching the sum over $x$ inward and recalling the expression in \cref{gp-explicit-def}, we have that 
\begin{equation}\label{first-gp-statement}
X^{-1}\sum_{h,h'\in \mb{Z}} \sum_{q\in \mc{Q}}\snorm{\Delta_{h} \Delta_{h'} f_4}^8_{U_{\GP}[X^{1/2}; \mu_{1,q}, \mu_{2,q}, \mu_{3,q}]} \gg \delta^{O(1)}L^2\end{equation}
where, for $q = (a, a', b, b') \in \mc{Q}$ and for $j \in \{1,2,3\}$,  $\mu_{j,q}$ is the uniform measure on $(a_4 b_j - b_4 a_j)[\pm H]$. Recalling the definitions \cref{oct-subs} of the $a_i, b_i$ in terms of $q = (a,a', b, b')$, we have

\begin{equation}\label{mus-explicit} \mu_{1,q} \sim  \Unif_{-(2n\gamma)^2 (aa' + nbb') [\pm H]} , \quad \mu_{2,q} \sim \Unif_{(2n\gamma)^2 (ab' - a'b) [\pm H]}, \quad  \mu_{3,q} \sim \Unif_{ -(2n\gamma)^2 (a'^2+ nb'^2) [\pm H]}. \end{equation}

Recall that we have this for a set $\mc{Q}$ of $\gg \delta^{O(1)} L^2$ quadruples $q = (a,a',b,b')$ satisfying the conditions \cref{ab-gcd-conds}.  Summing \cref{first-gp-statement} over all these tuples gives

\begin{equation}\label{second-gp-statement}X^{-1}\sum_{h,h'\in \mb{Z}}\sum_{q \in \mc{Q}} \snorm{\Delta_{h} \Delta_{h'} f_4}^8_{U_{\GP}[X^{1/2}; \mu_{1,q}, \mu_{2,q}, \mu_{3,q}]} \gg \delta^{O(1)} L^2.\end{equation}

To drive our application of concatenation, we need to know that convolutions of different measures $\mu_{j,q}$ spread out. The following lemma encapsulates this. 

Here, and for the remainder of the section, $E \lessapprox E'$ means $E \ll \delta^{-O(1)}(\log X)^{O(1)} E'$, and $E \lessapprox_m E'$ means $E \ll \delta^{-O_m(1)}(\log X)^{O_m(1)} E'$, and similarly for $E \gtrapprox E'$, $E\gtrapprox_m E'$.

\begin{lemma}\label{lem:diophantine-upp}
Let $m \ge 1$ be an integer and let $j \in \{1,2,3\}$. Let $H$ be as in \cref{H-def}, that is to say $H = \delta^{C_1}X^{1/2}/L$. Then we have 
\[ \E_{q_1,\dots, q_{m} \in \mc{Q}} \Big\Vert \conv_{i=1}^{m} \mu_{j,q_i}\Big\Vert^2_{2} \lessapprox_{m} \max(H^{-m},  X^{-1/2}). \]
\end{lemma}
\begin{remark}
The key point to note about the conclusion is that if $L = X^{1/2 - \kappa}$ then the bound is $\lessapprox X^{-1/2}$ for $m > 1/\kappa$. Since all measures are supported on $[\pm O(X^{1/2})]$, this statement is asserting rough uniform distribution of the $\conv_{i=1}^{m} \mu_{j,q_i}$ on $[\pm O(X^{1/2})]$ on average over $q_1,\dots, q_{m}$. In particular, in the regime of interest in \cref{prop:typeii-to-gowers}, where $L \le X^{3/8}$, we can take $m = 4$.
\end{remark}

\begin{proof}
We first handle the case $j = 1$, the other cases are analogous and we will indicate the minor modifications required at the end. First of all, note that 
\[  \Big\Vert \conv_{i=1}^{m} \mu_{j, q_i} \Big\Vert^2_{2} =  \Big( \conv_{i=1}^{m} ( \mu_{j,q_i} \ast \mu_{j,q_i} ) \Big)(0) \ll  \Big( \conv_{i=1}^{m} \tilde\mu_{j,q_i}\Big)(0) ,\] where $\tilde\mu_{j,q}$ is the variant of $\mu_{j,q}$ with $H$ replaced by $2H$, since $\mu_{j,q} \ast \mu_{j,q} \ll \tilde \mu_{j,q}$ pointwise.  It therefore suffices to show that 
\[
\E_{q_1,\dots, q_{m} \in \mc{Q}} \Big(\conv_{i=1}^{m} \tilde\mu_{j,q_i}\Big)(0) \lessapprox_{m} \max(H^{-m},  X^{-1/2}),
\]
or equivalently
\begin{equation}\label{suff-measures}
\nu^{(m)}(0) \lessapprox_{m} \max(H^{-m},  X^{-1/2}).
\end{equation}
where $\nu$ is the probability measure $\E_{q \in \mc{Q}} \mu_{j,q}$, and $\nu^{(m)}$ denotes the $m$th convolution power of $\nu$. Now $\nu \le H^{-1}\delta_0 + \nu_*$, where $\nu_*$ is uniform on the multiset $\{-2n^2\gamma^2 (aa' + n bb') h : (a,a', b, b') \in \mc{Q}, h \in [\pm 2H] \setminus \{0\}\}$. Note here that by the properties \cref{ab-gcd-conds} of $\mc{Q}$ we have $\nu_*(0) = 0$. The idea now is to show that $\nu^{(2)}_*$ is somewhat close to uniform on $[\pm O(X^{1/2})]$ in the sense that 
\begin{equation}\label{nu-uniformity-claim} \Vert \nu^{(2)}_* \Vert_{\infty} \lessapprox X^{-1/2}.\end{equation} Once this is proven, we can expand (since $\delta_0$ is the identity in convolution)
\[ \nu^{(m)} = H^{-m} \delta_0 + H^{-(m - 1)} m  \nu_* + \sum_{k=2}^{m}\binom{m}{k} H^{-(m - k)}  \nu_*^{(k)}.\]
Each measure $\nu_*^{(k)}$, $k \ge 2$, is the convolution of $\nu_*^{(2)}$ with some probability measure, and hence by \cref{nu-uniformity-claim} is pointwise $\lessapprox X^{-1/2}$. Since $\nu_*(0) = 0$, the desired bound \cref{suff-measures} then follows.

It remains to establish \cref{nu-uniformity-claim}. First, note that uniformly for $t \neq 0$ we have
\begin{equation}\label{aabb-div}
\#\{(a, a', b, b') \in \mathcal{Q} : aa' + nbb' = t\} \ll \sum_{|u|\ll L}\tau(u)\tau\bigg(\frac{t - u}{n}\bigg)\ll L(\log L)^{3};
\end{equation}
here the divisor function is extended so that $\tau(x) = \tau(|x|)$ and $\tau(x) = 0$ if $x=0$ or $x\notin \Z$. It follows that $\nu_* \ll |\mc{Q}|^{-1}H^{-1} L(\log L)^3 \tau \lessapprox X^{-1/2} \tau$ pointwise, noting here that $|\mc{Q}| \gg \delta^{O(1)} L^2$ and $H = \delta^{O(1)} X^{1/2}/L$. Noting also that $\nu_*$ is supported on $[\pm O(X^{1/2})]$, it follows using \cref{lem:divisor-moment} in the case $m = 2$ that 
\[ \nu_*^{(2)}(t) \lessapprox X^{-1} \sum_{|x| \ll X^{1/2}} \tau(x) \tau( t - x) \ll X^{-1} \sum_{|x| \ll X^{1/2}} \tau(x)^2 \lessapprox X^{-1/2},  \] which is the required bound \cref{nu-uniformity-claim} in the case $j = 1$.

In the cases $j = 2$ and $3$ we proceed entirely analogously, the only difference of any interest being the corresponding bound to \cref{aabb-div}. For $j = 3$ we have instead the bound
\begin{equation}\label{aabb-div-2}
\#\{(a, a', b, b') \in \mc{Q}: a^{\prime 2} + nb^{\prime 2} = t\} \ll L \tau(t);
\end{equation}
see \cref{repdivisor}. This gives the bound $\nu_* \lessapprox X^{-1/2} \tau^2$ pointwise, and we can conclude as before but now using the case $m = 4$ of \cref{lem:divisor-moment}.

The case $j = 2$ is almost identical to the case $j = 1$. Instead of \cref{aabb-div} we have the bound
\begin{equation}\label{aabb-div-3}
\#\{(a, a', b, b') \in \mc{Q} : a'b - ab' = t\} \ll L(\log L)^3 ,
\end{equation}
which may be proven in the same manner as \cref{aabb-div}.
\end{proof}
From now on we will assume that $L \le X^{3/8}$, which is one of the assumptions of \cref{prop:typeii-to-gowers}. Thus, applying \cref{lem:diophantine-upp} with $m = 4$ (we take a power of two with the forthcoming application of \cref{lem:concat-main} in mind), we obtain
\begin{equation}\label{lem71-cor} \E_{q_1,q_2,q_3,q_4 \in S} \Big\Vert \conv_{i=1}^{4} \mu_{j,q_i}\Big\Vert^2_{2} \ll \Big(\frac{\log X}{\delta}\Big)^{C_0} X^{-1/2} \end{equation} for some $C_0$ (beyond this point we will cease to use the $\lessapprox$ notation).

We now return to our main line of argument, specifically statement \cref{second-gp-statement}, and we are now ready to apply the key concatenation result of Kravitz, Kuca and Leng, \cref{lem:concat-main}. We apply that result with $s = 3$, with $I = \mc{Q}$, and measures given by $\mu_{1,i} = \mu_{1,q}$, $\mu_{2,i} = \mu_{2,q}$, $\mu_{3,i} = \mu_{3,q}$ and $N = O(X^{1/2})$ and with $m = 4$. The conclusion is that there is $t \ll 1$ such that 

\begin{equation}\label{concat-conclu} X^{-1}\sum_{h,h'\in \mb{Z}} \E_{q_1,\dots, q_t \in \mc{Q}} \Vert \Delta_{h} \Delta_{h'} f_4 \Vert_{U_{\GP}[X^{1/2}; \Omega(q_1,\dots, q_t)]} \gg \delta^{O(1)} ,\end{equation}
where $\Omega(q_1,\dots, q_t)$ is the collection of probability measures

\[ \Omega(q_1,\dots, q_t) = \Big( \conv_{i = 1}^{4}  \mu_{j q_{k_i}}\Big)_{\substack{1 \le k_1 < k_2 <k_3 < k_{4} \le t\\j \in \{1,2,3\}}}.\]

Set $T := \big(\frac{\log X}{\delta}\big)^{C_1}$ for some very large constant $C_1 > C_0$. We say that a tuple $(q_1,\dots, q_t) \in \mc{Q}^{\otimes t}$ is \emph{good} if $\Vert \mu \Vert_2^2 \le T/X^{1/2}$ for all $\mu \in \Omega(q_1,\dots, q_t)$. Write $\mc{G}$ for the set of all good tuples. By \cref{lem71-cor} and Markov's inequality, 
\[ \P((q_1,\dots, q_t) \notin \mc{G}) \ll \frac{1}{T} \Big(\frac{\log X}{\delta}\Big)^{C_0}.\] Therefore, from \cref{concat-conclu}, if $C_1$ is large enough we have
\begin{equation}\label{concat-conclu-2} X^{-1}\sum_{h,h'\in \Z} \E_{q_1,\dots, q_t \in \mc{Q}} 1_{(q_1,\dots, q_t) \in \mc{G}} \Vert \Delta_{h} \Delta_{h'} f_4 \Vert_{U_{\GP}[X^{1/2}; \Omega(q_1,\dots, q_t)]} \gg \delta^{O(1)} .\end{equation}
In particular, there is at least one good tuple $(q_1,\dots, q_t)$ such that 
\begin{equation}\label{concat-conclu-3} X^{-1}\sum_{h,h'\in \Z} 
\Vert \Delta_{h} \Delta_{h'} f_4 \Vert_{U_{\GP}[X^{1/2}; \Omega(q_1,\dots, q_t)]} \gg \delta^{O(1)}.\end{equation}
Write $\delta(h, h')$ for the size of the inner norm, thus
\begin{equation}\label{del-avg} X^{-1}\sum_{h,h'\in \Z} \delta(h, h') \gg \delta^{O(1)}.\end{equation}
Note that since $\on{supp}(f_4)\subseteq [\pm X^{1/2}]$, we have that $\delta(h,h') = 0$ if $\max(|h|,|h'|)\ge 5X^{1/2}$ (say). 

By \cref{cor53x}, for each pair $(h, h')$ we have
\[\Vert \Delta_{h} \Delta_{h'} f_4 \Vert^{2^{2r}}_{U^{2r}[X^{1/2}]} \gg T^{-O(1)}\delta(h, h')^{O(1)}, \] where $r = 3 \binom{t}{4}$ is the number of measures in each $\Omega(q_1,\dots, q_t)$.
Averaging over $h, h'\in [\pm 5X^{1/2}]$ and applying H\"older and \cref{del-avg}, we obtain
\begin{align*}
X^{-1}\sum_{h,h'\in \Z} \snorm{\Delta_{h} \Delta_{h'} f_4}^{2^{2r}}_{U^{2r}[X^{1/2}]}\gg (\delta/\log X)^{O(1)}.
\end{align*}
By expanding out the Gowers norm using \cref{gowers-def,gowers-norm-def-2}, the above is equivalent to 
\begin{align*}
\snorm{f_4}^{2^{2r+2}}_{U^{2r+2}[X^{1/2}]}\gg (\delta/\log X)^{O(1)}.
\end{align*}
This concludes the proof of \cref{prop:typeii-to-gowers}.

\section{Computing the asymptotic}\label{section8}
We turn now to the last main task in the paper, the proof of \cref{prop:main-term-comp}, which asserts an asymptotic for the `main term'
\begin{equation}\label{main-term-repeat}(\log X) \sum_{x,y \in \Z: x^2 + ny^2\le X}
\chi_{\infty}^{(\ell)} (x + y \sqrt{-n})\Lambda_{\cramer}(x)\Lambda_{\cramer}(y) 1_{x^2 + ny^2 \operatorname{prime}} .\end{equation}

\subsection{Hecke Gr\"o{\ss}encharaktere on imaginary quadratic fields}\label{basic-hecke-sec}

In this section we recall some basic facts about (Hecke) Gr\"o{\ss}encharaktere, as well as key results concerning sums over prime ideals that we will require. 

A good and fairly low-level introduction to the notion of Gr\"o{\ss}encharakter on an imaginary quadratic field is \cite[Section 3.8]{IK-book}. In the setting of imaginary quadratic fields we can be quite explicit, and the Gr\"o{\ss}encharaktere we are interested in are defined as follows. Let $\mf{m}$ be a non-zero integral ideal, that is to say an ideal in $\O_K$. By a \emph{Dirichlet character} to modulus $\mf{m}$ we mean the lift of a character on $(\O_K/\mf{m})^*$ to $\O_K$, defined to be zero for elements of $\O_K$ which are not coprime to $\mf{m}$ (the definition is analogous to that of a Dirichlet character on $\Z$). 

\begin{definition}
Let $K$ be an imaginary quadratic field embedded in $\C$. Let $\mf{m}$ be a non-zero ideal in $\O_K$. Denote by $I_{\bf{m}}$ the group of fractional ideals which are coprime to $\mf{m}$. Let $\ell \in \Z$. Then a Gr\"o{\ss}encharakter modulo $\mf{m}$ with frequency $\ell$ is a homomorphism $\psi : I_{\mf{m}} \rightarrow \C^*$ which satisfies $\psi((\alpha)) = \chi^{(\ell)}_{\infty}(\alpha) \chi(\alpha)$ for all $\alpha \in K^*$ coprime to $\mf{m}$, where (as in \cref{chi-infinity-def}) $\chi_{\infty}^{(\ell)}(\alpha) = (\alpha/|\alpha|)^{\ell}$, and $\chi$ is a Dirichlet character to modulus $\mf{m}$.
\end{definition}

\begin{remark}
We use the term Gr\"o{\ss}encharakter to emphasise that the frequency $\ell$ may not be zero. Some authors use the term Hecke character to mean a character of finite order, that is to say with $\ell = 0$. Others such as \cite{IK-book} use the terms Hecke character and (Hecke) Gr\"o{\ss}encharakter interchangably. This can be (and was!) a source of considerable confusion. We will use the term Hecke character only when the frequency $\ell$ is definitely zero (or equivalently the character has finite order).
\end{remark}

Note that if we have a Gr\"o{\ss}encharakter $\psi$ as above, then $\chi^{(\ell)}_{\infty},\chi$ satisfy the `units consistency condition' \begin{equation}\label{units-consistency}\chi^{(\ell)}_{\infty}(u) \chi(u) = 1 \qquad \mbox{for all $u \in \O_K^*$}.\end{equation} 
We now show that the converse is also true. Let $I$ be the group of non-zero fractional ideals in $K$. By a \emph{class group character} we mean a homomorphism $\psi : I \rightarrow \C$ which is trivial on principal ideals. Such a homomorphism factors through the class group of $K$, and so there are exactly $h_K$ such characters.

\begin{lemma}\label{hecke-lift}
Suppose that $\chi^{(\ell)}_{\infty}, \chi$ as above satisfy the units consistency condition \cref{units-consistency}.  Then there is a Gr\"o{\ss}encharakter $\psi$ to modulus $\mf{m}$ satisfying $\psi((\alpha)) = \chi^{(\ell)}_{\infty}(\alpha) \chi(\alpha)$ for all $\alpha \in K^*$ coprime to $\mf{m}$. Moreover $\psi$ is unique up to multiplication by class group characters.
\end{lemma}
\begin{proof}
Write $I^0_{\mf{m}}$ for the group of principal fractional ideals coprime to $\mf{m}$.
The units consistency condition is precisely what is needed to guarantee that there is a homomorphism $\psi_0 : I^0_{\mf{m}} \rightarrow \C^*$ with $\psi_0((\alpha)) = \chi^{(\ell)}_{\infty}(\alpha) \chi(\alpha)$ for all $\alpha \in K^*$. The lemma follows from the fact that $\psi_0$ extends to a homomorphism $\psi$ on $I_{\mf{m}}$. This is a special case of a general algebraic fact: if $G_0 \le G$ are abelian groups, and $Z$ is a divisible abelian group, then any homomorphism $\psi_0 : G_0 \rightarrow Z$ can be extended to a homomorphism $\psi : G \rightarrow Z$. In the case that $[G : G_0] < \infty$ (as in our situation) one may prove this fact in finitely many stages, first extending $\psi_0$ to a homomorphism on $G_1 := \langle G_0, x\rangle$ by defining $\psi_1 (g_0 x^r) := \psi_0(g_0) z^r$, where $z \in Z$ is any element with $z^d = \psi_0(x^d)$, where $d$ is the order of $x$ in $G_1/G_0$, and then proceeding similarly in further stages. (We leave it for the reader to check that $\psi_1$ is well-defined. A Zorn's lemma argument would handle the general case.) 

It is clear that $\psi$ is unique up to multiplication by characters $\xi : I_{\mf{m}} \rightarrow \C^*$ which vanish on $I_{\mf{m}}^0$. We claim that any such $\xi$ is (the restriction of) a class group character. To see this, consider the natural injective homomorphism $\xi : I_{\mf{m}}/I_{\mf{m}}^0 \rightarrow I/I^0$ induced by the inclusion $I_{\mf{m}} \hookrightarrow I$, where $I^0$ denotes the principal fractional ideals in $K$. It is a well-known fact that every ideal class in $K$ contains infinitely many prime ideals (this can be proven using the prime ideal theorem for class group characters, which is a special case of \cref{main-analytic-input} below) and in particular at least one ideal coprime to $\mf{m}$. Therefore $\xi$ is surjective and is hence an isomorphism. The claim follows.
\end{proof}

In this section most Gr\"o{\ss}encharaktere will have the same frequency $\ell$, which is the one in the statement of \cref{prop:main-term-comp}. Recall that we are making the assumption that $|\ell| \le e^{\sqrt{\log X}}$.

We define the von Mangoldt function $\Lambda_K$ on ideals by
\[ \Lambda_K(\mf{a}) = \left\{ \begin{array}{ll} \log N\mf{p} & \mbox{if $\mf{a} = \mf{p}^k$ for some $k \ge 1$} \\ 0 & \mbox{otherwise}. \end{array}\right.\]

\cref{main-analytic-input} below, a form of the prime ideal theorem for Gr\"o{\ss}encharaktere, is the main input we require from classical multiplicative number theory. Before stating the result, we recall the notion of one Gr\"o{\ss}encharakter being induced from another. If $\psi$ is a Gr\"o{\ss}encharakter to modulus $\mf{m}$, and if $\tilde{\mf{m}}$ is another ideal with $\mf{m} \mid \tilde{\mf{m}}$, then $\psi$ induces a Gr\"o{\ss}encharakter $\tilde{\psi}$ to modulus $\tilde{\mf{m}}$ by setting $\tilde{\psi}(\mf{a}) = \psi(\mf{a})$ if $\mf{a}$ and $\tilde{\mf{m}}$ are coprime, and $\tilde{\psi}(\mf{a}) = 0$ otherwise. The character which takes the value $1$ for all nonzero prime ideals is called the \emph{principal} character (to modulus $(1) = \O_K$). It induces a character to any modulus $\mf{m}$, which we call the principal character to that modulus. A Gr\"o{\ss}encharakter is called \emph{primitive} if it is not induced from any Gr\"o{\ss}encharakter of strictly smaller modulus. 

\begin{proposition}\label{main-analytic-input}
Let $K = \Q(\sqrt{-n})$ and let $X > 1$. Then there is at most one `exceptional' primitive Gr\"o{\ss}encharakter $\chi_{\mf{d}^*}$ to some modulus $\mf{d}^*$ and associated real number $\beta \in (\frac{1}{2}, 1)$ such that the following is true uniformly for all Gr\"o{\ss}encharaktere $\psi$ on $K$ whose modulus $\mf{m}$ and frequency $\ell$ satisfy $|N\mf{m}|, |\ell| \le e^{\sqrt{\log X}}$:
\[ \sum_{N\mf{a} \le X} \Lambda_K(\mf{a}) \psi(\mf{a}) = X 1_{\psi = \psi_0, \ell = 0} - \frac{X^{\beta}}{\beta} 1_{\psi \sim \chi_{\mf{d}^*}, \ell = 0} + O(Xe^{-c\sqrt{\log X}}).\]
Here $\psi_0$ is the principal character to modulus $\mf{m}$. The exceptional charakter $\chi_{\mf{d}^*}$, if it exists, is quadratic \textup{(}that is to say $\chi_{\mf{d}^*}^2$ is the principal charakter to modulus $\mf{d}^*$\textup{)}. The notation $\psi \sim \chi_{\mf{d}^*}$ means that $\psi$ is induced from $\chi_{\mf{d}^*}$, which in particular requires that $\mf{d}^* \mid \mf{m}$. Finally, $\beta$ satisfies
\begin{equation}\label{siegel-thm} 1 - \beta \gg_{\eps} (N\mf{d}^*)^{-\eps}
.\end{equation}
\end{proposition}
\begin{remark}
The exceptional character depends on the scale $X$ (which is fixed throughout the paper). The number $\beta$ is a root of $L(s, \chi_{\mf{d}^*})$. If no exceptional character exists (which is conjectured to be the case) then the term with $-X^{\beta}/\beta$ should be omitted.
\end{remark}
\begin{proof}[Sketch]
This result is well-known to experts and versions of it have been used in several places; see, for example, \cite[Lemma 4.5]{maynard-norm-forms} for an essentially identical statement. However it is very difficult to locate a proof in the literature. Ultimately, it can be demonstrated by applying \cite[Theorem 5.13]{IK-book} to appropriate zero-free regions for Gr\"o{\ss}encharaktere, all of which may be found either explicitly or with minor modifications (but sometimes with only skeleton proofs) across papers of Mitsui \cite{Mitsui} and Fogels \cite{fogels-1,fogels-2,fogels-3} from the late 1950s and early 1960s. The arguments are analogous to the proofs over $\Q$ (with Dirichlet characters) which are covered in complete detail in \cite{davenport-book} or \cite{MV-book}. Due to the somewhat unsatisfactory nature of the literature here, the first-named author is preparing an expository document \cite{BJG-hecke-exposition} with a complete proof. 
\end{proof}

\subsection{Decomposing \texorpdfstring{$\Lambda_{\cramer}$}{}}
We now begin the proof of \cref{prop:main-term-comp}. The first step is to effect a decomposition of $\Lambda_{\cramer}$ into two parts, using ideas related to the Brun sieve 
(as discussed for instance in \cite[Chapter 6]{opera-cribro} or \cite[Chapter 17]{koukoulopoulos}). We will show that, of the resulting four terms contributing to \cref{main-term-repeat}, all but one are small. The evaluation of the remaining term is a more substantial task, and is deferred to subsequent sections.

Write $\mc{P}_Q$ for the set of primes $\le Q$, and for $S \subseteq \mc{P}_Q$, write $P_S := \prod_{p \in S} p$.

Then by inclusion--exclusion we have
\[ \Lambda_{\cramer}(x) = \prod_{p \le Q} \Big(1 - \frac{1}{p}\Big)^{-1} \sum_{S \subseteq \mc{P}_Q} (-1)^{|S|} 1_{P_S | x}. \]
 Let \begin{equation}\label{t-choice} t := 20\log \log X,\end{equation} and split
\begin{equation}\label{cramer-split} \Lambda_{\cramer} = \Lambda_{\cramer}^{\sharp} + \Lambda_{\cramer}^{\flat}\end{equation} where
\begin{equation}\label{lam-sharp-def}
  \Lambda_{\cramer}^{\sharp} := \prod_{p \le Q} \Big(1 - \frac{1}{p}\Big)^{-1} \sum_{\substack{S \subseteq \mc{P}_Q \\ |S| \le t}} (-1)^{|S|} 1_{P_S | x}   
\end{equation} and
\begin{equation}\label{lam-flat-def} \Lambda_{\cramer}^{\flat} := \prod_{p \le Q} \Big(1 - \frac{1}{p}\Big)^{-1} \sum_{\substack{S \subseteq \mc{P}_Q \\ |S| > t}} (-1)^{|S|} 1_{P_S | x}.\end{equation}
The following lemma will be used to guarantee that the contribution of the $\Lambda_{\cramer}^{\flat}$ terms to \cref{main-term-repeat} will be small.

\begin{lemma}\label{lem81}
We have
\begin{equation}\label{lam-flat-ell-1} \sum_{x \le X^{1/2}}  |\Lambda_{\cramer}^{\flat}(x)| \ll X^{1/2}(\log X)^{-8}.\end{equation}
and
\begin{equation}\label{lam-sharp-ell-1} \sum_{x \le X^{1/2}}  |\Lambda_{\cramer}^{\sharp}(x)| \ll X^{1/2}(\log X)^2.\end{equation}
\end{lemma}
\begin{proof}
We begin with \cref{lam-flat-ell-1}. First note that $\prod_{p \le Q} (1 - \frac{1}{p})^{-1} \ll \log Q < \log X$. Therefore we have the pointwise bounds
\begin{equation}\label{cramer-pointwise} \Lambda_{\cramer}(x) \ll \log X\end{equation} and 
\begin{equation}\label{first-lam-flat-bd} |\Lambda_{\cramer}^{\flat}(x)| \ll (\log X)\sum_{\substack{S \subseteq \mc{P}_Q \\ |S| > t}} 1_{P_S | x} \ll (\log X) \tau(x). \end{equation}
Moreover, $\Lambda_{\cramer}^{\flat}$ is supported on $x$ with at least $t$ distinct prime factors, and thus on $x$ for which $|\tau(x)| \ge 2^t$. Thus, on the support of $\Lambda_{\cramer}^{\flat}$, we have $\tau(x) \le 2^{-t} \tau(x)^2 < (\log X)^{-12} \tau(x)^2$. Substituting into \cref{first-lam-flat-bd} and summing gives
\[ \sum_{x \le X^{1/2}}  |\Lambda_{\cramer}^{\flat}(x)| \ll (\log X)^{-11} \sum_{x \le X^{1/2}}\tau(x)^2 \ll X^{1/2} (\log X)^{-8}, \] where in the last step we used  \cref{lem:divisor-moment} in the case $m = 2$. This is \cref{lam-flat-ell-1}.

Turning to \cref{lam-sharp-ell-1}, observe the bound 

\begin{equation}\label{lam-cramer-crude} |\Lambda_{\cramer}^{\sharp}(x)| \ll  (\log X) \tau(x),\end{equation} which follows from \cref{cramer-split,cramer-pointwise,first-lam-flat-bd} and the triangle inequality. The desired bound \cref{lam-sharp-ell-1} follows by summing over $x$.
\end{proof}

Substitute the decomposition \cref{cramer-split} into \cref{main-term-repeat}. This gives a sum of four terms. Three of these terms are
\[ (\log X) \sum_{x,y \in \Z: x^2 + ny^2\le X}
\chi_{\infty}^{(\ell)}(x + y \sqrt{-n})\Lambda'_{\cramer}(x)\Lambda''_{\cramer}(y) 1_{x^2 + ny^2 \operatorname{prime}}  \]
where $\Lambda'_{\cramer}, \Lambda''_{\cramer} \in \{ \Lambda_{\cramer}^{\sharp}, \Lambda_{\cramer}^{\flat}\}$, and at least one of these functions is $\Lambda_{\cramer}^{\flat}$. Using \cref{lem81}, each such term can be bounded crudely by $X (\log X)^{-5}$, here ignoring the constraint that $x^2 + ny^2$ is prime entirely. The remaining term involves two copies of $\Lambda_{\cramer}^{\sharp}$. We have therefore reduced the task of proving \cref{prop:main-term-comp} to the task of showing the following asymptotic.

\begin{proposition}\label{main-term-sharps} Suppose that $n$ is even. Then we have
\[ 
(\log X) \sum_{\substack{x,y \in \Z \\ x^2 + ny^2 \le X}} \chi_{\infty}^{(\ell)}(x + y \sqrt{-n})\Lambda^{\sharp}_{\cramer}(x)\Lambda^{\sharp}_{\cramer}(y) 1_{x^2 + ny^2 \operatorname{prime}} =\frac{\pi  1_{\ell = 0} \kappa_n X}{\sqrt{n}}  + O(X (\log X)^{-1}),\]
where $\kappa_n$ is as given in \cref{kappa-def}.\end{proposition}

This will occupy us for the rest of the section. In particular, we assume that $n$ is even from now on.

\subsection{Writing as a sum over ideals}

The main work in establishing \cref{main-term-sharps} will be to prove the following proposition.

\begin{proposition}\label{main-term-sharps-weighted}
Suppose that $n$ is even. Then we have
\[ \sum_{\mf{a} : N\mf{a} \le X}  \Lambda_{\cramer}^{\sharp}\boxtimes_{\ell} \Lambda_{\cramer}^{\sharp}(\mf{a}) \Lambda_K(\mf{a}) =  \frac{\pi 1_{\ell = 0} \kappa_n}{\sqrt{n}} X + O(X (\log X)^{-1}).\]
\end{proposition}
The LHS here is closely related to the LHS of \cref{main-term-sharps}, with the key difference being that the $\log X$ weight on primes is replaced by a von Mangoldt weight. The proof of \cref{main-term-sharps-weighted} is quite substantial. Before getting to the details, we deduce \cref{main-term-sharps} from it by a partial summation argument.

\begin{proof}[Proof of \cref{main-term-sharps}, assuming \cref{main-term-sharps-weighted}]
Using the definition of $\boxtimes_{\ell}$ (see \cref{def-product-form}) we see that the LHS in \cref{main-term-sharps} is equal to 
\begin{equation}\label{866} \sum_{\mf{a} : N\mf{a} \le X}  \Lambda_{\cramer}^{\sharp}\boxtimes_{\ell} \Lambda_{\cramer}^{\sharp}(\mf{a}) \Lambda_K(\mf{a}) \frac{\log X}{\log N\mathfrak{a}}, \end{equation} up to a negligible error from the contribution of $\mathfrak{a} = (x + y \sqrt{-n})$ which are a proper power $\mathfrak{p}^j$, $j \ge 2$, of a prime ideal and the contribution of $x,y \in \Z$ for which $N(x + y \sqrt{-n}) = x^2 + ny^2$ is a proper power $p^j$, $j \ge 2$, of a rational prime. (This error is $\ll X^{1/2 + o(1)}$ using \cref{lam-cramer-crude} to bound the $\Lambda_{\cramer}^{\sharp}$ terms.)  To establish \cref{main-term-sharps}, it therefore suffices to prove that 
\begin{equation}\label{prop86-suffice} \sum_{\mf{a} : N\mf{a} \le X}  \Lambda_{\cramer}^{\sharp}\boxtimes_{\ell} \Lambda_{\cramer}^{\sharp}(\mf{a}) \Lambda_K(\mf{a})\big(  \frac{1}{\log N\mathfrak{a}} - \frac{1}{\log X}\big) \ll X (\log X)^{-2}. \end{equation}
For this we use partial summation, writing
\[ \frac{1}{\log N\mathfrak{a}} - \frac{1}{\log X} = \int^X_{N\mathfrak{a}} \,\frac{\mathrm{d}t}{t (\log t)^2}, \] so that the LHS of \cref{prop86-suffice} is
\[ \int^X_2 \frac{\mathrm{d} t}{t (\log t)^2} \Big( \sum_{N \mathfrak{a} \le t} \Lambda_{\cramer}^{\sharp}\boxtimes_{\ell} \Lambda_{\cramer}^{\sharp}(\mf{a}) \Lambda_K(\mf{a})\Big).\]
By \cref{main-term-sharps-weighted}, the bracketed expression is $O(t)$ uniformly in $t$, and so the LHS of \cref{prop86-suffice} is bounded by
\[ \ll \int^X_2 \frac{\mathrm{d} t}{(\log t)^2} \ll X(\log X)^{-2}.\]
This concludes the proof.
\end{proof}

\subsection{Expanding with Gr\"o{\ss}encharaktere}
We now turn to the main task of establishing \cref{main-term-sharps-weighted}. The key idea is to write the expression on the LHS in terms of sums $\sum_{\mf{a}} \Lambda_K(\mf{a}) \psi(\mf{a})$ for Gr\"o{\ss}encharaktere $\psi$, which will then allow us to bring in the key analytic input, \cref{main-analytic-input}. Opening up the definitions of the two copies of $\Lambda_{\cramer}^{\sharp}$ (recalling the definition \cref{lam-sharp-def}), we may write 
\begin{equation}\label{in-ex} \sum_{\mf{a} : N\mf{a} \le X}  \Lambda_{\cramer}^{\sharp}\boxtimes_{\ell} \Lambda_{\cramer}^{\sharp}(\mf{a}) \Lambda_K(\mf{a}) = \prod_{p \le Q} \Big(1 - \frac{1}{p}\Big)^{-2} \sum_{\substack{S_1, S_2 \subseteq \mc{P}_Q \\ |S_1|, |S_2| \le t}} (-1)^{|S_1| + |S_2|} E(S_1, S_2),\end{equation}
where
\begin{equation}\label{es1s2-def}
E(S_1, S_2) := \sum_{\mf{a} : N\mf{a} \le X} \varphi_{P_{S_1}} \boxtimes_{\ell} \varphi_{P_{S_2}}(\mf{a}) \Lambda_K(\mf{a}),\end{equation} where 
\[ \varphi_{P_S}(x) := 1_{x \in \Z, P_S | x}.\]
To understand this better, write 
\begin{equation}\label{rd1d2-def} \Gamma(S_1, S_2) := \{ \gamma \in \O_K : \gamma = x + y \sqrt{-n} \; \mbox{with} \; x, y \in \Z, P_{S_1} \mid x, P_{S_2} \mid y\}, \end{equation} and observe that if $\mf{a} = (\alpha)$ is principal then
\begin{equation}\label{psi1-times-psi2} \varphi_{P_{S_1}} \boxtimes_{\ell} \varphi_{P_{S_2}}(\mf{a}) = \sum_{u \in \O_K^*} \chi_{\infty}^{(\ell)}(u \alpha) 1_{\Gamma(S_1, S_2)} (u \alpha).\end{equation} The definition does not depend on the choice of $\alpha$; also, by definition of $\boxtimes_{\ell}$, $\varphi_{P_{S_1}} \boxtimes_{\ell} \varphi_{P_{S_2}}(\mf{a})$ = 0 if $\mf{a}$ is not principal.

 Let us also take the opportunity to remark that, for any $P_S$ appearing  in a term $E(S_1, S_2)$ involved in the sum \cref{in-ex}, we have $P_S \le Q^t < \exp(\log^{1/8} X)$. 

We now expand the function $\varphi_{P_{S_1}} \boxtimes_{\ell} \varphi_{P_{S_2}}$ using Gr\"o{\ss}encharaktere. First we note that (see \cref{def-product-form}) it is supported on principal ideals, a condition we detect using a sum over class group characters:
\begin{equation}\label{detect-principal} 1_{[\mf{a}] = 1_{\classgroup}}(\mf{a}) = h_K^{-1}\sum_{\chi_1 \in \widehat{\classgroup}} \chi_1(\mf{a}).\end{equation}

Define $\mf{m} := (D)$ where $D := 2nP_{S_1} P_{S_2}$. We claim that $\Gamma(S_1, S_2)$ is a union of congruence classes mod $\mf{m}$. To see this, suppose that $\gamma \in \Gamma(S_1, S_2)$ and that $\gamma' \in \gamma + \mf{m}$. Then $\gamma' = \gamma + D \beta$ for some $\beta \in \O_K$. Recalling (from \cref{int-contain}) that $2n \O_K \subseteq\Z[\sqrt{-n}]$, one sees from the choice of $D$ that $D \beta \in P_{S_1}P_{S_2} \Z[\sqrt{-n}]$, and the claim follows. 

We will only care about $1_{\Gamma(S_1, S_2)}(\gamma)$ at values of $\gamma$ which are coprime to $\mf{m}$. Restricted to such values we have an expansion 
 \begin{equation}\label{r-to-dirichlet} 1_{\gamma \in \Gamma(S_1, S_2)} = \sum_{\chi \mdsub{\mf{m}}} c^{S_1, S_2}_{\chi} \chi(\gamma)\end{equation} in Dirichlet characters $\chi$ to modulus $\mf{m}$ for some coefficients $c^{S_1, S_2}_{\chi}$.

 Substituting into \cref{psi1-times-psi2} gives
 \[\varphi_{P_{S_1}} \boxtimes_{\ell} \varphi_{P_{S_2}}(\mf{a}) =  \sum_{\chi \mdsub{\mf{m}}} c^{S_1,S_2}_{\chi} \sum_{u \in \O_K^*} \chi_{\infty}^{(\ell)}(u \alpha) \chi(u\alpha).  \]
 Note that the inner sum vanishes unless $\chi_{\infty}^{(\ell)}, \chi$ satisfy the units consistency condition described at the start of \cref{basic-hecke-sec}, since $\chi_{\infty}^{(\ell)}, \chi$ are both characters and \[ \sum_u \chi_{\infty}^{(\ell)}(u) \chi(u) = \sum_{j = 0}^{|\O^*_K| - 1} (\chi_{\infty}^{(\ell)}(u_0)\chi(u_0))^j,\] where $u_0$ is a generator for the group of units. If the units consistency condition is satisfied then it follows from \cref{hecke-lift} that $\chi_{\infty}^{(\ell)}\chi$ is the restriction to principal ideals of some Gr\"o{\ss}encharakter $\psi$ to modulus $\mf{m}$, with frequency $\ell$. Therefore, for principal ideals $\mf{a}$, coprime to $\mf{m}$, we have
\begin{equation}\label{coords-in-chars-2} 
 \varphi_{P_{S_1}} \boxtimes_{\ell} \varphi_{P_{S_2}}(\mf{a}) =  |\O^*_K| \sum_{\psi}  c^{S_1,S_2}_{\psi} \psi(\mf{a}),\end{equation} where $\psi$ ranges over a complete set of Gr\"o{\ss}encharaktere $\md{\mf{m}}$ with frequency $\ell$ modulo equivalence by class group characters, and we have abused notation by writing $c^{S_1,S_2}_{\psi} = c^{S_1,S_2}_{\chi}$. Finally, combining with \cref{detect-principal} we obtain an expression valid for \emph{all} ideals coprime to $\mf{m}$ (not just principal ones) namely
\begin{equation} \label{cutoff-in-Heckes}  \varphi_{P_{S_1}} \boxtimes_{\ell} \varphi_{P_{S_2}}(\mf{a}) =  |\O^*_K| h_K^{-1} \sum_{\psi} c^{S_1,S_2}_{\psi} \psi(\mf{a}),\end{equation}for all $\mf{a}$ coprime to $\mf{m}$, where now the sum is over all Gr\"o{\ss}encharaktere to modulus $\mf{m}$ with frequency $\ell$, and again we have abused notation by writing $c^{S_1,S_2}_{\psi}$ for $c^{S_1,S_2}_{\chi}$, where $\chi$ is the Dirichlet character associated to $\psi$.

For use later on note that we have the bound
 \begin{equation}\label{c-chi-bd} |c^{S_1,S_2}_{\psi}| \le c^{S_1,S_2}_{\psi_0} \le 1 ,\end{equation} where $\psi_0$ is the principal character; this follows from the proof of \cref{r-to-dirichlet}, $c^{S_1,S_2}_{\chi}$ being the average of $1_{\gamma \in \Gamma(S_1, S_2)} \overline{\chi(\gamma)}$ over $\gamma \in (\O_K/\mf{m})^*$. The right-hand bound in \cref{c-chi-bd} is crude and will be improved later.

\subsection{Contribution of the principal character}\label{nstar-12}
From this point on we will use the quantities $n_*$ (squarefree part of $n$), $r$ (defined to be $\sqrt{n/n_*}$) and $\omega \in \{\frac{1}{2}, 1\}$ (depending on the value of $n_*$ mod $4$, see \cref{omega-def}).

We now turn to the key computation which will lead to the constant in our asymptotic formula, which is that of the coefficient of the principal character $\psi_0$ to modulus $\mf{m}$, namely
\begin{equation}\label{princ-char} \sigma(S_1, S_2) := |\O^*_K| h_K^{-1} c^{S_1,S_2}_{\psi_0}.\end{equation} 

Note that this section is only relevant in the case $\ell = 0$, since if $\ell \neq 0$ then the sum \cref{cutoff-in-Heckes} does not contain the principal character. Assume, then, that $\ell = 0$.
The idea is to sum \cref{r-to-dirichlet} for $\gamma$ in a set of representatives for $(\O_K/\mf{m})^*$. By orthogonality of Dirichlet characters, this implies that 
\begin{equation}\label{chi0-first} \# \{ \gamma \in \Gamma(S_1, S_2) /\mf{m} , \;\gcd(\gamma, \mf{m}) = 1 \} = |(\O_K/\mf{m})^*| c_{\psi_0}^{S_1,S_2}.\end{equation}
We have (recalling the definition \cref{rd1d2-def} of $\Gamma(S_1, S_2)$)
\begin{equation}\label{gamma-s1s2-again}  \Gamma(S_1, S_2)  = \{ \alpha = x + ry \sqrt{-n_*} : x, y \in \Z, P_{S_1} \mid x, P_{S_2} \mid y\} .\end{equation}
Using this we can almost immediately observe that, in certain cases, $\sigma(S_1, S_2)$ is zero.  In the following lemma, and for the rest of the section, set
\[ T := \{ p \le Q : p \mid n\}.\]

\begin{lemma}\label{es1s2-small-lemma}
Suppose that $\sigma(S_1, S_2) \neq 0$. Then $S_1 \cap S_2  = S_1 \cap T = \emptyset$.
\end{lemma}
\begin{proof}
If $p \in S_1 \cap S_2$ then clearly $p$ divides both $\mf{m}$ and every element of $\Gamma(S_1, S_2)$, so in this case we have $\sigma(S_1, S_2) = 0$ by \cref{princ-char,chi0-first}.

Alternatively suppose $p \in S_1 \cap T$.  We have either (i) $p \mid r$ or (ii) $p \mid n_*$ (or both). In case (i), $p$ divides both $\mf{m}$ and every element of $\Gamma(S_1, S_2)$, and we conclude as before. In case (ii), $p$ ramifies (since $n_*$ divides the discriminant $\Delta$; see \cref{section2}), thus $p = \mf{p}^2$ for some prime ideal $\mf{p}$. We must then have $\mf{p} \mid \sqrt{-n_*}$. Since $\mf{p} \mid p \mid P_{S_1}$, we see that $\mf{p}$ divides both $\mf{m}$ and every element of $\Gamma(S_1, S_2)$, and we conclude as before.
\end{proof}

Assume henceforth that $S_1 \cap S_2 = S_1 \cap T = \emptyset$. To progress in this case, we must evaluate the two sides of \cref{chi0-first}. Looking first at the right-hand side, we have 
\begin{equation}\label{chinese-rem}
 |(\O_K/\mf{m})^*|  =  (N\mf{m})  \prod_{\mf{p} \mid \mf{m}} \Big(1 - \frac{1}{N\mf{p}}\Big) = (N\mf{m}) \prod_{p \mid \mf{m}} \Big(1 - \frac{1}{p}\Big) \bigg(1 - \frac{\legendre{\Delta}{p}}{p}\bigg),
 \end{equation} where $\mf{p}$ ranges over the distinct prime ideals dividing $\mf{m}$, and $p$ ranges over the distinct rational primes dividing $\mf{m}$. (To see the second expression, consider the possible splitting types of $p$, with reference to the remarks in \cref{section2}.)
 
Using the fact that $\O_K = \Z[\sqrt{-n_*}]$ in the case $\omega = 1$ and $\Z[\frac{1}{2}(1 + \sqrt{-n_*})]$ in the case $\omega = \frac{1}{2}$, one may check that a complete set of representatives for $\Gamma(S_1, S_2)$ mod $\mf{m}$ is
\[ \Gamma(S_1, S_2) = \{ \alpha = x + ry \sqrt{-n_*} : 0 \le x < 2n P_{S_1} P_{S_2}, \; 0 \le y < 2\omega  nP_{S_1} P_{S_2}/ r, \; P_{S_1} \mid x, \; P_{S_2} \mid y\},\] or equivalently
\begin{equation}\label{gs1s2-explicit} \Gamma(S_1, S_2) = \{ \alpha = aP_{S_1} + br P_{S_2} \sqrt{-n_*} : 0 \le a < 2n P_{S_2}, \; 0 \le b < 2\omega n P_{S_1}/r\}.\end{equation}

\begin{lemma}
Suppose that $S_1 \cap S_2 = S_1 \cap T = \emptyset$. Then the elements of \cref{gs1s2-explicit} having a common factor with $\mf{m}$ are precisely:
\begin{enumerate}
\item Those with $p \mid a$ for some $p \in S_2 \cup T$;
\item Those with $p \mid b$ for some $p \in S_1$.
\end{enumerate}
\end{lemma}
\begin{remark}
Note that, in our main theorem, $n$ \emph{is} even. 
\end{remark}
\begin{proof}
The set of primes factors of $\mf{m}$ is $S_1 \cup S_2 \cup T$ (here we use the assumption that $n$ is even in \cref{main-term-sharps}).

Suppose that $\mf{p}$ is a prime ideal common factor of $\alpha = aP_{S_1} + br P_{S_2} \sqrt{-n_*}$ and $\mf{m}$. Let $p$ be the rational prime above $\mf{p}$. Then we have exactly one of (i) $p \in S_2 \cup T$ or  (ii) $p \in S_1$. 

In case (i), we break in further cases based on whether $p\in S_2$, $p \mid n_{\ast}$, or $p \mid r$. Suppose first that $p \in S_2$. Then $\mf{p} \mid p \mid P_{S_2}$, whilst $\mf{p}$ is coprime to $P_{S_1}$ since $S_1, S_2$ are disjoint. Thus the condition that $\mf{p} \mid \alpha$ is then precisely that $\mf{p} \mid a$, or equivalently $p \mid a$. Next suppose (still in case (i)) that $p \mid n_*$. Then $p$ is ramified and $p = \mf{p}^2$, and $\mf{p} \mid \sqrt{-n_*}$. Thus $\mf{p} \mid a P_{S_1}$ which, since we are assuming $p \notin S_1$, means that $\mf{p} \mid a$ and hence $p \mid a$ is again the condition.
Finally suppose (still in case (i)) that $p \mid r$, so $\mf{p} \mid r$. Then, since $\mf{p} \nmid P_{S_1}$, we immediately see that the condition is again $p \mid a$.

In case (ii), since $p \in S_1$ we have $\mf{p} \mid P_{S_1}$ and so $\mf{p} \mid br P_{S_2}\sqrt{-n_*}$. Now $\mf{p}$ is coprime to $r$ (since $S_1 \cap T = \emptyset$), to $P_{S_2}$ (since $S_1 \cap S_2 = \emptyset$) and to $\sqrt{-n_*}$ (since otherwise $p \mid n$, as explained above, but this contradicts $S_1 \cap T = \emptyset$).
Therefore $p \mid b$, as claimed.

All the above implications are reversible, as may easily be checked.
\end{proof}

As a consequence of this lemma and \cref{gs1s2-explicit} it follows that 
\begin{equation}\label{lhs-local} \# \{ \gamma \in \Gamma(S_1, S_2) /\mf{m} , \;\gcd(\gamma, \mf{m}) = 1 \}
 = \frac{4 \omega n^2 P_{S_1} P_{S_2} }{ r} \prod_{p \mid \mf{m}} \Big(1 - \frac{1}{p}\Big) = \frac{ \omega N\mf{m}}{ rP_{S_1} P_{S_2}} \prod_{p \mid \mf{m}} \Big(1 - \frac{1}{p}\Big).\end{equation}
 Comparing \cref{lhs-local,chinese-rem,chi0-first,princ-char} gives
 \begin{equation}\label{sig-first} \sigma(S_1, S_2) = \frac{\omega |\O^*_K|}{ r h_K P_{S_1} P_{S_2}} \prod_{p \mid \mf{m}} \bigg( 1 - \frac{\legendre{\Delta}{p}}{p}\bigg)^{-1}.\end{equation}
 It is convenient to write this as

\begin{equation}\label{sigma-s1-s2} \sigma(S_1, S_2) = \frac{ \omega |\O^*_K| P_{T \setminus S_2} }{ rh_K} \prod_{p \in T \cup S_1 \cup S_2}\nu(p),\end{equation}
where 
\begin{equation}\label{nu-def} \nu(p) := (p - \legendre{\Delta}{p})^{-1},
\end{equation}
and recall that this is under the assumption that $S_1 \cap S_2 = S_1 \cap T = \emptyset$ (otherwise \cref{es1s2-small-lemma} applies).

\subsection{Applying the prime ideal theorem}
Now we return to our main task of proving \cref{main-term-sharps-weighted}. Recall the expansion \cref{in-ex} and the definition \cref{es1s2-def} of the terms $E(S_1, S_2)$ appearing there, that is to say 
\begin{equation}\label{817-rpt}
E(S_1, S_2) := \sum_{\mf{a} : N\mf{a} \le X} \varphi_{P_{S_1}} \boxtimes_{\ell} \varphi_{P_{S_2}}(\mf{a}) \Lambda_K(\mf{a}).\end{equation}
Moreover, in \cref{cutoff-in-Heckes} we obtained an expansion for $\varphi_{P_{S_1}} \boxtimes_{\ell} \varphi_{P_{S_2}}(\mf{a})$ in characters, in the case that $\mf{a}$ is coprime to $\mf{m} = \mf{m}^{S_1, S_2} = (2n P_{S_1} P_{S_2})$. Combining this with \cref{817-rpt} gives
\[
    E(S_1, S_2)  = \frac{|\O^*_K|}{h_K} \sum_{\psi} c^{S_1, S_2}_{\psi} \sum_{\substack{\mf{a}: N\mf{a} \le X \\ \gcd(\mf{a}, \mf{m}^{S_1,S_2}) = 1}} \psi(\mf{a})\Lambda_K(\mf{a})  + O(\log X),
\]
where the sum is over Gr\"o{\ss}encharaktere $\psi$ to modulus $\mf{m}$ with frequency $\ell$, and where the $O(\log X)$ error term comes from those $\mf{a}$ which are not coprime to $\mf{m}^{S_1, S_2}$, the contribution of these being
\[ \ll \sum_{\mf{p}^j \mid \mf{m}^{S_1,S_2}}\log N\mf{p} \ll \log X.\]

In the following $c_1, c_2, c_3 > 0$ are absolute constants. Applying \cref{main-analytic-input} gives

\begin{equation}\label{es-1} E(S_1, S_2) =  \frac{|\O^*_K|}{h_K} \bigg( X c_{\psi_0}^{S_1, S_2}  -  \frac{X^{\beta}}{\beta}  c^{S_1,S_2}_{\psi^*} \bigg) + O \bigg( X e^{-c_1\sqrt{\log X}}\Big(1 +  \sum_{\psi} |c^{S_1, S_2}_{\psi}|\Big) \bigg), \end{equation}
where here all Gr\"o{\ss}encharaktere $\psi$ are to modulus $\mf{m} = \mf{m}^{S_1, S_2}$ and have frequency $\ell$, $\psi_0$ is the principal character to modulus $\mf{m}$ (which is only present if $\ell = 0$) and the `Siegel term'  (the one involving $X^{\beta}/\beta$) is only present if $\ell = 0$ and $\mf{d}^* \mid \mf{m}$, which is the condition for there to be some Gr\"o{\ss}encharakter $\psi_*$ to modulus $\mf{m}$ with frequency $\ell$ which is induced from the exceptional charakter $\chi_{\mf{d}^*}$. 

In particular, the whole main term here is absent unless $\ell = 0$.

Recall (see \cref{princ-char}) that we gave the name $\sigma(S_1, S_2)$ to the coefficient $|\O^*_K| h_K^{-1} c^{S_1, S_2}_{\psi_0}$ appearing here, and derived expressions for it (\cref{es1s2-small-lemma,sigma-s1-s2}). Thus, from \cref{es-1} (and bounding the error term there fairly trivially using \cref{c-chi-bd} and the bound $N\mf{m} = (2n P_{S_1} P_{S_2})^2  \ll \exp (O(\log^{1/8} X))$), we derive from \cref{es-1} that 
\begin{equation}\label{es-2} E(S_1, S_2) =   X\sigma(S_1, S_2)   -  \frac{|\O_K^*|}{h_K}\frac{X^{\beta}}{\beta} c^{S_1,S_2}_{\psi^*}  + O \big( X e^{-c_2\sqrt{\log X}}\big). \end{equation}
Now it is time to substitute this into the formula \cref{in-ex} for the expression of interest in \cref{main-term-sharps-weighted}. Doing this, we obtain

\begin{equation} \sum_{\mf{a} : N\mf{a} \le X}  \Lambda_{\cramer}^{\sharp}\boxtimes_{\ell}  \Lambda_{\cramer}^{\sharp}(\mf{a}) \Lambda_K(\mf{a})  = E_{\main} + E_{\siegel} + E_{\textnormal{error}},\label{lam-sharp-three-terms} \end{equation}
 where
 \begin{equation}\label{t-main-def} E_{\main} := X \prod_{p \in \mc{P}_Q} \Big(1 - \frac{1}{p}\Big)^{-2} \sum_{\substack{S_1, S_2 \subseteq \mc{P}_Q \\ |S_1|, |S_2| \le t }} (-1)^{|S_1| + |S_2|} \sigma(S_1, S_2), \end{equation}

 \begin{equation}\label{t-siegel-def} E_{\siegel} \ll (\log Q)^2 X^{\beta} \sum_{\substack{S_1, S_2 \subseteq \mc{P}_Q \\ |S_1|, |S_2| \le t }}  \sum_{\substack{\psi^* \mdsub{\mf{m}^{S_1, S_2}} \\ \psi_* \sim \chi_{\mf{d}^*}}} |c_{\psi^*}^{S_1, S_2}|, \end{equation}
 and
\begin{equation}\label{t-error} E_{\textnormal{error}} \ll X(\log Q)^2 \binom{|\mc{P}_Q|}{t}^2 e^{-c_2 \sqrt{\log X}} \ll Xe^{-c_3 \sqrt{\log X}} . \end{equation} 
The term $E_{\textnormal{error}}$ is acceptable in \cref{main-term-sharps-weighted}. This in fact completes the proof of \cref{main-term-sharps-weighted} in the case $\ell \neq 0$, since in this case both $E_{\main}$ and $E_{\siegel}$ are absent.

Suppose from now on that $\ell = 0$. In the remaining sections, we evaluate $E_{\main}$ asymptotically, and show that $E_{\siegel}$ is small; this will then conclude the proof of \cref{main-term-sharps-weighted} in all cases.

\subsection{Estimating the Siegel term}

In this subsection we estimate the Siegel term $E_{\siegel}$. Write $d_* := N\mf{d}^*$. First note that a pair $S_1, S_2$ of subsets of $\mc{P}_Q$ only contributes to the sum if $d_* \mid (2n P_{S_1} P_{S_2})^2 = N\mf{m}^{S_1, S_2}$, since $\mf{d}^* \mid \mf{m}$ for any exceptional modulus $\mf{m}$. 

To bound the contribution of these terms, we first upgrade the crude bound \cref{c-chi-bd} using the knowledge about $c^{S_1,S_2}_{\psi_0}$ gained in previous sections. Specifically, from \cref{c-chi-bd,princ-char,sig-first}  we have
\[ |c_{\psi^*}^{S_1, S_2}| \le c_{\psi_0}^{S_1, S_2} \ll \frac{\log Q}{P_{S_1} P_{S_2}}.\] It follows that
\begin{equation}\label{siegel-2} E_{\siegel} \ll (\log Q)^3 X^{\beta}  \sum_{\substack{S_1, S_2 \subseteq \mc{P}_Q \\ d_* \mid (2n P_{S_1} P_{S_2})^2}} \frac{1}{P_{S_1} P_{S_2}}. \end{equation}
To bound this sum, note that the condition $d_* \mid (2n P_{S_1} P_{S_2})^2$ implies that $d_1 \mid P_{S_1}$ and $d_2 \mid P_{S_2}$ for some divisors $d_1, d_2$ of $d_*$ with $d_1 d_2 \ge \sqrt{d_*/4n^2}$. Therefore 
\begin{align*} E_{\siegel} &\ll (\log Q)^3 X^{\beta} \tau(d_*)^2  \sup_{\substack{d_1|d_{\ast}, d_2|d_{\ast}\\d_1d_2\ge \sqrt{d_*/4n^2}}}\sum_{\substack{S_1, S_2 \subseteq\mc{P}_Q \\ d_1 \mid P_{S_1}, d_2 \mid P_{S_2}}} \frac{1}{P_{S_1} P_{S_2}} \ll (\log Q)^3 X^{\beta} \tau(d_*)^2 \cdot (\log Q)^2 \cdot d_{\ast}^{-1/2}\\ & \ll (\log X) X^{\beta} d_*^{2\eta -1/2} \ll X (\log X) X^{-C_{\eps} d_*^{-\eps}} d_*^{2\eta -1/2},\end{align*}
where in the last step we applied Siegel's theorem \cref{siegel-thm}, and $\eta > 0$ is any positive constant coming from the divisor bound $\tau(d_*) \ll_{\eta} d_*^{\eta}$.
Therefore, computing the maximum over $d_*$ of the preceding expression, we see that
\begin{equation}\label{final-seigel} E_{\siegel} \ll_{\eta,\eps} X (\log X)^{1 - \frac{\frac{1}{2} - 2\eta}{\eps}} \ll \frac{X}{\log X}\end{equation} no matter the value of $d_*$, taking any $\eps,\eta < \frac{1}{8}$ for the second bound. 

\subsection{Evaluation of \texorpdfstring{$E_{\main}$}{}}

We now come to the final task, which is the asymptotic evaluation of $E_{\main}$, whose definition is \cref{t-main-def}. First, note that by \cref{es1s2-small-lemma} we may restrict the sum to $S_1 \cap S_2 = S_1 \cap T = \emptyset$, that is to say

\begin{equation}\label{t-main-def-rest} E_{\main} := X \prod_{p \in \mc{P}_Q} \Big(1 - \frac{1}{p}\Big)^{-2} \sum_{\substack{S_1, S_2 \subseteq \mc{P}_Q \\ S_1 \cap S_2 = S_1 \cap T = \emptyset \\ |S_1|, |S_2| \le t }} (-1)^{|S_1| + |S_2|} \sigma(S_1, S_2) .\end{equation}

For the remaining terms in the sum we have the formula \cref{sigma-s1-s2} and the accompanying definition \cref{nu-def} of $\nu(p)$, which in particular gives the bound
\begin{equation}\label{sig-s-bd} \sigma(S_1, S_2) \ll \prod_{p \in S_1 \cup S_2} \frac{1}{p-1} .\end{equation}
Using this, we show using another Brun sieve-type computation that we may remove the constraints $|S_1|, |S_2| \le t$ in \cref{t-main-def-rest} with little penalty. Indeed, by \cref{sig-s-bd} the error in doing this is 
\begin{align*} \ll X (\log Q)^2  \sum_{\substack{S_1, S_2 \subseteq \mc{P}_Q \\ |S_2| > t \\ S_1 \cap S_2 = \emptyset}} \prod_{p \in S_1 \cup S_2} \frac{1}{(p-1)} & \ll X(\log Q)^2 \prod_{p \in \mc{P}_Q} \Big(1 + \frac{1}{p-1}\Big) \cdot \sum_{\substack{S \subseteq \mc{P}_Q \\ |S| > t}} \prod_{p \in S}\frac{1}{p-1} \\ & \ll X (\log Q)^3 \sum_{\substack{S \subseteq \mc{P}_Q \\ |S| > t}} \prod_{p \in S}\frac{1}{p-1} \ll X(\log X)^{-10},\end{align*} where the last step follows from \cref{large-t-bd}. Thus
\begin{equation}\label{t-main-complete} E_{\main} = X \prod_{p \in \mc{P}_Q} \Big(1 - \frac{1}{p}\Big)^{-2} \sum_{\substack{S_1, S_2 \subseteq \mc{P}_Q \\ S_1 \cap S_2 = S_1 \cap T = \emptyset }} (-1)^{|S_1| + |S_2|} \sigma(S_1, S_2)  + O(X (\log X)^{-10}).\end{equation}
The main expression here is of a rather algebraic nature and we can evaluate it using \cref{sigma-s1-s2}. To do this, set
\[ S'_1 := S_1, \quad S'_2 := S_2 \setminus T,\quad S'_3 := T \setminus S_2, \quad S'_4 := T \cap S_2.\] Note that these sets are all disjoint and that $S'_2, S'_4$ give a partition of $S_2$, whilst $\bigcup_{i = 1}^4 S'_i$ is a partition of $S_1 \cup S_2 \cup T$. Then we obtain from \cref{sigma-s1-s2} that 
\begin{align}\nonumber
\frac{rh_K}{\omega|\O^*_K|} \sum_{\substack{S_1, S_2 \subseteq \mc{P}_Q \\ S_1 \cap S_2 = S_1 \cap T = \emptyset}} (-1)^{|S_1| + |S_2|}\sigma(S_1, S_2) & = \sum_{\substack{S'_1, S'_2 \subseteq \mc{P}_Q \setminus T, S'_1 \cap S'_2 = \emptyset \\ S'_3 \subseteq T, S'_4 = T \setminus S'_3 }} (-1)^{|S'_1| + |S'_2| + |S'_4|} P_{S'_3}\prod_{p \in \bigcup_{i=1}^4 S'_i} \nu(p) \\ & = \prod_{p \in \mc{P}_Q \setminus T} (1 - 2 \nu(p)) \cdot \prod_{p \in T} (p-1)\nu(p) \nonumber \\ & = \nu(2) \cdot \prod_{3 \le p \le Q} (1 - 2 \nu(p)) \cdot \prod_{\substack{p \mid n \\ p \ge 3}}
 \frac{(p-1)\nu(p)}{1 - 2 \nu(p)}.\label{big-prod-form}\end{align}
 (We are using here the fact that $n$ is even, so $2 \in T$; note that $1 - 2 \nu(p) \ne 0$ for $p \ge 3$.)
 
 Therefore from \cref{t-main-complete} we have
\begin{equation}\label{T-main-abs-conv} E_{\main} = X \frac{\omega rh_K}{|\O^*_K|} \cdot 4 \nu(2) \cdot  \prod_{3 \le p \le Q}\Big(1 - \frac{1}{p}\Big)^{-2} (1 - 2 \nu(p)) \cdot \prod_{\substack{p \mid n \\ p \ge 3}} \frac{(p-1)\nu(p)}{1 - 2 \nu(p)} + O(X(\log X)^{-10}).\end{equation}
The product over $p$ here is absolutely convergent, with each term having size $1 - O(p^{-2})$, and so we may extend it to all $p$ with a multiplicative loss of $1 + O(1/Q)$, the effect of which may comfortably be absorbed into the $O(X (\log X)^{-10})$ error term.

Finally, we apply the class number formula \cref{class-number}, which in our notation gives
\begin{equation}\label{class-number-truncate} \frac{2\pi h_K}{|\O_K^*| |\Delta|^{1/2}} =  \prod_{p} p \nu(p).\end{equation} This, the formula $\Delta = -4\omega^2 n_*$ and \cref{T-main-abs-conv} combine to give

 \begin{equation}\label{tmain-evaluated}
 E_{\main} =  X\frac{\pi}{\sqrt{n}}\cdot 2 \cdot \bigg(\prod_{3 \le p \le Q} \Big(1 - \frac{1}{p}\Big)^{-2} \frac{1 - 2 \nu(p)}{p \nu(p)} \bigg) \cdot    \prod_{\substack{p \mid n \\ p \ge 3}}
 \frac{(p-1)\nu(p)}{1 - 2 \nu(p)} + O(X (\log X)^{-10}).
 \end{equation}
 We claim that the main expression here is precisely $\frac{\pi}{\sqrt{n}}\kappa_n X$ with $\kappa_n$ as in \cref{kappa-def}, which is exactly the constant in \cref{prop:main}. We verify this one prime $p$ at a time (the contribution of the prime $p = 2$ being $2$). Here, once more, we use the fact that $n$ is even. 

 If $\legendre{-n}{p} = 1$ then $\legendre{\Delta}{p} = 1$ and $p \nmid n$, so the contribution from $p$ is
 \[ \Big(1 - \frac{1}{p}\Big)^{-2} \frac{1 - 2 \nu(p)}{p \nu(p)} = \frac{p(p-3)}{(p-1)^2}.\] 
 If $\legendre{-n}{p} = -1$ then $\legendre{\Delta}{p} = -1$ and $p \nmid n$, so the contribution from $p$ is
 \[ \Big(1 - \frac{1}{p}\Big)^{-2} \frac{1 - 2 \nu(p)}{p \nu(p)} = \frac{p}{p-1}.\] 
 If $\legendre{-n}{p} = 0$ then $p \mid n$, and in this case we also get a contribution from $p$ of $p/(p-1)$, both when $p \ge 3$ and $p = 2$.
 This verifies the claim, and therefore
 \[ E_{\main} = \frac{\pi \kappa_n}{\sqrt{n}}X  + O(X (\log X)^{-10}).\]
Combining the preceding estimate with \cref{final-seigel,t-error,lam-sharp-three-terms}, we obtain the conclusion of \cref{main-term-sharps-weighted} in the case $\ell = 0$. Recall that the case $\ell \neq 0$ has already been established. \vspace*{8pt}

This concludes the proof of all the statements in the paper. 

\appendix
\section{Properties of the Gowers and Gowers--Peluse norms}\label{appendixA}

In this section, we collect some basic properties of Gowers--Peluse norms which were defined in \cref{def:box}. Let $g : \Z \rightarrow \C$ be a finitely-supported function. Recall that $\Delta_h g(x) = g(x) \overline{g(x+h)}$ and that $\Delta_{(h,h')} g(x) = g(x + h) \overline{g(x+h')}$ as in \cref{difference} and \cref{double-difference} respectively. We begin by noting that
\begin{equation}\label{gowers-pos}
\sum_{x,h \in \Z} \Delta_h g(x) = \Big| \sum_x g(x) \Big|^2 \ge 0
\end{equation}
and that, for any probability measure $\mu$ on $\Z$, 
\begin{equation}\label{gowers-peluse-pos}
  \E_{h,h' \sim \mu}  \sum_{x \in \Z} \Delta_{(h,h')}g(x) = \sum_{x\in \Z} \Big| \E_{h \sim \mu} g(x + h) \Big|^2 \ge 0.
\end{equation}

We first record that the Gowers--Peluse norm is nonnegative and that we have a version of the Gowers--Cauchy--Schwarz inequality. To the latter end, if we have a collection $(f_{\omega})_{\omega \in \{0,1\}^k}$ of functions, we define the Gowers--Peluse inner product
\begin{equation}\label{gpip} \langle (f_{\omega})_{\omega \in \{0,1\}^k}\rangle_{U_{\GP}[N;\mu_1,\dots, \mu_k]} := \frac{1}{N} \sum_{x \in \Z} \E_{h_i, h'_i \sim \mu_i} \prod_{\omega \in \{0,1\}^k} \mc{C}^{|\omega|} f_{\omega} (x + (\mathbf{1} - \omega)\cdot h + \omega \cdot h'),\end{equation} where here $\mathbf{1} - \omega = (1 - \omega_i)_{i = 1}^k$ and $h = (h_i)_{i = 1}^k$, $h' = (h'_i)_{i = 1}^k$. Note in particular that if $f_{\omega} = f$ for all $\omega$ then
\[ \langle (f_{\omega})_{\omega \in \{0,1\}^k}\rangle_{U_{\GP}[N;\mu_1,\dots, \mu_k]} = \Vert f \Vert_{U_{\GP}[N;\mu_1,\dots, \mu_k]}^{2^k},\] with notation as in \cref{def:box}. 

\begin{lemma}\label{lem:GCS}
Let $\mu_1,\ldots,\mu_k : \Z \rightarrow [0,1]$ be probability measures, let $N\ge 1$, and suppose that $f : \Z \rightarrow \C$ is a function. Then for any function $f : \Z \rightarrow \C$ we have
\begin{equation}\label{GP-pos}\snorm{f}_{U_{\GP}[N;\mu_1,\ldots,\mu_k]}^{2^k} \ge 0.\end{equation}
Moreover if $(f_{\omega})_{\omega \in \{0,1\}^k}$ are functions from $\Z$ to $\C$ then we have that 
\begin{equation}\label{gpcs} \langle (f_{\omega})_{\omega \in \{0,1\}^k}\rangle_{U_{\GP}[N;\mu_1,\dots, \mu_k]} \le  \prod_{\omega \in \{0,1\}^k} \Vert f_{\omega} \Vert_{U_{\GP}[N;\mu_1,\dots, \mu_k]}.\end{equation}

\end{lemma}
\begin{proof}
Statement \cref{GP-pos} is immediate from \cref{gowers-peluse-pos} and the definition \cref{gp-explicit-def}, taking 
\[ g = \E_{\substack{h_i, h'_i \sim \mu_i \\ 1 \le i \le k-1}} \Delta_{(h_1, h'_1)} \cdots \Delta_{(h_{k-1}, h'_{k-1})} f,\qquad \mu = \mu_k,\] in \cref{gowers-peluse-pos}.
For \cref{gpcs}, we proceed by induction on $k$. First note that by splitting $\omega = (\tilde\omega, \eps)$ with $\tilde\omega \in \{0,1\}^{k-1}$ and $\eps \in \{0,1\}$, and writing $\tilde{h} = (h_i)_{i=1}^{k-1}$ and $\tilde{h}' = (h'_i)_{i=1}^{k-1}$,
we may write the Gowers--Peluse inner product $\langle (f_{\omega})_{\omega \in \{0,1\}^k}\rangle^{2^k}_{U_{\GP}[N;\mu_1,\dots, \mu_k]}$   as \[ \frac{1}{N} \cdot \sum_{x\in \Z}\E_{\substack{h_i,h_i'\sim \mu_i\\1\le i\le k-1}}\prod_{\eps \in \{0,1\}}\bigg(\E_{h_k\sim \mu_k}\prod_{\tilde\omega\in\{0,1\}^{k-1}} \mc{C}^{|\tilde{\omega}| + \eps} f_{(\tilde\omega,\eps)}(x + (\mathbf{1}-\tilde\omega)\cdot \tilde{h} + \tilde\omega\cdot \tilde{h}' +  h_k)\bigg). \]
By Cauchy--Schwarz, this is bounded above by
\[ \prod_{\eps \in\{0,1\}}\bigg(\frac{1}{N} \cdot\sum_{x\in \Z}\E_{\substack{h_i,h_i'\sim \mu_i\\1\le i\le k-1}}\Big|\E_{h_k\sim \mu_k}\prod_{\tilde{\omega}\in\{0,1\}^{k-1}}\mc{C}^{|\tilde{\omega}| + \eps}f_{(\tilde{\omega},\eps)}(x + (\mathbf{1}-\tilde{\omega})\cdot \tilde{h} + \tilde{\omega}\cdot \tilde{h'} + h_k)\Big|^{2}\bigg)^{1/2},\] which upon expanding the square and introducing a dummy variable $h'_k$ may be rewritten as 
\[ \prod_{\eps \in\{0,1\}}\bigg(\frac{1}{N} \cdot\sum_{x\in \Z}\E_{\substack{h_i,h_i'\sim \mu_i\\1\le i\le k-1}}\E_{h_k,h_k' \sim \mu_k}\prod_{\tilde{\omega}\in\{0,1\}^{k-1}}\Delta_{(h_k,h_k')}\mc{C}^{|\tilde{\omega}|} f_{(\tilde{\omega},\eps)}(x + (\mathbf{1}-\tilde{\omega})\cdot \tilde{h} + \tilde{\omega}\cdot \tilde{h'})\bigg)^{1/2}.\]
By induction (that is, by \cref{gpcs} in the case $k - 1$) this is bounded by
\[ \prod_{\eps \in\{0,1\}} \bigg(\E_{h_k,h_k'\sim \mu_k}\prod_{\tilde{\omega}\in \{0,1\}^{k-1}}\snorm{\Delta_{(h_k,h_k')}f_{(\tilde{\omega},\eps)}}_{U_{\GP}[N;\mu_1,\ldots,\mu_{k-1}]}\bigg)^{1/2} . \] Using H\"older's inequality, this is at most
\[ \prod_{\eps \in\{0,1\}} \prod_{\tilde{\omega}\in \{0,1\}^{(k-1)}}\bigg(\E_{h_k,h_k' \sim \mu_k}\snorm{\Delta_{(h_k,h_k')}f_{(\tilde{\omega},\eps)}}_{U_{\GP}[N;\mu_1,\ldots,\mu_{k-1}]}^{2^{k-1}}\bigg)^{1/{2^{k}}},\] which equals
\[ \prod_{\eps \in\{0,1\}} \prod_{\tilde{\omega}\in \{0,1\}^{(k-1)}} 
 \big\Vert f_{(\tilde{\omega},\eps)} \big\Vert_{U_{\GP}[N;\mu_1,\ldots,\mu_{k}]} = \prod_{\omega \in \{0,1\}^k} \Vert f_{\omega} \Vert_{U_{\GP}[N;\mu_1,\dots, \mu_k]}.\] 
 This concludes the proof.
 \end{proof}

The next lemma asserts a kind of monotonicity of Gowers--Peluse norms with respect to the sequence of measures. It is essentially identical to \cite[Lemma~3.5 (iii)]{KKL24}, but formulated using measures.
\begin{lemma}\label{lem:prop-box}
Fix $\delta\in (0,1/2)$, $N\ge 1$ and let $f:\Z\to\C$ be a $1$-bounded function supported on $[\pm N]$.
Let $\mu_1,\ldots,\mu_k$ be probability measures supported on $[\pm N]$. Then
    \[\snorm{f}_{U_{\GP}[N;\mu_1,\ldots,\mu_k]}^{2^k}\ge \frac{1}{2k+3} \snorm{f}_{U_{\GP}[N;\mu_1,\ldots,\mu_{k-1}]}^{2^{k}}.\]\end{lemma}
\begin{proof}
We have
\[ \snorm{f}_{U_{\GP}[N;\mu_1,\ldots,\mu_k]}^{2^k} =\frac{1}{N} \E_{\substack{h_i,h_i'\sim \mu_i\\1\le i\le k-1}}\sum_{x\in \Z}\Big|\E_{h_k\sim \mu_k}\Delta_{(h_1,h_1')}\cdots \Delta_{(h_{k-1},h_{k-1}')}f(x+h_k)\Big|^2.\]
Observe that, since $f$ is supported on $[\pm N]$, the sum over $x$ is supported on $|x| \le (k+1) N$. By Cauchy--Schwarz, it follows that
\begin{align*}
 \snorm{f}_{U_{\GP}[N;\mu_1,\ldots,\mu_k]}^{2^k}
&\ge ((2k+3) N^2)^{-1} \E_{\substack{h_i,h_i'\sim \mu_i\\1\le i\le k-1}}\Big|\sum_{x\in \Z}\E_{h_k\sim \mu_k}\Delta_{(h_1,h_1')}\cdots \Delta_{(h_{k-1},h_{k-1}')}f(x+h_k)\Big|^2\\
&= ((2k+3) N^2)^{-1}  \E_{\substack{h_i,h_i'\sim \mu_i\\1\le i\le k-1}}\Big|\sum_{x\in \Z}\Delta_{(h_1,h_1')}\cdots \Delta_{(h_{k-1},h_{k-1}')}f(x)\Big|^2\\
&\ge (2k+3)^{-1} \Big| \frac{1}{N}\cdot \E_{\substack{h_i,h_i'\sim \mu_i\\1\le i\le k-1}} \sum_{x\in \Z}\Delta_{(h_1,h_1')}\cdots \Delta_{(h_{k-1},h_{k-1}')}f(x)\Big|^2\\
& = (2k+3)^{-1}  \snorm{f}_{U_{\GP}[N;\mu_1,\ldots,\mu_{k-1}]}^{2^{k}}.
\end{align*}
\end{proof}

\section{Proof of concatenation estimates}\label{appendixB}

We give a complete proof of \cref{lem:concat-main} in this section. Our proof follows that in \cite[Section~6]{KKL24}, which in turn closely follows \cite[Section 2]{kuca}, with certain modifications. The reader may wish to review the basic notation for probability measures in \cref{sec:notation}. In this section, all probability measures are additionally assumed symmetric, that is to say $\mu(x) = \mu(-x)$. 

Managing notation is problematic in the arguments in this section. To ease matters at least a little, we write $\Vert f \Vert^{\bullet}_{U_{\GP}[N;\Omega]}$, we mean $\Vert f \Vert^{2^d}_{U_{\GP}[N;\Omega]}$, where $d = |\Omega|$. (That is, the Gowers--Peluse norm is always raised to the `obvious' power.)

A convenient piece of notation is to write $\binom{[N]}{r}$ for the collection of $r$-element subsets of $[N]$, and more generally $\binom{[M, N]}{r}$ for the collection of $r$-element subsets of the discrete interval $[M, N]$.

We begin by showing that the general case of \cref{lem:concat-main} follows from the case $m = 2$. This is a minor modification of \cite[Corollary 2.5]{kuca} or \cite[Corollary 6.4]{KKL24}.

\begin{proof}[Proof of \cref{lem:concat-main}, assuming the case $m = 2$]
For the proof, it is convenient to restate the hypothesis and conclusion of \cref{lem:concat-main} in a slightly more abstract form. We define a \emph{measure tuple} $\mc{T} = (\Omega, I, s)$ to be the following data: $I$ is an indexing set, $s \ge 1$ is an integer, and for each $i \in I$, we have a sequence $\Omega_i = (\mu_{1i},\dots, \mu_{si})$ of symmetric probability measures on $\Z$. Given such a tuple, and given $t, m \ge 1$, we define the tuple $\wedge^{m, t} \mc{T} = (\Omega', I', s')$, where $I' = I^t$, $s' = s \binom{t}{m}$, and 
\[ \Omega'_{i_1,\dots, i_t} := \big(\conv_{r \in R} \mu_{ji_r}\big)_{j \in [s], R \in \binom{[t]}{m}},\] where we place an arbitrary order here in order to create a sequence rather than a just a set. 

With this notation, the hypothesis of \cref{lem:concat-main} is that we have a measure tuple $\mc{T} = (\Omega, I, s)$, with all the measures in each $\Omega_i$ supported on $[\pm N]$, and a $1$-bounded function $f : \Z \rightarrow \C$ such that \begin{equation}\label{i-dash-hyp} \E_{i\in I}\snorm{f}^{\bullet}_{U_{\GP}[N;\Omega_i]}\ge \delta.\end{equation}
The conclusion for a given value of $m$, which we call \cref{lem:concat-main}$(m)$, is that there exists $t = t_0(m, s)$ such that, if we write $\mc{T}' = \wedge^{m, t}\mc{T} = (\Omega', I', s')$, we have
\[\E_{i' \in I'}\snorm{f}_{ U_{\GP}[N; \Omega'_{i'}]}
^{\bullet}\ge \delta^{O_{m,s}(1)}.\] For technical reasons which will be apparent shortly, we work with the equivalent statement
\begin{equation}\label{i-dash-conclu}\E_{i' \in I'}\snorm{f}_{ U_{\GP}[m N; \Omega'_{i'}]}
^{\bullet}\ge \delta^{O_{m,s}(1)}.\end{equation} 
(The equivalence is clear from the definition of the Gowers--Peluse norm \cref{gp-explicit-def}; the factor of $m$ may be absorbed into a power $\delta^{O_{m}(1)}$.)

Suppose we have established \cref{lem:concat-main}$(m)$ and \cref{lem:concat-main}$(2)$. Set $t := t_0(m, s)$ and $s' = s \binom{t}{m}$. Thus, we have \cref{i-dash-conclu} for this value of $m$. As can be seen from the similarity in form between \cref{i-dash-hyp,i-dash-conclu}, we can feed this into \cref{lem:concat-main}$(2)$, taking the value $N' = m N$ in this application of the lemma. Note that the underlying measures, which are all $m$-fold convolutions of the $\mu_{ji}$, are all supported on $[\pm N']$ and so the lemma applies.

Setting $t' := t_0(2,s')$, the conclusion is that if we write $\wedge^{2,t'} \mc{T}' = (\Omega'', I'', s'') $ then 
\begin{equation}\label{i-dash-conclu-2}\E_{i'' \in I''}\snorm{f}_{ U_{\GP}[N'; \Omega''_{i''}]}
^{\bullet}\ge \delta^{O_{m,s}(1)}.\end{equation}
Unpacking the definitions, we see that $I'' = (I^t)^{t'}$, and if $i'' = \big((i_u)_v\big)_{u \in [t], v \in [t']} \in I''$ then
\[ \Omega''_{i''} = \Big(\conv_{a \in A} \conv_{b \in B} \mu_{j, (i_{b})_a} \Big)_{j \in [s], A \in \binom{[t']}{2}, B \in \binom{[t]}{m}}.\] 
We enlarge these collections of measures to 
\[ \tilde\Omega''_{i''} =  \big(\conv_{r \in R} \mu_{j, i_r}\Big)_{j \in [s], R \in \binom{[tt']}{2m}}, \] here abusing notation by identifying $(I^t)^{t'}$ with $I^{tt'}$, and noting the natural inclusion $\binom{[t']}{2} \times \binom{[t]}{m} \hookrightarrow \binom{[tt']}{2m}$. By \cref{lem:prop-box}, this enlargement cannot decrease the Gowers--Peluse norms by more than a factor $O_{m, s}(1)$, arising from the inconsequential factors of $2k+3$ in that lemma. Thus, \cref{i-dash-conclu-2} implies that 
\begin{equation}\label{i-dash-conclu-3}\E_{i'' \in I''}\snorm{f}_{ U_{\GP}[N'; \tilde\Omega''_{i''}]}
^{\bullet}\ge \delta^{O_{m,s}(1)}.\end{equation}
However, unpicking the notation we see that $\wedge^{2m, tt'} \mc{T} = (\tilde\Omega'', I'', s'')$, so \cref{i-dash-conclu-3} is precisely the conclusion of \cref{lem:concat-main} with $m$ replaced by $2m$, and with $N$ replaced by $N'$. Much as before, we can trivially change $N'$ to $N$; in fact, this increases the Gowers--Peluse norm.

We have therefore shown that \cref{lem:concat-main}$(m)$ and \cref{lem:concat-main}$(2)$ imply \cref{lem:concat-main}$(2m)$. By induction the whole of \cref{lem:concat-main} (for all powers of two $m$) follows from the (as yet unestablished) base case $m = 2$. Moreover, we have shown that we can take
\begin{equation}\label{b4-iter} t_0(2 m, s) = t_0\bigg(2, s \binom{t_0(m, s)}{m}\bigg) t_0(m, s).
\end{equation}
\end{proof}

\begin{remark}\label{rmk:height}
   We will show below that we can take $t_0(2,s) = 2^s$. In the main part of the paper, the Gowers norm we require is the $U^k$-norm with $k = 2 + 6 \binom{t_0(4,3)}{4}$ ; by repeated use of \cref{b4-iter} one may see that this is roughly on the order of $2^{347}$, that is to say surprisingly close to a googol. If one wanted control of Type II sums in \cref{prop:typeii-to-gowers} up to $L \le X^{1/2 - \kappa}$ using a Gowers $U^k$-norm then, via this argument, $k$ would need to be a tower of twos of height $\sim \log(1/\kappa)$). 
\end{remark}

We are now free to focus exclusively on the case $m = 2$ of \cref{lem:concat-main}, and consequently it is somewhat natural to use the language of graphs. Suppose throughout the following that we have sequences of measures $\Omega_i = (\mu_{1i}, \dots, \mu_{si})$, $i \in I$, as in the statement of \cref{lem:concat-main}.

\begin{definition}\label{graph-system-def}
Let $t \ge 1$ be an integer. Then a \emph{graph system} $\Gamma$ of level $t$ is the following data. For each $j \in [s]$ we have a set $V_j \subseteq [t]$ of vertices and a set $E_j \subseteq \binom{[t]}{2}$ of edges. (The edges do not have to be between the vertices in $V_j$.)  To a graph system of level $t$, and for any tuple $(i_1,\dots, i_t) \in I^t$, we associate the following collection $\Omega^{\Gamma}_{i_1,\dots, i_t}$ of probability measures: $\mu_{ji_v}$ for $v \in V_j$, and $\mu_{ji_v} \ast \mu_{ji_w}$ for $(v,w) \in E_j$, $j = 1,\dots, s$.  
\end{definition}

With this definition, the case $m = 2$ of \cref{lem:concat-main} (with $t_0(2,s) = 2^s$) can be rephrased as follows.

\begin{lemma}\label{main-concat-graph}
Let $\delta \in (0, \frac{1}{2})$, and let $s, I, \Omega_i$ be as above. Let $f : \Z \rightarrow \C$ be a 1-bounded function such that 
\begin{equation}\label{assump-again} \E_{i \in I} \Vert f \Vert^{\bullet}_{U_{\GP}[N; \Omega_i]} \ge \delta. \end{equation}
Then
\begin{equation}\label{completed} \E_{i_1,\dots, i_{M} \in I} \Vert f \Vert^{\bullet}_{U_{\GP}[N; \Omega^{\Gamma}_{i_1,\dots, i_{M}}]} \gg \delta^{O_s(1)}, \end{equation} 
where $M = 2^s$ and $\Gamma$ is the graph system at level $M := 2^s$ such that $V_j = \emptyset$ and $E_j = \binom{[M]}{2}$ for all $j \in [s]$.
\end{lemma}
Thus our remaining task is to prove \cref{main-concat-graph}. The proof idea is as follows: we note that the assumption \cref{assump-again} can be written 
\begin{equation}\label{assump-again-again} \E_{i_1 \in I^1} \Vert f \Vert^{\bullet}_{U_{\GP}[N; \Omega^{\Gamma^1}_{i_1}]} \ge \delta \end{equation} where $V^1_j = \{1\}$ and $E^1_j = \emptyset$. Thus, our aim is to move from \cref{assump-again-again} (with the graphs `vertex-complete' but `edge-empty') to the conclusion \cref{completed} (with the graphs `vertex-empty' but `edge-complete'). We will do this by incrementing the level parameter of the graph system from $1$ to $M = 2^s$ in single steps.

The following key lemma is the driver for this process. A version of this lemma where the $\mu_{si}$ are uniform measures on subgroups in an ambient abelian group appears in the work of Kuca \cite[Lemma~2.2]{kuca}. A slight variant of the precise lemma we prove appears as \cite[Lemma~6.1]{KKL24}.

\begin{lemma}\label{key-dup-step}
Fix $\delta\in (0,1/2)$, $k\ge 1$, and $I$ an indexing set. For each $i \in I$, let $\nu_{1i},\dots, \nu_{ki}$ be symmetric probability measures supported on $[\pm N]$. Let $f:\Z\to \C$ be $1$-bounded such that $\on{supp}(f)\subseteq [\pm N]$. Suppose that
\[\E_{i\in I}\snorm{f}^{\bullet}_{U_{\GP}[N;\nu_{1i},\dots, \nu_{ki}]}\ge \delta.\]
Then 
\[\E_{i,i'\in I}\snorm{f}^{\bullet}_{U_{\GP}[N;\nu_{1i}, \dots, \nu_{(k-1) i},\nu_{1i'},\dots, \nu_{(k-1)i'} ,\nu_{ki} \ast \nu_{ki'}]}\ge \delta^{O_k(1)}. \]
\end{lemma}
\begin{proof} 
Written out, the hypothesis is 
\[ \E_{i \in I}\sum_{x' \in \Z} \sum_{h'_1,\dots, h'_k \in \Z} \sum_{h''_1,\dots, h''_k \in \Z}\Big(\prod_{j = 1}^k \nu_{ji}(h'_j) \nu_{ji}(h''_j)  \Big)\Delta_{(h'_1, h''_1)} \cdots \Delta_{(h'_k, h''_k)} f(x') \ge \delta N.\]
Substituting $h_j := h''_j - h'_j$ and $x := x' + \sum_{i = 1}^k h'_i$, this is equivalent to
\begin{equation}\label{b3-first}\E_{i\in I}\sum_{x\in \Z}\sum_{h_1,\ldots,h_k\in \Z}\bigg(\prod_{j=1}^{k}\nu^{(2)}_{ji}(h_j)\bigg)\Delta_{h_1}\Delta_{h_2}\cdots \Delta_{h_k}f(x)\ge \delta N,\end{equation} where here $\nu^{(2)}_{ji} = \nu_{ji} \ast \nu_{ji}$.
In what follows, for $t \ge 1$ and $\ell = (\ell_1,\dots, \ell_t) \in \Z^t$ we denote 
\[(\Delta_{\ell_1}\cdots\Delta_{\ell_t})^{\ast}f(x) = \prod_{\omega\in \{0,1\}^t\setminus \{0\}}\mc{C}^{|\omega|}f(x+  \omega \cdot \ell),\] where as usual $\mc{C}$ denotes complex conjugation. With this notation, \cref{b3-first} becomes
\[\sum_{x\in \Z}f(x)\cdot \E_{i\in I}\sum_{h_1,\ldots,h_k\in \Z}\bigg(\prod_{j=1}^{k}\nu^{(2)}_{ji}(h_j)\bigg)(\Delta_{h_1}\Delta_{h_2}\cdots \Delta_{h_k})^{\ast}f(x)\ge \delta N.\]
Via Cauchy--Schwarz, we have that
\begin{align*}
\sum_{x\in \Z}&~\E_{i,i'\in I}\sum_{\substack{h_1,\ldots,h_k\in \Z\\h_1',\ldots,h_k'\in \Z}}\bigg(\prod_{j=1}^{k}\nu^{(2)}_{ji}(h_j)\nu^{(2)}_{ji'}(h_j')\bigg)(\Delta_{h_1}\Delta_{h_2}\cdots \Delta_{h_k})^{\ast}f(x)(\Delta_{h_1'}\Delta_{h_2'}\cdots \Delta_{h_k'})^{\ast}\overline{f}(x)\gg \delta^2 N.
\end{align*}
We may write this as
\begin{align*}
\sum_{x\in \Z}&~\E_{i,i'\in I}\sum_{\substack{h_1,\ldots,h_k\in \Z\\h_1',\ldots,h_k'\in \Z}}\bigg(\prod_{j=1}^{k}\nu^{(2)}_{ji}(h_j)\nu^{(2)}_{ji'}(h_j')\bigg) \times \\
& f(x+h_k)\overline{f(x+h_k')}(\Delta_{h_1}\Delta_{h_2}\cdots \Delta_{h_{k-1}})^{\ast}(\Delta_{h_k}f)(x)(\Delta_{h_1'}\Delta_{h_2'}\cdots \Delta_{h_{k-1}'})^{\ast}(\Delta_{h_k'}\ol{f})(x)\gg \delta^2 N.
\end{align*}
We now isolate the terms corresponding to $\nu^{(2)}_{ki}$ and $\nu^{(2)}_{ki'}$; in particular the above is equivalent to
\begin{align}\nonumber
& \E_{i,i'\in I}\sum_{h_k,h_k'\in \Z}\nu^{(2)}_{ki}(h_k)\nu^{(2)}_{ki'}(h_k')\sum_{x\in \Z}~\sum_{\substack{h_1,\ldots,h_{k-1}\in \Z\\h_1',\ldots,h_{k-1}'\in \Z}}\bigg(\prod_{j=1}^{k-1}\nu^{(2)}_{ji}(h_j)\nu^{(2)}_{ji'}(h_j')\bigg) \times \\
&\; \; f(x+h_k)\ol{f(x+h_k')}(\Delta_{h_1}\Delta_{h_2}\cdots \Delta_{h_{k-1}})^{\ast}\Delta_{h_k}f(x)(\Delta_{h_1'}\Delta_{h_2'}\cdots \Delta_{h_{k-1}'})^{\ast}\Delta_{h_k'}\ol{f}(x)\gg \delta^2 N.
\label{b332}\end{align}

For fixed $i,i', h_k, h'_k$, we now interpret the inner sum over $x, h_1,\dots, h_{k-1}, h'_1,\dots, h'_{k-1}$ as a Gowers--Peluse inner product of dimension $2k - 2 = (k-1) + (k-1)$ with respect to the measures $\nu_{1i}, \dots, \nu_{(k-1)i}, \nu_{1i'},\dots, \nu_{(k-1)i'}$. This is the key idea of the proof. As preparation, consider an arbitrary such inner product
\begin{equation}\label{gow-p} \langle (F_{\omega, \omega'})_{\omega, \omega' \in \{0,1\}^{k-1}} \rangle_{U_{\GP}[N; \nu_{1i}, \dots, \nu_{(k-1)i}, \nu_{1i'},\dots, \nu_{(k-1)i'}]}\end{equation} for some functions $F_{\omega, \omega'}$ (here of course we identify $\{0,1\}^{2k - 2}$ with the product of two copies of $\{0,1\}^{k-1}$). Written out in full, this is
\[ \frac{1}{N} \sum_{\substack{x' \in \Z\\ a_j, b_j \in \Z \\ a'_j, b'_j \in \Z}} \prod_{j = 1}^{k-1} \nu_{ji}(a_j) \nu_{ji}(b_j) \nu_{ji'}(a'_j) \nu_{ji'} (b'_j) \!\! \prod_{\substack{\omega \in \{0,1\}^{k-1} \\ \omega' \in \{0,1\}^{k-1}}} \!\! \mc{C}^{|\omega| + |\omega'|} F_{\omega, \omega'} (x' + (\mathbf{1} - \omega) \cdot a + \omega \cdot b + (\mathbf{1} - \omega') \cdot a' + \omega' \cdot b').\]
Substitute $h_j := b_j - a_j$, $h'_j := b'_j - a'_j$, $x := x' + \mathbf{1}\cdot a + \mathbf{1} \cdot a'$, and we see that this equals
\[ \frac{1}{N} \sum_{x \in \Z} \sum_{\substack{h_1,\ldots,h_{k-1}\in \Z\\h_1',\ldots,h_{k-1}'\in \Z}} \Big(\prod_{j = 1}^{k-1} \nu^{(2)}_{ji}(h_j) \nu^{(2)}_{ji'}(h'_j)\Big) \prod_{\omega, \omega' \in \{0,1\}^{k-1}} \mc{C}^{|\omega| + |\omega'|} F_{\omega, \omega'} (x + \omega \cdot h + \omega' \cdot h') .\] (The computation here is very similar to the one leading to \cref{b3-first}.) One may now see that the inner sum in \cref{b332} can be written as $N$ times the Gowers--Peluse inner product \cref{gow-p}, where $F_{\mathbf{0}, \mathbf{0}} = \Delta_{(h_k, h'_k)}f $, $F_{\omega,\omega'}(x) = \Delta_{h_k}f$ when only $\omega'$ is zero, $F_{\omega,\omega'}(x) = \Delta_{h_k'}\ol{f}(x)$ when only $\omega$ is zero and all the other $F_{\omega, \omega'}$ are $1_{|x| \le 10k N}$. Note that, due to the support properties of $f$ and the measures $\nu$, the insertion of these cutoffs is harmless and makes no difference to the expressions; note also that $h_k, h'_k$ are being considered as fixed.

From this rewriting of the inner sum in \cref{b332} and the inequality \cref{gpcs}, it follows that 
\begin{align*}
\E_{i,i'\in I}&\sum_{h_k,h_k'\in \Z}\nu^{(2)}_{ki}(h_k)\nu^{(2)}_{ki'}(h_k')\snorm{\Delta_{(h_k,h_k')}f}^{\bullet}_{U_{\GP}[N;\nu_{1i},\ldots,\nu_{(k-1)i},\nu_{1i'},\ldots,\nu_{(k-1)i'}]}\gg \delta^{O_k(1)}.
\end{align*}
Via shift invariance of the Gowers--Peluse norm, we have 
\begin{align*}
\E_{i,i'\in I}&\sum_{t\in \Z}(\nu^{(2)}_{ki}\ast \nu^{(2)}_{ki'})(t) \snorm{\Delta_{t}f}^{\bullet}_{U_{\GP}[N;\nu_{1i},\ldots,\nu_{(k-1)i},\nu_{1i'},\ldots,\nu_{(k-1)i'}]}\gg \delta^{O_k(1)}.
\end{align*}
This is equivalent to 
\begin{align*}
\E_{i,i'\in I}&\sum_{t_1,t_2\in \Z}(\nu_{ki}\ast \nu_{ki'})(t_1) (\nu_{ki}\ast \nu_{ki'})(t_2) \snorm{\Delta_{(t_1,t_2)}f}^{\bullet}_{U_{\GP}[N;\nu_{1i},\ldots,\nu_{(k-1)i},\nu_{1i'},\ldots,\nu_{(k-1)i'}]}\gg \delta^{O_k(1)}.
\end{align*}
Finally, this is equivalent to 
\[\E_{i,i'\in I}\snorm{f}^{\bullet}_{U_{\GP}[N;\nu_{1i},\ldots,\nu_{(k-1)i},\nu_{1i'},\ldots,\nu_{(k-1)i'} \nu_{ki} \ast \nu_{ki'}]}\ge \delta^{O_k(1)}, \]
as desired. 
\end{proof}

We return now to the proof of \cref{main-concat-graph}, which is an equivalent rephrasing of the case $m = 2$ of the main result of interest, \cref{lem:concat-main}. The task now is to explain how repeated use of \cref{key-dup-step} leads to the proof of that result. In this we follow the arguments of \cite[pages 10--11]{kuca} or the proof of \cite[Proposition 6.2]{KKL24}. Here is the statement of the steps we will take.

\begin{definition}\label{duplication-def}
Suppose we have a graph system $\Gamma$ of level $t$ (see \cref{graph-system-def}). Let $k \in [s]$ and let $u \in [t]$, and suppose that $u \in V_k$. Then the \emph{duplication} of $\Gamma$ along $(k,u)$ is the graph system $\tilde\Gamma$ at level $t+1$ defined as follows:
\begin{itemize}
\item if $j \in [s] \setminus \{k\}$ then include all of $V_j \setminus \{u\}$ in $\tilde V_j$;
\item if $j \in [s] \setminus \{k\}$ and if $u \in V_j$, include both $u$ and $t+1$ in $\tilde V_j$;
\item Set $\tilde V_k = V_k \setminus \{u\}$;
\item If $j \in [s]$ then include all edges $(v,w) \in E_j$ with $w \neq u$ in $\tilde E_j$;
\item If $j \in [s]$ and if $(v,u) \in E_j$, include both $(v,u)$ and $(v,t+1)$ in $\tilde E_j$;
\item Include $(u,t+1)$ in $\tilde E_k$.
\end{itemize}
An \emph{enlargement} of $\Gamma$ is simply any graph system of level $t$ obtained by adding further edges (on the same underlying vertex set $[t]$).
\end{definition}

\begin{lemma}
Let $\delta \in (0, \frac{1}{2})$, and let $s, I, \Omega_i$ be as above. Let $t \ge 1$ be an integer, and suppose that $\Gamma$ is a graph system at level $t$. Let $f : \Z \rightarrow \C$ be a 1-bounded function such that 
\begin{equation}\label{iter-step-t} \E_{i_1,\dots, i_t \in I} \Vert f \Vert^{\bullet}_{U_{\GP}[N; \Omega^{\Gamma}_{i_1,\dots, i_t}]} \ge \delta. \end{equation}
Then, if $\tilde \Gamma$ is obtained from $\Gamma$ by duplication and enlargement we have 
\begin{equation}\label{iter-step-t1} \E_{i_1,\dots, i_{t+1} \in I} \Vert f \Vert^{\bullet}_{U_{\GP}[N; \Omega^{\tilde\Gamma}_{i_1,\dots, i_{t+1}}]} \ge \delta^{O(1)}. \end{equation}
\end{lemma}
\begin{proof}
The notion of duplication (\cref{duplication-def}) is symmetric with respect to permuting $[s]$ and $[t]$, so we may assume that the duplication is with respect to $(s,t)$. In particular, $t \in V_s$.

We start with \cref{iter-step-t} and write out the Gowers--Peluse norm in terms of derivatives. Thus, for a fixed index tuple $(i_1,\dots, i_t)$, we have an average 
\[ \frac{1}{N} \sum_{x \in \Z} \E_{h_{(v,w)}, h'_{(v,w)} \sim \mu_{ji_v} \ast \mu_{ji_w}} \E_{h_v, h'_v \sim \mu_{j i_v}}  \big\{ \Delta_{h_{(v,w)}, h'_{(v,w)}} \big\}_{(v,w) \in E_j, j \in [s]}  \big\{ \Delta_{(h_v, h'_v)}\big\}_{v \in V_j , j \in [s]} f(x),\] where here the notation means that we take \emph{all} the indicated derivatives (the order being immaterial), and the $h, h'$-averages are over all these derivatives. (For instance, $\E_{h_i \sim \mu_i} \{ \Delta_{h_i} \}_{i \in [3]}$ means the same as $\E_{h_1 \sim \mu_1, h_2 \sim \mu_2, h_3 \sim \mu_3} \Delta_{h_1}\Delta_{h_2}\Delta_{h_3}$.) Now take the average over indices $(i_1,\dots, i_t)$, and rearrange this so that only terms involving $i_t$ are on the inside. To notate this, it is convenient to write $E'_j$ for the subset of $E_j$ consisting of edges $(v,t)$, and $V'_j$ for the subset of $V_j$ which is either $\{t\}$ if $t \in V_j$, or $\emptyset$ otherwise. Thus we get
\[ \bigavg_{\substack{i_1,\dots, i_{t-1} \in I \\ h_{(v,w)}, h'_{(v,w)} \sim \mu_{ji_v} \ast \mu_{ji_w} \\ h_v, h'_v \sim \mu_{j i_v}}}  \bigavg_{i_t \in I} \Big\Vert  \Big\{ \Delta_{h_{(v,w)}, h'_{(v,w)}} \Big\}_{\substack{(v,w) \in E_j\setminus E'_j \\ j \in [s]}} \Big\{ \Delta_{h_v, h'_v} f \Big\}_{\substack{v \in V_j \setminus V'_j \\ j \in [s]}}\Big\Vert^{\bullet}_{U_{\GP}[N; \tilde \Omega_{i_1,\dots, i_t}]} \gg \delta^{O(1)}, \]
where here $\tilde\Omega_{i_1,\dots, i_t}$ are the measures which do involve $i_t$, namely the $\mu_{ji_v} \ast \mu_{ji_t}$ for $(v,t) \in E_j$ for all $j \in [s]$, and the $\mu_{ji_t}$ for $j$ such that $t \in V_j$. In particular, since $t \in V_s$, we see that the measure $\mu_{si_t}$ itself lies in $\tilde\Omega_{i_1,\dots, i_t}$.

We now apply \cref{key-dup-step} with the $\nu$'s being the elements of $\tilde\Omega_{i_1,\dots, i_t}$ and with $\nu_{ki}$ in that lemma (that is, the measure to be `duplicated') being  $\mu_{si_t}$, and with $N$ replaced by $2N$ so that the support condition holds. Writing $i_{t+1}$ for the resulting $i'$ variables, the conclusion is then that 

\[ \bigavg_{\substack{i_1,\dots, i_{t-1} \in I \\ h_{(v,w)}, h'_{(v,w)} \sim \mu_{ji_v} \ast \mu_{ji_w} \\ h_v, h'_v \sim \mu_{j i_v}}}  \!\!\!\!  \bigavg_{i_t, i_{t+1} \in I}\Big\Vert 
\Big\{ \Delta_{h_{(v,w)}, h'_{(v,w)}} \Big\}_{\substack{(v,w) \in E_j\setminus E'_j \\ j \in [s]}} \Big\{ \Delta_{h_v, h'_v} f \Big\}_{\substack{v \in V_j \setminus V'_j \\ j \in [s]}}
\Big\Vert^{\bullet}_{U_{\GP}[2N; \tilde \Omega^*_{i_1,\dots, i_t, i_{t+1}}]} \gg \delta^{O(1)}, \] where now $\Omega^*_{i_1,\dots, i_t, i_{t+1}}$ consists of:
\begin{itemize}
\item $\mu_{ji_v} \ast \mu_{ji_t}$ and $\mu_{ji_v} \ast \mu_{ji_{t+1}}$ for $(v,t) \in E_j$ for all $j \in [s]$;
\item $\mu_{ji_t}$ and $\mu_{ji_{t+1}}$ for $j \in [s-1]$ such that $t \in V_j$;
\item $\mu_{si_t} \ast \mu_{si_{t+1}}$.   
\end{itemize}
Folding the outer derivatives back into the Gowers--Peluse norm notation (and replacing $2N$ by $N$, which loses only a factor of 2) one sees that 
\[ \E_{i_1,\dots, i_{t+1} \in I} \Vert  f  \Vert^{\bullet}_{U_{\GP}[N;\Omega^{\tilde\Gamma}]} \gg  \delta^{O(1)},\]
where $\tilde\Gamma$ (of level $t+1$) consists of the following edges and vertices:
\begin{itemize}

\item $\tilde V_j$ contains $V_j \setminus \{t\}$ for $j \in [s-1]$;
\item $\tilde V_j$ contains $t$ and $t+1$ if $t \in V_j$ for $j \in [s-1]$;
\item $\tilde V_s = V_s \setminus \{t\}$;
\item $\tilde E_j$ contains $(v,w) \in E_j$ for $w < t$, for all $j \in [s]$;
\item $\tilde E_j$ contains $(v,t), (v, t+1)$ if $(v,t) \in E_j$, for $j \in [s]$;
\item $\tilde E_s$ contains $(t, t+1)$.
\end{itemize}
One may check that $\tilde\Gamma$ is precisely the duplication of $\Gamma^t$ along $(s,t)$; indeed the six bullets above are in 1-1 correspondence with those in \cref{duplication-def}, if one takes $(k,u) = (s,t)$ there.

Any enlargement of $\tilde\Gamma$ cannot decrease the Gowers--Peluse norm by more than harmless losses of $O(1)$ factors, by \cref{lem:prop-box}.
\end{proof}

We now show how, by a sequence of duplications and enlargements, we may move from an edge-empty graph to an edge-complete one. This essentially combines the proof of \cite[Proposition~6.2,~Corollary~6.3]{KKL24} into a double induction.

\begin{lemma}
 There is a sequence of graph systems $\Gamma^t$, $t = 1,2,3\dots, 2^s$ such that $\Gamma^1$ is the `vertex-complete, edge-empty' system in \cref{assump-again-again}, and $\Gamma^{2^s}$ is the `edge-complete, vertex-empty' system in \cref{completed}, and moreover $\Gamma^{t+1}$ is obtained from $\Gamma^t$ by a duplication and an enlargement.   
\end{lemma}
\begin{proof} For $r = 0,1,2,\dots$ and for $0 \le \ell < 2^r$ write $t = t(r,\ell) := 2^r + \ell$, and define the graph system $\Gamma^{t}$ in terms of its edges and vertices, as follows: 
\[ E_{s - r + 1}^t = \cdots = E_s^t = \binom{[t]}{2},\quad E_{s-r}^t = \binom{[2^r - \ell + 1, t]}{2}, \quad  E_1^t = \cdots = E_{s-r - 1}^t =  \emptyset,\] and
\[ V_1^t = \cdots = V_{s - r-1}^t = [t], \quad V_{s - r}^t = [2^r - \ell], \quad V_{s-r+1}^t = \dots = V_s^t = \emptyset.\] 
Note in particular that when $t = 1$ (and so $r = 0$ and $\ell = 0$) these graphs are exactly the graphs $\Gamma^1$ in \cref{assump-again-again}.
Duplicate $\Gamma^t$ along $(s - r, 2^r - \ell)$. The resulting graph then has
\[E_{s - r + 1}^{t+1} = \cdots = E_s^{t+1} = \binom{[t+1]}{2}, \quad E_{s-r}^{t+1} = \binom{[2^r - \ell + 1, t]}{2} \cup \{ (2^r - \ell, t+1)\}, \] and $E_1^t = \cdots = E_{s-r - 1}^t =  \emptyset$, and
\[ V_{1}^{t+1} = \cdots = V_{s-r - 1}^{t+1} = [t+1],\quad V_{s-r}^{t+1} = [2^r - \ell -1] , \quad  \quad V_{s-r+1}^t = \dots = V_s^t = \emptyset.\]  Enlarging $E_{s-r}^{t+1}$ to $\binom{[2^r - \ell , t+1]}{2}$ then gives $\Gamma^{t+1}$. (Note that this is true for $0 \le \ell \le 2^r - 2$ and also when $\ell = 2^r - 1$, when $t + 1= t(r+1,0)$.)

\end{proof}

\section{A large sieve bound in several dimensions}\label{appendixC}

In this appendix we give a brief discussion of the higher-dimensional large sieve, for which it is hard to find a suitable reference.

\begin{proposition}\label{multi-ls}
Let $\theta_j$ be $\delta$-well spaced points in the $\ell^{\infty}$ distance on $(\R/\Z)^k$,  let $(a_n)_{n \in [N]^k}$ be a sequence of complex numbers, and write 
\begin{equation}\label{s-theta-def} S(\theta) := \sum_{n \in [N]^k} a_n e (\theta \cdot n).\end{equation}
Then
\[ \sum_j |S(\theta_j)|^2 \le  (N^{1/2} + \delta^{-1/2})^{2k} \Vert a \Vert^2,\]
where $\Vert a \Vert := \big( \sum_n a_n \big)^{1/2}$.
\end{proposition}
\begin{proof}
See \cite[Theorem 1]{huxley}. We note that, for our purposes (where we take $k = 2$), it would suffice to have $\ll_k$ instead of $\le$ in this inequality. If one's aim is a result of this type, the need for special trigonometric constructions as used in \cite{huxley} (who quotes a construction of Bombieri and Davenport \cite{bombieri-davenport}) can be avoided and replaced with cruder bounds. See the remarks in \cite[p177--178]{IK-book}.  Note also that the proof of \cite[Theorem 7.7]{IK-book} (who use Hilbert's inequality, following Montgomery and Vaughan \cite{montgomery-vaughan-paper}) does not adapt so straightforwardly to the high-dimensional case. For further discussion see \cite[Chapter 4]{kowalski-ls}.
\end{proof}

\begin{corollary}\label{ls-cor}
With notation as above, we have
\[ \sum_{q_1,\dots, q_k \in [N^{1/2}] } \sideset{}{'}\sum_{a_i \mdsublem{q_i}} \big|S\big((\frac{a_i}{q_i})_{i = 1}^k\big)\big|^2 \le (2N)^k \Vert a \Vert^2,\] where the dash means that the sum is taken over $a_i$ coprime to $q_i$.
\end{corollary}
\begin{proof}
This follows immediately from  \cref{multi-ls} and the fact that the points $(\frac{a_i}{q_i})_{i = 1}^k$ are $N^{-1}$-well spaced in the $\ell^{\infty}$ distance on $(\R/\Z)^k$. 
\end{proof}

\begin{lemma}\label{large-sieve-higher}
For each prime $p$, suppose that we have a set $\Omega_p \subseteq(\Z/p\Z)^k$ of density $\alpha_p$. Then 
\begin{equation}\label{large-sieve-est} \# \{ |x|, |y| \le N  : \prod_{p} 1_{\Omega_p^c}(x,y)= 1 \} \le (2N)^k \big( \sum_{\substack{q \le N^{1/2} \\ \mu^2(q) = 1}} h(q) \big)^{-1},\end{equation}
where $h$ is the multiplicative function satisfying $h(p) = \alpha_p/(1 - \alpha_p)$.
\end{lemma}
\begin{proof}
This is a routine generalisation of \cite[Theorem 7.14]{IK-book} to several variables, using \cref{ls-cor} in place of \cite[Theorem 7.11]{IK-book}; note that \cite[Lemma 7.15]{IK-book}, which is the main argument, goes over essentially verbatim to the multidimensional setting.
    \end{proof}

\section{Number-theoretic bounds}\label{appendixD}

In this appendix we collect some more-or-less standard bounds of a number-theoretic flavour.

\begin{lemma}\label{lem:divisor-moment}
For $Y \ge 1$ and for integer $m \ge 1$, we have $\sum_{|y| \le Y} \tau(y)^{m} \ll_{m} Y (\log Y)^{2^m-1}$.
\end{lemma}
\begin{proof}
Here we take $\tau(0) := 0$. This result is standard; see e.g. \cite[Equation~(1.80)]{IK-book}. 
\end{proof}

\begin{lemma}\label{repdivisor}
Let $n \ge 1$ be fixed and let $r(t)$ be the number of representations of $t$ by $x^2 + ny^2$ with $x, y \in \Z$. Then $r(t) \le 6\tau(t)$.
\end{lemma}
\begin{proof} We work with ideals in $K = \Q(\sqrt{-n})$. We have $t = x^2 + ny^2$ if and only if $N((x + y\sqrt{-n})) = t$. Since a principal ideal is generated by $|\O^*_K| \le 6$ different elements of $\O_K$, we see that $r(t)$ is bounded above by $6$ times the number of ideals of norm $t$. To count the number of such ideals, factor $t = p_1^{a_1} \cdots p_k^{a_k}$ as a product of prime powers. The number of ideals of norm $p^a$ is either $a+1$ (if $p$ splits in $K$) or $1_{2 \mid a}$ (if $p$ is inert in $K$) or $1$ if $p$ ramifies. In all cases, the number of such ideals is at most $a+1$, and so the number of ideals of norm $t$ is $\le (a_1 + 1) \cdots (a_k + 1) = \tau(t)$.
\end{proof}

We need the following classical estimate on the number of ideals of norm up to a given threshold.
\begin{lemma}\label{num-ideals-lem}
For any number field $K$, the number of ideals in $\O_K$ of norm at most $X$ is $C_K X + O(X^{1 - [K:\Q]^{-1}})$.
\end{lemma}
\begin{proof}
Results of this form date back to work of Dedekind, Weber, and Hecke and take the form $C_K X + O_K(X^{1 - [K:\Q]^{-1}})$. Here $C_K=|\O^*_K|^{-1} |\Delta|^{-1/2}2^{r_1 + r_2} \pi^{r_2} R$ where $R$ is the regulator and $r_1, r_2$ the number of real and pairs of conjugate complex embeddings. For a sketch proof, see \cite[Theorem~1]{heilbronn}. 
\end{proof}

Finally, we need an analogue of Mertens' `second theorem' in number fields.
\begin{lemma}\label{number-field-mertens}
Let $K$ be any number field. Then there is some constant $M_K$ such that 
\[ \sum_{\mf{p} : N\mf{p} \le X} (N\mf{p})^{-1}  =  \log \log X  + M_K + O_K(\frac{1}{\log X}).\]
\end{lemma}
\begin{proof}
This may be derived via partial summation from \cref{main-analytic-input}; however a rather more elementary treatment is possible. We briefly repeat the proof given in \cite[Section~10.1]{DH08}, which is a minor elaboration of the well-known proof over the integers.

Let $[X]! = \prod_{N\mf{n}\le X}\mf{n}$ and $R(n) = \sum_{N\mf{n} \le n}1$. By partial summation and \cref{num-ideals-lem}, we have that 
\begin{align}\nonumber \log N([X]!) =  \sum_{n\le X}(R(n) - R(n-1)) \log n & =  R(X) \log X - \sum_{n<X}R(n)\log(n/(n+1)) \\ & = C_K X\log X + O_K(X).\label{factorial-1}\end{align}
Furthermore, by factorization into prime ideals and a further application of \cref{num-ideals-lem}, we have that 
\begin{align*} \log N([X]!) & = \sum_{\mf{p} : N\mf{p} \le X} R\big( \frac{X}{N\mf{p}}\big)\log(N\mf{p}) \\ & = C_K X \sum_{\mf{p} : N\mf{p} \le X} \frac{\log(N\mf{p})}{N\mf{p}}  + O_K\Big(\sum_{\mf{p} : N\mf{p} \le X} \Big(\frac{X}{N\mf{p}}\Big)^{1-[K:\Q]^{-1}}\log(N\mf{p})\Big) .\end{align*}
To bound the error term, we sum over rational primes $p$ with $\mf{p}$ as a factor; since such a prime has $p \le N\mf{p}$, and there are at most $[K : \Q]$ prime ideals associated to each rational prime, we have
\[\sum_{\mf{p} : N\mf{p} \le X} \Big(\frac{X}{N\mf{p}}\Big)^{1-[K:\Q]^{-1}}\log(N\mf{p}) \ll X^{1-[K:\Q]^{-1}} \sum_{p\le X} \frac{\log p}{p^{1-[K:\Q]}} \ll X, \]
where in the final line we have used Chebyshev's estimate that $\sum_{n\le p\le 2n}\log p\le \log\binom{2n}{n}\ll n$ and dyadic summation. Comparing with \cref{factorial-1}, it follows that
\[\sum_{\mf{p} : N\mf{p} \le X} \frac{\log(N\mf{p})}{N\mf{p}} = \log X + O(1),\] which is (a weak form of) Mertens' `first theorem' over number fields.
The desired result then follow via partial summation, using the identity
\[ \sum_{N \mf{p} \le X} \frac{1}{N\mf{p}} = \frac{1}{\log X} \sum_{N\mf{p} \le X} \frac{\log (N\mf{p})}{N\mf{p}} + \int^X_2 \frac{dy}{y (\log y)^2} \sum_{N\mf{p} \le y} \frac{\log (N\mf{p})}{N\mf{p}}\] and the fact that $\frac{1}{y \log y}$ is the derivative of $\log \log y$. 
\end{proof}

\bibliographystyle{amsplain0}
\bibliography{main.bib}

\end{document}